\documentclass[11pt]{article}
\usepackage[english]{babel}
\usepackage[left=1in,right=1in,top=1in,bottom=1.5in]{geometry} 
\usepackage[skip=5pt, indent=10pt]{parskip}
\usepackage{amssymb, amsmath, amsthm}
\usepackage{hyperref}
\usepackage{xcolor}
\usepackage{nicefrac}
\usepackage[margin=1cm]{caption}
\usepackage{natbib}
\usepackage{bbm}
\usepackage[scr]{rsfso}
\usepackage{algorithm}
\usepackage{algpseudocode}
\usepackage{times}
\usepackage{graphicx}
\usepackage{stackrel}
\usepackage{cases}
\usepackage[shortlabels]{enumitem}

\definecolor{darkred}{HTML}{880000}
\definecolor{darkblue}{HTML}{000088}

\hypersetup{
  colorlinks,
  linkcolor={darkred},
  citecolor={darkblue}
}

\renewcommand{\leq}{\leqslant}

\renewcommand{\geq}{\geqslant}

\renewcommand{\tilde}{\widetilde}
\renewcommand{\hat}{\widehat}

\newcommand{\eps}{\varepsilon}

\newcommand{\dd}{{\mathrm{d}}}
\newcommand{\Var}{{\rm Var}}
\newcommand{\rdiv}{{\rm div}}

\newcommand{\sfk}{\mathsf{k}}
\newcommand{\sfp}{\mathsf{p}}
\newcommand{\sfP}{\mathsf{P}}

\newcommand{\NN}{\mathsf{NN}}

\newcommand{\E}{\mathbb E}
\newcommand{\N}{\mathbb N}
\newcommand{\p}{\mathbb P}
\newcommand{\R}{\mathbb R}
\newcommand{\Z}{\mathbb Z}
\newcommand{\1}{\mathbbm 1}

\newcommand{\sL}{\mathscr{L}}

\newcommand{\bj}{{\mathbf j}}
\newcommand{\bk}{{\mathbf k}}
\newcommand{\be}{{\mathbf e}}
\newcommand{\bp}{{\mathbf p}}

\newcommand{\cA}{{\mathcal{A}}}

\newcommand{\cD}{{\mathcal{D}}}
\newcommand{\cE}{{\mathcal{E}}}
\newcommand{\cF}{{\mathcal{F}}}
\newcommand{\cG}{{\mathcal{G}}}
\newcommand{\cH}{{\mathcal{H}}}

\newcommand{\cL}{{\mathcal{L}}}

\newcommand{\cN}{{\mathcal{N}}}
\newcommand{\cO}{{\mathcal{O}}}

\newcommand{\cR}{{\mathcal{R}}}
\newcommand{\cS}{{\mathcal{S}}}
\newcommand{\cT}{{\mathcal{T}}}

\newcommand{\cW}{{\mathcal{W}}}

\newcommand{\floor}[1]{{\lfloor #1 \rfloor}}

\def\argmin{\operatornamewithlimits{argmin}}

\newcommand{\gelu}{\mathrm{GELU}}

\newcommand{\integral}[1]{\int\limits_{#1}}
\newcommand{\tr}{\mathop{\mathrm{Tr}}\limits}

\newtheorem{Th}{Theorem}[section]
\newtheorem{Lem}[Th]{Lemma}
\newtheorem{Def}[Th]{Definition}

\newtheorem{Rem}[Th]{Remark}

\newtheorem{As}[Th]{Assumption}

\newcommand{\myendproof}{\hfill$\square$}

\title{Implicit score matching meets denoising score matching:\\
improved rates of convergence and log-density Hessian estimation}

\author{
Konstantin Yakovlev\thanks{HSE University, Russian Federation, kdyakovlev@hse.ru}
\and
Anna Markovich\thanks{HSE University, Russian Federation, aamarkovich@hse.ru}
\and
Nikita Puchkin\thanks{HSE University, Russian Federation, npuchkin@hse.ru}
}

\date{}

\begin{document}

\maketitle

\begin{abstract}
    We study the problem of estimating the score function using both implicit score matching and denoising score matching.
    Assuming that the data distribution exhibiting a low-dimensional structure, we prove that implicit score matching is able not only to adapt to the intrinsic dimension, but also to achieve the same rates of convergence as denoising score matching in terms of the sample size.
    Furthermore, we demonstrate that both methods allow us to estimate log-density Hessians without the curse of dimensionality by simple differentiation. This justifies convergence of ODE-based samplers for generative diffusion models.
    Our approach is based on Gagliardo-Nirenberg-type inequalities relating weighted $L^2$-norms of smooth functions and their derivatives.
\end{abstract}

\bigskip

\section{Introduction}

Given i.i.d. samples $Y_1, \dots, Y_n \in \R^D$ drawn from an absolutely continuous distribution with a density $\sfp^*$, we are interested in estimation of the log-density gradient $s^*(y) = \nabla \sfp(y)$ also referred to as score function.
This task arises naturally in generative modelling, where a learner has to deal with intractable densities and must use either classical algorithms (such as Langevin \citep{cheng2017underdamped} or Hamiltonian Monte Carlo \citep{neal2011mcmc}) or modern score-based diffusion models \citep{song2019generative, song20a, song2021scorebased} to produce a new sample.
Besides, the problem of interest can be considered as a particular case of density estimation where the performance is measured by the Fisher divergence.
The advantages of the Fisher divergence were discussed, for instance, by \cite{sriperumbudur17}.
For these reasons, statisticians have paid a lot of attention to theoretical analysis of various score estimation methods including kernel-based approaches \citep{li05, wibisono2024optimal, zhang2024minimax}, Stein's method \citep{li2018gradient, shi2018spectral}, implicit score matching \citep{hyvarinen2005estimation,hyvarinen2007extensions,sriperumbudur17, sutherland2018efficient, sasaki18mode, zhou2020nonparametric, koehler2023statistical}, and denoising score matching \citep{oko2023diffusion, tang2024adaptivity, azangulov2024convergence, yakovlev2025generalization}.

In the present paper, we examine denoising \citep{vincent2011connection} and implicit \citep{hyvarinen2005estimation} score matching. The former algorithm is the cornerstone for learning score-based diffusion models \citep{song2021scorebased}. Many advances on score estimation with deep feedforward neural networks appeared in few recent years. In \cite{oko2023diffusion}, the authors considered a nonparametric setup, where the target density belongs to the Besov space $B^\beta_{p, q}([-1, 1]^D)$. Assuming that $\sfp^*$ is bounded away from zero, they showed\footnote{According to \cite{yakovlev2025generalization}, the analysis of estimation error in \cite{oko2023diffusion} has a flaw (in particular, the issue occurs in the proof of Theorem C.4, see \citep[page 9]{yakovlev2025generalization} for the discussion). However, following the approach of \cite{yakovlev2025generalization}, one can obtain the rates of convergence (with respect to the sample size $n$) announced by \cite{oko2023diffusion} with slightly worse dependence on the stopping time $\underline T$.} that denoising score matching can achieve the expected squared error of order $\cO(n^{-2\beta / (2\beta + D)})$. While this rate of convergence is common for nonparametric statistics, it obviously suffers from the curse of dimensionality when $D = \Omega(\log n)$ (as well as kernel-based approaches \cite{wibisono2024optimal, zhang2024minimax}).
Fortunately, real-world data sets, such as high-resolution images, often have intrinsic low-dimensional structures \citep{bengio2013representation,pope2021the}.
To mitigate the curse of dimensionality, \cite{chen2023score} suggested a model with a target distribution supported on a low-dimensional linear subspace and derived the score estimation rates that do not deteriorate fast when the ambient dimension is large.
However, this setup looks oversimplified and poorly reflects non-linear structures in the read-world data. This issue was addressed in subsequent works \citep{tang2024adaptivity, azangulov2024convergence}, where the authors assumed that $\sfp^*$ is supported on a smooth submanifold of dimension $d < D$. Under some technical assumptions, they proved\footnote{The papers of \cite{tang2024adaptivity, azangulov2024convergence} inherit the mistake of \cite{oko2023diffusion}, as their proofs rely on Theorem C.4, which has a flaw. Similarly to \cite{oko2023diffusion}, one can obtain the rates of convergence (with respect to the sample size $n$) announced in these papers (with slightly worse dependence on the stopping time) following the approach of \cite{yakovlev2025generalization}.} that the rate of convergence for for denoising score matching in this setup is determined by the intrinsic dimension $d$, rather than the ambient one. Despite a significant progress in understanding score-based diffusion models, both \cite{tang2024adaptivity} and \cite{azangulov2024convergence} required $\sfp^*$ to be bounded away from zero. According to \cite{zhang2024minimax}, this may be too demanding. In \cite{yakovlev2025generalization}, the authors bypassed the restrictive assumptions of density lower bounds and bounded data distribution support, thereby capturing a broader and more realistic class of data distributions.

On the other hand, the ability of implicit score matching to adapt to low intrinsic dimension (as in the setups of \cite{tang2024adaptivity, azangulov2024convergence, yakovlev2025generalization}) is much less explored.
In \cite{sriperumbudur17,sutherland2018efficient,zhou2020nonparametric}, the authors studied nonparametric score estimation using reproducing kernel Hilbert space (RKHS). In addition, \cite{sasaki18mode} use similar machinery to derive the estimation error bounds for the score function and its derivatives under analogous conditions.
Of course, it is known that diffusion kernels can keep information about underlying manifold structure, and some popular dimension reduction techniques, such as diffusion maps \cite{coifman06} or Laplacian eigenmaps \cite{belkin03}, are based on this fact. However, this property mostly holds in the situation when the data points lie exactly on the manifold or in its very small vicinity. For this reason, even if the kernel is chosen properly, the rates of convergence in \cite{sriperumbudur17,sutherland2018efficient,sasaki18mode,zhou2020nonparametric} may significantly deteriorate in noisy setups, which are common in generative diffusion models (see, for example, \citep{tang2024adaptivity, azangulov2024convergence, yakovlev2025generalization}). A similar issue appears in \cite{shen2024differentiable}, where the authors assumed that
the target density is supported on a thin tube around a low-dimensional manifold. They leave the case of unbounded support beyond the scope and, moreover, require the tube to shrink as the sample size grows. This leads us to the following question.

\begin{quote}
    \emph{
    \textbf{Q1}:
    Can a score estimator trained with an implicit score matching objective escape the curse of dimensionality in the noisy setup (for instance, as in \cite{yakovlev2025generalization})? If so, can it achieve the same rates of convergence as denoising score matching?
    }
\end{quote}

Besides score estimation, there is also an important line of research devoted to inference with generative diffusion models.
Researchers have made much effort to study iteration complexity of SDE- and ODE-based samplers (see, for instance, \citep{de_bortoli2022convergence, chen2023sampling, chen2023improved, li2024towards, benton2024nearly, huang24faster, li2024adapting}). It turns out that deterministic samplers often require accurate estimates of both the score function and its Jacobian matrix (that is, the log-density Hessian) \citep{li2024accelerating, li2024towards, li2024sharp, huang2025fast, li2025faster}. Sometimes, a faster inference demands the score estimate to be Lipschitz \citep{zhang2025sublinear}. The drawbacks of the approach of \cite{shen2024differentiable} and of the RKHS-based methods discussed in the previous paragraph are also valid for the log-density Hessian estimation. This brings us to another question.

\begin{quote}
    \emph{
    \textbf{Q2}:
    Can we estimate the score Jacobian matrix without the curse of dimensionality?
    }
\end{quote}

In what follows, we give a positive answer to this question. Moreover, we show that it is enough to take derivatives of denoising and implicit score matching estimates for this purpose.

\medskip
\noindent
\textbf{Our contribution.}\quad
In this paper, motivated by the aforementioned research questions, we investigate the properties of feedforward neural networks with smooth activation function \citep{yakovlev2025simultaneous} in estimating both the score function and its Jacobian matrix without the curse of dimensionality under the noisy setting as in \cite{yakovlev2025generalization}.
Our main contributions are summarized below.

\begin{enumerate}
    \item 
    We demonstrate that implicit score matching can surpass the curse of dimensionality and achieve the same rates of convergence as denoising score matching (in terms of the sample size).
    \item We further show that the Jacobian matrix of the true score function can be estimated without the curse of dimensionality by differentiating the outputs of either denoising or implicit score matching.
    \item Our main results rely on extensions of the Gagliardo-Nirenberg inequality for the space $L^2(\sfp^*)$. 
    The first inequality reveals favourable geometry of the implicit score matching loss landscape and results in sharper rates of convergence.
    The second inequality links $L^2(\sfp^*)$-norm of a smooth function with its Sobolev seminorm $W^{1,2}(\sfp^*)$ and yields an upper bound on the estimation error of the score Jacobian matrix.
    \item As a byproduct, we develop a tail inequality for suprema of unbounded empirical processes, building upon \cite{adamczak2008tail} and \cite{bartlett2005local}.
    Our new machinery extends the standard localization technique to the unbounded setting, thereby eliminating the need for an $\varepsilon$‑net argument and simplifying the analysis while maintaining the sharpness of the resulting convergence rates.
\end{enumerate}

\medskip
\noindent
\textbf{Paper structure.}\quad
The remainder of the paper is organized as follows.
Section \ref{sec:prelim_notation} introduces the necessary preliminaries and notation.
Our main results are presented in Section \ref{sec:results}.
Section \ref{sec:estimation_proof} contains the proof of Theorem \ref{th:estimation} on the statistical efficiency of implicit score matching.
The Appendix includes auxiliary results and deferred proofs.

\medskip
\noindent
\textbf{Notation.}\quad
We use the following notations throughout the paper.
For real numbers $x$ and $y$, we define $x \vee y = \max\{x, y\}$ and $x \wedge y = \min\{x, y\}$. The condition $f \lesssim g$ and $g \gtrsim f$ means $f = \cO(g)$.
Moreover, the notation $f \asymp g$ indicates that $f \lesssim g$ and $g \lesssim f$.
Finally, for any function class $\cF$ equipped with a metric $\rho$, we denote its diameter (with respect to $\rho$) by $\cD(\cF, \rho)$:
\[
    \cD(\cF, \rho) = \sup\limits_{f, g \in \cF} \rho(f, g).
\]

\section{Preliminaries and notations}
\label{sec:prelim_notation}

\noindent
\textbf{Multi-index notation.}\quad
We denote the set of non-negative integers by \(\Z_+\) and for a multi-index \(\bk = (k_1, k_2, \dots, k_m)\in\Z_+^m\) we write
\begin{align*}
    |\bk| = k_1 + k_2 + \ldots + k_m,
    \quad
    \bk! = k_1! \cdot k_2! \cdot \ldots \cdot k_m!
\end{align*}
For \(\bk, \bj \in \Z^m_+\) notation \(\bk  \geq \bj, \bk \leq \bj\) stands for element-wise comparison.
In addition, summation and subtraction are also defined element-wise: $\bk \pm \bj = (k_1 \pm j_1, k_2 \pm j_2, \dots, k_m \pm j_m)$.
We also define an indicator function as $\1[\bk = \bj]$, which is one if $\bj = \bk$ and zero otherwise.
Let \(f: \Omega\to\R^r\) be an arbitrary function defined on a set \(\Omega \subseteq \R^m\). For a multi-index \(\bk \in \Z_+^r\), we define the corresponding partial differential operator $\partial^\bk$ as a component-wise partial derivative
\begin{align*}
    \partial^\bk f(x) =  (\partial^{\bk} f_1(x), \dots, \partial^{\bk} f_m(x))^\top, \quad \partial^{\bk} f_i(x) = \frac{\partial^{|\bk|} f_i}{\partial x_1^{k_1} \ldots \partial x_r^{k_r}}, \quad i \in \{1, \dots, m\}.
\end{align*}
The divergence of $f : \R^r \to \R^r$, denoted as $\rdiv[f](x)$, is given by
\begin{align*}
    \rdiv[f](x) = \sum_{i = 1}^r \frac{\partial f_i(x)}{\partial x_i}.
\end{align*}
The Jacobian matrix of $f$ is represented by $\nabla f(x)$.

\noindent
\textbf{Norms.}\quad
We denote the Euclidean vector norm by \(\|\cdot\|\). 
For a vector \(v\) and a matrix \(A\) we use  \( \|v\|_\infty\) and \(\|A\|_\infty\) respectively for maximal absolute values of their entries.
Similarly,  \(\|v\|_0\) and \(\|A\|_0\) stand for the number of non-zero entries of \(v\) and \(A\).
Moreover, we denote by $\|A\|_F$ and $\|A\|$ the Frobenius and the spectral norm of a matrix $A$, respectively.
For a vector-valued \(\mu\)-measurable function \(f: \Omega \to \R^m\) we use 
\begin{align*}
    \|f\|_{L^p(\Omega)}
    = \left\{\integral{\Omega} \|f(x)\|^p \, \dd \mu(x) \right\}^{1 / p},
    \quad
    \|f\|_{L^\infty(\Omega)} = \sup_{x\in\Omega} \|f(x)\|.
\end{align*}
Similarly, for a non-negative weighting function $\sfp : \Omega \to \R$, we define the weighted $L^p$-norm and inner product as
\begin{align*}
    \|f\|_{L^p(\Omega, \sfp)} = \left\{\integral{\Omega} \|f(x)\|^p \, \sfp(x) \dd \mu(x) \right\}^{1 / p},
    \quad
    \langle f, g \rangle_{(\Omega, \sfp)} = \integral{\Omega} f(x)^\top g(x) \; \sfp(x) \dd \mu(x) .
\end{align*}
When $\Omega$ is clear from context, we will write $\|f\|_{L^p(\sfp)}$ and $\langle f, g\rangle_\sfp$ correspondingly.
Finally, for a real-valued random variable $\xi$ its $\psi_1$-norm is defined by
\begin{align}
\label{eq:psi_1_norm_def}
    \|\xi\|_{\psi_1} = \inf\{t > 0 \, : \, \E \exp\left\{ |\xi| / t \right\} \leq 2 \} .
\end{align}

\medskip
\noindent
\textbf{Neural networks.}\quad
In this paper we use a Gaussian error linear unit activation function, given by
\begin{align}
    \label{eq:gelu_def}
    \gelu(x) = x\cdot\Phi(x), \quad \Phi(x) = \frac{1}{\sqrt{2\pi}}\integral{-\infty}^x e^{-t^2 / 2} \, \dd t,
    \quad x \in \R .
\end{align}
The shifted activation function, $\gelu_b(x)$ for $b = (b_1, \dots, b_r)$, maps a vector $x$ from $\R^r$ to $\R^r$ as follows:
\begin{align*}
    \gelu_b(x) = \big( \gelu(x_1 - b_1), \dots, \gelu(x_r - b_r) \big), \quad x = (x_1, \dots, x_r) \in \R^r .
\end{align*}
For $L \in \N$ and an architecture vector $W = (W_0, W_1, \dots, W_L) \in \N^{L + 1}$, a feedforward neural network of depth $L$ and architecture $W$ is a function $f : \R^{W_0} \to \R^{W_L}$ defined by the composition:
\begin{align}
    \label{eq:feed_forward_nn_def}
    f(x) = -b_L + A_L \circ \gelu_{b_{L - 1}} \circ A_{L - 1} \circ \gelu_{b_{L - 2}} \circ \ldots \circ A_2 \circ \gelu_{b_1} \circ A_1 \circ x ,
\end{align}
where $A_j \in \R^{W_{j} \times W_{j - 1}}$ is a weight matrix and $b_j \in \R^{W_j}$ is a bias vector for each $j \in \{1, \dots, L\}$.
The maximum number of neurons in each layer is defined as $\|W\|_\infty$, representing the width of the network.
We define the class of neural networks of the form \eqref{eq:feed_forward_nn_def} with at most $S \in \N$ non-zero weights and the weight magnitude bounded by $B > 0$ as follows:
\begin{align}
    \label{eq:nn_class_def}
    \NN(L, W, S, B) = \Bigg\{
    &\notag
    \text{$f$ of the form \eqref{eq:feed_forward_nn_def}}:
    \sum_{j = 1}^L \big( \|A_j\|_0 + \|b_j\|_0 \big) \leq S,
    \\&
    \qquad\qquad\quad
    \text{and } \max_{1 \leq j \leq L} \left\{ \|A_j\|_\infty \vee \|b_j\|_\infty \right\} \leq B \Bigg\}.
\end{align}

\medskip
\noindent
\textbf{Smoothness classes.}\quad
For any \(r, m, s \in \N\) and \(\Omega \subseteq \R^r\) let  $C^s(\Omega)$ be the space of functions $f : \Omega \to \R^m$ with bounded and continuous derivatives up to order $s$. Formally, we define 
\begin{align*}
    C^s(\Omega) = \left\{f : \Omega \to \R^m : \max_{\bk \in \Z_+^r, \; |\bk| \leq s} \|\partial^\bk f\|_{L^\infty(\Omega)} < \infty \right\}.
\end{align*}
For any \(\Omega \subseteq \R^r, r \in \N\) and any positive \(0 \leq \delta \leq 1\) a function \(f: \Omega \to \R\) is called \(\delta\)-Hölder continuous, if
\begin{align*}
    [f]_\delta = \sup_{\substack{x, y \in\Omega\\ x \neq y}} \frac{|f(x) - f(y)|}{1 \wedge \|x-y\|_\infty^\delta} < \infty.
\end{align*}

\begin{Def}[Hölder class]\label{def:holder_class}
    For given $\beta > 0$ we let $\floor{\beta}$ to be the largest integer strictly less than $\beta$. The Hölder class \(\cH^\beta(\Omega, \R^m)\) contains functions from $C^\floor{\beta}(\Omega)$, such that their derivatives of order $\floor{\beta}$ are $(\beta - \floor{\beta})$-Hölder-continuous. For a \(H > 0\) a Hölder ball \(\cH^\beta(\Omega, R^m, H)\) is given by
    \begin{align*}
        \cH^\beta(\Omega, \R^m, H) = \left\{f \in C^\floor{\beta}(\Omega) : \max_{1\leq i\leq m}\max_{\substack{\bk \in \Z_+^r \\ |\bk| \leq \floor{\beta}}} \|\partial^\bk f_i\|_{L^\infty(\Omega)} \leq H,
        \; \max_{1\leq i\leq m}\max_{\substack{\bk \in \Z_+^r \\ |\bk| = \floor{\beta}}} [\partial^\bk f_i]_{\beta - \floor{\beta}} \leq H \right\} .
    \end{align*}
\end{Def}

\begin{Def}[Sobolev space]
\label{def:sobolev_space}
    Given an open set $\Omega \subseteq \R^r$ for some $r \in \N$ the Sobolev space $W^{k, p}(\Omega)$ with $k\in \Z_+$ and $1 \leq p \leq \infty$ is defined as
    \begin{align*}
        W^{k, p}(\Omega) = \{f \in L^p(\Omega) : \partial^\bk f \in L^p(\Omega) \quad \text{for every $\bk \in \Z_+^r$ with $|\bk| \leq k$} \},
    \end{align*}
    where $L^p(\Omega)$ is the Lebesgue space.
    In what follows, we assume that $p < \infty$.
    For each $l \in \{0, 1, \dots, k\}$, define the Sobolev seminorms on $W^{k, p}(\Omega)$ as
    \begin{align*}
        |f|_{W^{l, p}(\Omega)} = \left\{\sum_{\substack{\bk \in \Z_+^r,\; |\bk| = l}}\|\partial^\bk f\|^p_{L^p(\Omega)}\right\}^{1/p},
        \quad |f|_{W^{l, \infty}(\Omega)} = \max_{\substack{\bk \in \Z_+^r, \; |\bk| = l}}\|\partial^\bk f\|_{L^\infty(\Omega)} .
    \end{align*}
    We also introduce the Sobolev norm on $W^{k, p}(\Omega)$ based on the seminorm
    \begin{align*}
        \|f\|_{W^{k, p}(\Omega)} = \left\{\sum_{0 \leq m \leq k}|f|_{W^{m, p}(\Omega)}^p\right\}^{1 / p}, \quad \|f\|_{W^{k, \infty}(\Omega)} = \max_{0 \leq m \leq k}|f|_{W^{m, \infty}(\Omega)}.
    \end{align*}
    Now let $\sfp : \Omega \to \R$ be a non-negative weight function.
    The weighted Sobolev space $W^{k, p}(\Omega, \sfp)$ consists of all functions $f : \Omega \to \R$ for which the weighted Sobolev seminorm
    \begin{align*}
        |f|_{W^{l, p}(\Omega, \sfp)} = \left\{\sum_{\bk \in \Z_+^r, \, |\bk| = l}\int_\Omega \|\partial^\bk f\|_{L^p(\Omega, \sfp)}^p \right\}^{1 / p}, \quad l \in \{0, 1, \dots, k\} .
    \end{align*}
    is finite.
    In addition, for vector-valued functions $f = (f_1, \dots, f_r) : \Omega \to \R^r$ the definition applies componentwise, i.e. $f \in W^{k, p}(\Omega, \sfp)$ if and only if $f_i \in W^{k, p}(\Omega, \sfp)$ for all $1 \leq i \leq r$.
    The corresponding Sobolev seminorm is defined as follows:
    \begin{align*}
        |f|_{W^{l, p}(\Omega, \sfp)}
        = \left\{ \sum_{i = 1}^r |f_i|^p_{W^{l, p}(\Omega, \sfp)} \right\}^{1/p},
        \quad l \in \{0, 1, \dots, k\} .
    \end{align*}
    When the domain $\Omega$ is clear from the context, we simply write $W^{k, p}(\sfp)$.
\end{Def}

\medskip
\noindent
\textbf{Hermite polynomials.}\quad
Probabilist's Hermite polynomial \(H_k(x)\) for \(x\in\R\) and \(k\in\Z_+\) is given by
\begin{align*}
    H_k(x) = (-1)^k\exp\left\{\frac{x^2}{2}\right\} \frac{\dd^k}{\dd x^k} \exp\left\{-\frac{x^2}{2}\right\}.
\end{align*}
For some \(r\in\Z_+\) and for a multi-index  \(\bk= (k_1, \dots, k_r)\in\Z_+^r\) \textit{multivariate Hermite polynomial} \(\cH_\bk\) is defined as 
\begin{align*}
    \cH_\bk(x) = H_{k_1}(x_1)\cdot \ldots \cdot H_{k_r}(x_r), \quad  x\in\R^r.
\end{align*}
To avoid confusion with the Hölder function class, subscript indexing is preserved for Hermite polynomials.
Let
\[
    \phi_r(x) = \frac{e^{-\|x\|^2 / 2}}{\sqrt{2\pi}} 
\]
stand for the standard Gaussian measure on $\R^r$.
It is known that univariate Hermite polynomials form a complete orthogonal system in $L^2(\R, \phi_1)$ (see \cite[Theorem 5.7.1]{szego1975orthogonal}).
Furthermore, Hermite functions $\{\cH_\bk \, : \, \bk \in \Z_+^r\}$ are complete in $L^2(\R^r, \phi_r)$
(see \cite[Proposition 1]{bongioanni2006sobolev} and \cite{stempak2003poisson}).
In addition, for any $\bk, \bj, \bp \in \Z_+^r$ such that $\bp \leq \bk$, we have
\begin{align*}
    \left\langle \cH_\bk, \cH_\bj \right\rangle_{\phi_r} = \bk! \cdot \1[\bk=\bj],
    \quad \partial^\bp \cH_\bk = \frac{\bk!}{(\bk-\bp)!} \cH_{\bk-\bp} .
\end{align*}

\medskip
\noindent
\textbf{Implicit score matching.}\quad
Let $Y \sim \sfp^*$, where $\sfp^*$ is a probability distribution on $\R^D$.
To avoid confusion, $\sfp^*$ denotes its probability density function.
The \emph{score function} of $\sfp^*$ is given by $s^*(x) = \nabla \log \sfp^*(x)$ for any $x \in \R^D$.
Implicit score matching aims to minimize
\begin{align*}
    \min_{s \in \cS} \E \, \|s(Y) - s^*(Y)\|^2,
\end{align*}
where $\cS$ is the class of score estimators, which will be specified later in the paper.
Since $s^*$ is unknown, \cite{hyvarinen2005estimation} proposed replacing the objective with $\E \ell(s, Y)$, where
\begin{equation}
    \label{eq:ism_loss}
    \ell(s, Y) = \frac{1}{2}\|s(Y)\|^2 + \rdiv[s](Y).
\end{equation}
Under regularity conditions (satisfied in our setting), integration by parts yields
\begin{align*}
    \frac{1}{2}\E \|s(Y) - s^*(Y)\|^2
    = \E \ell(s, Y) + \frac{1}{2}\E \|s^*(Y)\|^2  
\end{align*}
for any $s \in \cS$.
Consequently, for all $s \in \cS$, we have that
\begin{align}
\label{eq:ism_loss_diff_fisher}
    \frac{1}{2}\E \|s(Y) - s^*(Y)\|^2 = \E \ell(s, Y) - \E \ell(s^*, Y).
\end{align}
Given i.i.d. samples $Y_1, \dots, Y_n$ from $\sfp_*$, the implicit score matching estimator is the empirical risk minimizer
\begin{align}
\label{eq:ism_erm}
    \hat{s} = \argmin_{s \in \cS} \left\{ \frac{1}{n}\sum_{i = 1}^n \ell(s, Y_i) \right\} .
\end{align}

\medskip
\noindent
\textbf{Denoising score matching.}\quad
Given the Ornstein-Uhlenbeck forward process defined by
\begin{align}
\label{eq:ou_forward}
    \dd X_t = -X_t \, \dd t + \sqrt{2} \, \dd B_t,
    \quad t \in [0, T],
\end{align}
with initial condition $X_0 \sim \sfp^*$.
Let $\sfp_t^*$ denote the probability density function and the distribution of $X_t$ to avoid ambiguity.
Under mild regularity conditions \citep{anderson1982reverse}, the corresponding reverse process satisfies
\begin{align}
\label{eq:ou_backward}
    \dd Z_t = \big(Z_t + 2 \nabla\log \sfp_{T - t}^*(Z_t) \big) \dd t + \sqrt{2} \, \dd B_t,
    \quad t \in [0, T],
\end{align}
with $Z_0 \sim \sfp_T^*$.
Here, $(W_t)_{t \geq 0}$ and $(B_t)_{t \geq 0}$ are independent Wiener processes on $\R^D$.
By construction, $Z_T \sim \sfp^*$.
Therefore, to draw samples from $\sfp^*$, one can first sample $Z_0 \sim \sfp_T^*$ and then model the backward dynamic \eqref{eq:ou_backward}.
Since we do not have access to the \emph{score function} $\nabla\log \sfp_t^*(y)$, it must be estimated.
For each timestamp $t > 0$, this is achieved by minimizing the empirical version of the score matching loss:
\begin{align*}
    \min_{s \in \cS_{DSM}} \E \, \|s(X_t) - s_t^*(X_t)\|^2.
\end{align*}
Here, $s_t^*(x) = \nabla\log\sfp_t^*(x)$ and $\cS_{DSM}$ denotes the class of score estimates at the considered timestamp.
As in the case of implicit score matching, the score function $s_t^*$ is unknown.
\cite{vincent2011connection} suggest using a surrogate objective, $\E[\ell_t(s, X_0)]$, defined as
\begin{align*}
    \ell_t(s, X_0)
    = \E\left[ \big\|s(X_t) - \nabla_{X_t} \log\sfp^*_t(X_t \,|\, X_0) \big\|^2 \, \big\vert \, X_0\right] .
\end{align*}
Given the Ornstein-Uhlenbeck forward process \eqref{eq:ou_forward}, we have that
\begin{align}
\label{eq:cond_law_ou}
    (X_t \,|\, X_0) \sim \cN \big(m_t X_0, \sigma_t^2 I_D \big),
    \quad
    \text{where $m_t = e^{-t}$ and $\sigma_t^2 = 1 - e^{-2t}$}.
\end{align}
Consequently, the score function of the conditional density is tractable and, thus,
\begin{align*}
    \ell_t(s, X_0)
    = \E\left[\left\| s(X_t) + \frac{X_t - m_t X_0}{\sigma_t^2} \right\|^2 \, \bigg\vert \, X_0\right] .
\end{align*}
Furthermore, \cite{vincent2011connection} claims that
\begin{align*}
    \E \, \|s(X_t) - s_t^*(X_t)\|^2
    = \E \ell_t(s, X_0) + C_t,
\end{align*}
where $C_t$ is a constant depending on $t$ but independent of $s_t^*$ and $s$.
Therefore, we conclude that
\begin{align}
\label{eq:dsm_exp_loss_diff_fisher}
    \E \, \|s(X_t) - s_t^*(X_t)\|^2
    = \E \ell_t(s, X_0) - \E \ell_t(s_t^*, X_0).
\end{align}
Finally, we define the denoising score matching estimate as the empirical risk minimizer
\begin{align}
\label{eq:dsm_erm}
    \hat{s} = \argmin_{s \in \cS_{DSM}}\left\{\frac{1}{n} \sum_{i = 1}^n \ell_t(s, Y_i) \right\},
\end{align}
where $Y_1, \dots, Y_n \sim \sfp^*$ are independent.

\section{Results}
\label{sec:results}

This section collects our main results. We show that both denoising and implicit score matching are able to estimate not only the score function but also its Jacobian matrix without the curse of dimensionality. The results of such type usually require the underlying density $\sfp^*$ to have a low intrinsic dimension. Following \citep{yakovlev2025generalization}, we assume that the data distribution is a convolution of isotropic Gaussian noise with a measure supported on a low-dimensional surface.

\begin{As}[\citep{yakovlev2025generalization}]
\label{asn:relax_man}
Let $g^* \in \cH^\beta([0, 1]^d, \R^D, H)$, for some $H > 0$ and $d, D \in \N$, with $\|g^*\|_{L^\infty([0, 1]^d)} \leq 1$.
We assume that the observed samples $Y_1, \dots, Y_n$ are independent copies of a random element $Y_0 \in \R^D$, generated according to the model
\begin{align*}
    Y_0 = g^*(U) + \sigma \xi,
\end{align*}
where $U \sim \mathrm{Un}([0, 1]^d)$ and $\xi \sim \cN(0, I_D)$ are independent and $\sigma \in [0, 1)$.
\end{As}

In Assumption \ref{asn:relax_man}, we suppose that the intrinsic dimension $d$ is small compared to the ambient dimension $D$. We would like to note that it is not the only way to impose structural assumptions on the data distribution. For instance, \cite{koehler2023statistical} investigated statistical efficiency of implicit score matching under the condition that $\sfp^*$ has a small isoperimetric constant. In \cite{tang2024adaptivity, azangulov2024convergence}, the authors assumed the existence of a hidden smooth low-dimensional manifold. Nevertheless, Assumption \ref{asn:relax_man} has several advantages over the aforementioned setups. First, Assumption accommodates highly multimodal distributions, where the isoperimetric approach is not applicable. Second, unlike \cite{tang2024adaptivity} and \cite{azangulov2024convergence}, we do not require the image of $g^*$ to have a positive reach. More importantly, the density of $g^*(U)$ (with respect to the volume measure) does not have to be bounded away from zero. In \cite{zhang2024minimax}, the authors discussed that the density lower bound assumption may be restrictive and, in particular, does not explain the true strength of generative diffusion models.

We also highlight that adding a small amount of Gaussian noise is crucial when learning the score function, especially when the data distribution lies near a low-dimensional manifold (see \citep{song2019generative}).
The normalization $\|g^*\|_{L^\infty([0, 1]^d)} \leq 1$ together with $\sigma \in [0, 1)$ ensures a controlled signal‑to‑noise ratio, making the score estimation problem well‑posed.
Furthermore, Assumption \ref{asn:relax_man} encompasses the cases when the data distribution has multiple well-separated components, a typical feature real-world data \citep{brown2023verifying}.

Finally, we emphasize that similar data distribution assumptions were used to analyze statistical efficiency of other generative models, including generative adversarial networks \citep{schreuder2021statistical,stephanovitch2024wasserstein,chakraborty2025statistical} and diffusion-based generative models \citep{yakovlev2025generalization}.
Recently, a more general assumption, where the data sample is a convolution of a bounded random vector and Gaussian noise, was used to analyze the iteration complexity of diffusion models \citep{beyler25conv}.

\subsection{Weighted Gagliardo-Nirenberg inequalities for $\sfp^*$}
\label{sec:towards_bernstein_cond}

We start with a key technical result, which can be considered as an extension of the Gagliardo–Nirenberg inequality \citep{nirenberg1959elliptic} with a weight function $\sfp^*$.

\begin{Lem}
\label{lem:bernstein_cond_div}
    Suppose Assumption \ref{asn:relax_man} holds.
    Let also $h : \R^D \to \R^D$ satisfy $|h|_{W^{\alpha, 2}(\sfp^*)} < \infty$ for some integer $\alpha \geq 2$.
    Then it holds that
    \begin{align*}
        (i)
        &\quad
        \|\rdiv[h]\|_{L^2(\sfp^*)}^2 \leq \frac{\alpha D}{\sigma^2} \left( \|h\|_{L^2(\sfp^*)}^2 + \sigma^{2\alpha} \, |h|_{W^{\alpha, 2}(\sfp^*)}^2 \right)^{1 / \alpha} \|h\|_{L^2(\sfp^*)}^{2 - 2 / \alpha},
        \\
        (ii)
        &\quad
        \big\| \|\nabla h\|_F \big\|_{L^2(\sfp^*)}^2 \leq \frac{\alpha D^{1 - 1 / \alpha}}{\sigma^2} \left( D \, \|h\|_{L^2(\sfp^*)}^2 + \sigma^{2\alpha} \, |h|_{W^{\alpha, 2}(\sfp^*)}^2 \right)^{1 / \alpha} \|h\|_{L^2(\sfp^*)}^{2 - 2 / \alpha} .
    \end{align*}
\end{Lem}

We postpone the proof of Lemma \ref{lem:bernstein_cond_div} to Appendix \ref{sec:lem_bernstein_cond_div_proof} and focus on its implications. First, inequality $(i)$ helps us to verify the Bernstein condition (see \citep[Definition 2.6]{bartlett2006empirical}) for the excess loss class. While for denoising score matching this was already done in \cite{yakovlev2025generalization}, this task remained challenging for implicit score matching.
This property is crucial for establishing sharp rates of convergence as it allows us to use localization technique for unbounded empirical processes (see Appendix \ref{sec:stat_learn_th}). Second, the inequality $(ii)$ allows us to control the score Jacobian matrix estimation error. We just have to verify that both $s^*$ and its estimate are sufficiently smooth.
We would like to emphasize that the dependence on the ambient dimension in Lemma \ref{lem:bernstein_cond} is polynomial. This aspect did not receive much attention in some previous works on statistical properties of generative diffusion models, such as \cite{oko2023diffusion, tang2024adaptivity}, where the hidden constant in the rates of convergence is of order $\cO(e^D)$. This significantly limits applicability of those results.

\subsection{Score estimation using implicit score matching}

We proceed with analysis of implicit score matching. To leverage the low intrinsic dimension of $\sfp^*$ following from Assumption \ref{asn:relax_man}, we take a reference class $\cS$ of special structure. Let us note that, under Assumption \ref{asn:relax_man}, the density $\sfp^*$ admits the following form:
\begin{align*}
    \sfp^*(y) = (2\pi\sigma^2)^{-D/2} \int\limits_{[0, 1]^d} \exp\left\{-\frac{\|y - g^*(u)\|^2}{2\sigma^2}\right\} \, \dd u, \quad y \in \R^D .
\end{align*}
Therefore, the score function is given by
\begin{align}
    \label{eq:s_star_f_star_def}
    s^*(y)
    = \nabla \log \sfp^*(y)
    = -\frac{y}{\sigma^2} + \frac{f^*(y)}{\sigma^2},
    \quad f^*(y) = \frac{\integral{[0, 1]^d}g^*(u)\exp\left\{-\frac{\|y - g^*(u)\|^2}{2\sigma^2}\right\} \dd u  }{\integral{[0, 1]^d} \exp\left\{-\frac{\|y - g^*(u)\|^2}{2\sigma^2}\right\} \dd u } .
\end{align}
In other words, $s^*$ is a sum of linear and bounded nonlinear parts. We take this structure of the score function \eqref{eq:s_star_f_star_def} into account when designing the reference class.

\begin{Def}[class of score estimators]
\label{def:score_class}
 For $\alpha \in \N$ and $\sigma_{\min}\in(0, 1)$ the class of the score estimators is defined as 
\begin{align*}
    &\cS(L, W, S, B)
    = \Big\{s(x) = -a x + af(x) \, : \, a \in (1, \sigma^{-2}_{\min}], \, f = (f_1, \dots, f_D)^\top \in\NN(L, W, S, B), \\
    &\qquad \max_{1\leq l \leq D}|f_l|_{W^{0, \infty}(\R^D)} \leq C_0,
    \quad \max_{1\leq l \leq D}|f_l|_{W^{1, \infty}(\R^D)} \leq C_1 \sigma^{-2}_{\min},
    \quad \max_{1\leq l \leq D}|f_l|_{W^{\alpha, 2}(\sfp^*)} \leq C_\alpha \sigma^{-2\alpha}_{\min}\Big\} .
\end{align*}
\end{Def}

We would like to make several remarks regarding Definition \ref{def:score_class}.
The conditions on the Sobolev seminorms are motivated by Lemma \ref{lem:bernstein_cond_div}: before we apply the weighted Gagliardo-Nirenberg inequalities, we must ensure higher-order derivatives of the score estimate candidates are bounded uniformly over $\cS(L, W, S, B)$. We understand that enforcing boundedness of higher-order derivatives across the entire space is a challenging task in practice.
To the best of our knowledge, only first-order regularity is achieved during training, either by adding a regularization term to the loss function \citep{yoshida2017spectral, gulrajani2017improved, pauli2021training} or through the introduction of a more sophisticated neural network architecture \citep{anil19sorting, qi2023lipsformer}.

In what follows, we often $\cS$, instead of $\cS(L, W, S, B)$, when there is no ambiguity.
Note that, as defined in Assumption \ref{asn:relax_man}, the value of $\sigma$ is unknown.
To ensure that the linear term of the score estimate aligns with the true score function, as given in \eqref{eq:s_star_f_star_def}, we introduce a learnable parameter $a \in (1, \sigma_{\min}^{-2}]$.
This requires that $\sigma_{\min}$ is known and satisfies $\sigma_{\min} \leq \sigma$.
These constraints are formalized in the following assumption.

\begin{As}[non-degenerate noise]
\label{as:nondeg_noise}
    Assume that $\sigma_{\min} \leq \sigma$ for a known constant $\sigma_{\min} \in (0, 1)$, where $\sigma$ is a noise level defined in Assumption \ref{asn:relax_man}.
\end{As}

We are ready to formulate our main result on generalization properties of implicit score matching.

\begin{Th}[implicit score matching generalization bound]
\label{th:estimation}  
Grant Assumptions \ref{asn:relax_man} and \ref{as:nondeg_noise}, and let the sample size $n$ be large enough, in the sense that it satisfies
\begin{align*}
    n\sigma_{\min}^{52 + 4P(d, \beta)}
    \gtrsim 1 \vee \left\{ D\sigma^{-2} (\log n + \log(\sigma_{\min}^{-2}) ) \log(nD\sigma^{-2}\log(\sigma_{\min}^{-2})) \right\}^{\frac{2\beta + d}{\beta \wedge 1}} .
\end{align*}
Define $P(d, \beta) = \binom{d + \floor{\beta}}{d}$.
Then, for any $\delta \in (0, 1)$ with probability at least \(1-\delta\) the risk of an empirical risk minimizer over the class \(\cS(L, W, S, B)\) (as defined in Definition \ref{def:score_class} and \eqref{eq:ism_erm}) with 
\begin{align*}
    &L \lesssim \log(nD\sigma_{\min}^{-2}),
    \quad \log B \lesssim D^8 (\log n)^{65} (\log D)^{26}(\log(\sigma_{\min}^{-2}))^{69}, \\
    &\|W\|_\infty \vee S \lesssim D^{16 + P(d, \beta)} (n\sigma_{\min}^{52 + 4P(d, \beta)})^{\frac{d}{2\beta + d}} \sigma_{\min}^{-48 - 4P(d, \beta)} (\log(nD\sigma_{\min}^{-2}))^{142 + 17P(d, \beta)}
\end{align*}
and $C_0 \asymp 1$, $C_1 \asymp (D \log n \log\sigma_{\min}^{-2})^{\cO(1)}$, $C_\alpha \asymp (D \log n \log\sigma_{\min}^{-2})^{\cO(\log(n\sigma_{\min}^{-2}))}$ satisfies the inequality
\begin{equation}
    \label{eq:ism_l2_norm_bound}
    \|\hat{s} - s^*\|_{L^2(\sfp^*)}^2
    \lesssim \sigma^{-8} \left( n \sigma_{\min}^{52 + 4P(d, \beta)}\right)^{-\frac{2\beta}{2\beta + d}} \log(e / \delta) \sL(\sigma_{\min}, D, n),
\end{equation}
where the logarithmic factors are contained within
\begin{align*}
    \sL(\sigma_{\min}, D, n)
    = \left(D \log(n) \log(\sigma_{\min}^{-2}) \right)^{\cO\left(\sqrt{\log n} + \sqrt{\log(\sigma_{\min}^{-2})} \right)} .
\end{align*}
Furthermore, with probability at least $1 - \delta$,
\begin{equation}
    \label{eq:ism_jacobian_bound}
    \big\| \|\nabla(\hat{s} - s^*)\|_F \big\|^2_{L^2(\sfp^*)}
    \lesssim \sigma^{-10}\sigma_{\min}^{-4} \left(n \sigma_{\min}^{52 + 4P(d, \beta)} \right)^{-\frac{2\beta}{2\beta + d}} \log(e / \delta) \sL(\sigma_{\min}, D, n).
\end{equation}
\end{Th}

The proof of Theorem \ref{th:estimation} is moved to Section \ref{sec:estimation_proof}.
Before comparing our result with the previous work, let us highlight key implications.
First, we derive the rate of convergence $\cO(n^{-2\beta / (2\beta + d)})$ for both score and its Jacobian matrix estimation, which is determined by the effective dimension, rather than the ambient one.
Unlike kernel-based score estimators \citep{wibisono2024optimal, zhang2024minimax} and the procedures based on Stein's method \citep{li2018gradient,shi2018spectral}, the bounds of Theorem \ref{th:estimation} do not become vacuous when $D \asymp \log n$. 
It also indicates that estimating the Jacobian matrix is as easy as learning the score function in terms of the sample size.
It is not surprising, because the score function $s^*$ is analytic in our setup (see \cite[Lemma 4.1]{yakovlev2025generalization}).
Second, we note that $\sL(\sigma_{\min}, D, n)$ grows in the sample size $n$ slower than any polynomial and, hence, it does not ruin the derived fast rate.
Moreover, the rate does not deteriorate if $D \asymp \log n$.

To the best of our knowledge, only \cite{shen2024differentiable} has considered a setup similar to ours. This work assumed that the data distribution is supported on a neighborhood of a Riemannian submanifold.
However, this bounded support assumption does not explain the success of score-based generative models \citep{song2019generative} when the data is corrupted with Gaussian noise.
In contrast, our result successfully addresses this case.
Furthermore, \cite{shen2024differentiable} do not provide convergence rates for the score Jacobian matrix estimation.

In \citep[Lemma 12]{sasaki18mode}, the authors establish estimation error bounds on the score function and its Jacobian matrix with respect to the infinity norm on a compact set, using a least-squares density-derivative ratio method.
However, as discussed in the introduction, these bounds may be affected by the curse of dimensionality, which is avoided in Theorem \ref{th:estimation}. The same concerns the methods based on kernel density estimation \cite{wibisono2024optimal, zhang2024minimax}.
Furthermore, kernel density estimation methods suffer not only from the curse of dimensionality but also from deteriorating convergence rates as the derivative order increases \citep{genovese2014nonparametric}.

Theorem \ref{th:estimation} also extends the statistical guarantees of \cite{ghosh2025stein} for implicit score matching, who considered a univariate case.
Specifically, Theorem 7 in \cite{ghosh2025stein} derives fast score estimation rates under the assumption that the data is a convolution of a distribution with support contained in $[-M, M]$ and a sub-Gaussian noise.
However, this work does not explicitly state the dependence on $M$ or, in view of Assumption \ref{asn:relax_man}, does not provide a dependence on the noise scale parameter $\sigma$, which is a crucial aspect for analyzing distribution estimation in diffusion models, as highlighted by \cite{tang2024adaptivity, azangulov2024convergence, yakovlev2025generalization}.

By comparing the convergence rates provided by Theorem \ref{th:estimation} with with those achieved by denoising score matching \citep{yakovlev2025generalization} under identical assumptions, we see that implicit score matching achieves the same rate of $\cO(n^{-2\beta / (2\beta + d)})$.
Thus, implicit score matching is statistically efficient as denoising score matching, even though the latter has access to the original noiseless samples.
Although the dependence on the noise scale parameter in Theorem \ref{th:estimation} is less favourable than in \cite{yakovlev2025generalization}, it still remains polynomial.

\subsection{Score estimation using denoising score matching}

Finally, we discuss our findings on statistical properties of denoising score matching.
In particular, we show that, under Assumption \ref{asn:relax_man}, Jacobian matrix of the denoising score matching output yields a consistent log-density Hessian estimate, provided that the reference class $\cS_{DSM}$ is chosen properly.
To facilitate our analysis, we fix an arbitrary $t > 0$ and investigate the statistical efficiency of denoising score matching at this specific timestamp.
This focus is motivated by the requirement for accurate estimation of both the score function and its Jacobian matrix at designated timestamps \citep{li2024accelerating, li2024towards, li2024sharp, li2025faster}.

Assumption \ref{asn:relax_man} in conjunction with the conditional law of the Ornstein-Uhlenbeck process given in \eqref{eq:cond_law_ou} implies that the density of $X_t$, for any $t > 0$, can be written in a closed form as
\begin{align*}
    \sfp_t^*(y) = \big( 2\pi(m_t^2\sigma^2 + \sigma_t^2) \big)^{-D / 2}\integral{[0, 1]^d}\exp\left\{-\frac{\|y - m_tg^*(u)\|^2}{2(m_t^2\sigma^2 + \sigma_t^2)}\right\} \, \dd u,
    \quad y \in \R^D .
\end{align*}
Now let us fix an arbitrary $t > 0$.
Therefore, by differentiating the logarithm of the derived density, we arrive at
\begin{align}
\label{eq:true_score_def}
    s_t^*(y) = -\frac{y - m_t f^*(y, t)}{m_t^2\sigma^2 + \sigma_t^2},
    \quad \text{where} \quad
    \quad f^*(y, t) = \frac{\integral{[0, 1]^d}g^*(u)\exp\left\{-\frac{\|y - m_tg^*(u)\|^2}{2(m_t^2\sigma^2 + \sigma_t^2)} \right\} \, \dd u }{\integral{[0, 1]^d}\exp\left\{-\frac{\|y - m_tg^*(u)\|^2}{2(m_t^2\sigma^2 + \sigma_t^2)} \right\} \, \dd u} .
\end{align}
Based on the expression for the true score function given in \eqref{eq:true_score_def}, we formulate the following definition of its estimate.

\begin{Def}
\label{def:dsm_score_class}
    For $\alpha \in \N$ and $t > 0$ define
    \begin{align*}
        \cS_{DSM}(L, W, S, B)
        = \bigg\{ s(x) = -\frac{x}{m_t^2\tilde{\sigma}^2 + \sigma_t^2} + \frac{m_t f(x)}{m_t^2\tilde{\sigma}^2 + \sigma_t^2} \, : \, \tilde{\sigma} \in [0, 1), \, f \in \NN(L, W, S, B),
        &\\
        \max_{1 \leq l \leq D}|f_l|_{W^{0, \infty}(\R^D)} \leq C_0,
        \, \max_{1 \leq l \leq D}|f_l|_{W^{\alpha, 2}(\sfp_t^*)} \leq C_\alpha\sigma_t^{-2\alpha}
        &
        \bigg\}.
    \end{align*}
\end{Def}
Non-asymptotic high-probability upper bounds on the generalization error and the accuracy of score Jacobian matrix estimation by denoising score matching are given in the next theorem.

\begin{Th}[denoising score matching generalization bound]
\label{th:estimation_dsm}
Under Assumption \ref{asn:relax_man}, suppose the sample size satisfies
\begin{align*}
    n\sigma_{\min}^{52 + 4P(d, \beta)}
    \gtrsim \left\{ D(m_t^2\sigma^2 + \sigma_t^2)^{-1} \log^2(nD\sigma_t^{-2}) \right\}^{\frac{2\beta + d}{\beta \wedge 1}} .
\end{align*}
Define $P(d, \beta) = \binom{d + \floor{\beta}}{d}$.
Then, for any $\delta \in (0, 1)$, with probability at least $1 - \delta$, the empirical risk minimizer $\hat{s}$ over the class $\cS_{DSM}(L, W, S, B)$ (as defined in \eqref{eq:dsm_erm}) with
\begin{align*}
    &L \lesssim \log(nD\sigma_t^{-2}),
    \quad \log B \lesssim D^8 (\log(Dn\sigma_t^{-2}))^{90}, \\
    &\|W\|_\infty \vee S \lesssim D^{16 + P(d, \beta)} (n\sigma_t^{52 + 4P(d, \beta)})^{\frac{d}{2\beta + d}} \sigma_t^{-48 - 4P(d, \beta)} (\log(nD\sigma_t^{-2}))^{142 + 17 P(d, \beta)}
\end{align*}
and
\begin{align*}
    C_0 \asymp 1,
    \quad C_\alpha \asymp (D\log n \log(\sigma_t^{-2}))^{\cO(\log(n\sigma_t^{-2}))}
\end{align*}
satisfies the inequality
\begin{align*}
    \|\hat{s} - s_t^*\|^2_{L^2(\sfp_t^*)}
    \lesssim D^{25 + P(d, \beta)}(m_t^2\sigma^2 + \sigma_t^2)^{-4} (n\sigma_t^{52 + 4P(d, \beta)})^{-\frac{2\beta}{2\beta + d}} \log(e / \delta) \sL(\sigma_t, D, n),
\end{align*}
where
\begin{align*}
    \sL(\sigma_t, D, n) = (\log(Dn\sigma_t^{-2}))^{239 + 17 P(d, \beta)} \exp\left\{\cO\left( \sqrt{\log n} + \sqrt{\log(\sigma_t^{-2})} \right) \right\} .
\end{align*}
Furthermore, on the same event of probability at least $1 - \delta$, it holds that
\begin{align*}
    \big\| \|\nabla(\hat{s} - s_t^*)\|_F \big\|^2_{L^2(\sfp_t^*)}
    \lesssim \frac{(D\log n \log(\sigma_t^{-2}))^{\cO(\sqrt{\log n} + \sqrt{\log(\sigma_t^{-2})})} \log(e / \delta)}{\sigma_t^4(m_t^2\sigma^2 + \sigma_t^2)^5} (n\sigma_t^{52 + 4P(d, \beta)})^{-\frac{2\beta}{2\beta + d}} .
\end{align*}

\end{Th}

The proof of Theorem \ref{th:estimation_dsm} is postponed to Appendix \ref{sec:th_estimation_dsm_proof}.
Now compare the result outlined in Theorem \ref{th:estimation_dsm} with previous work.
\cite{yakovlev2025generalization} consider the identical setup and achieve the same convergence rate in terms of the sample size.
However, our work exhibits slower dependence on $\sigma_t$, which remains polynomial.
We argue that this is due to the approximation result given in Theorem \ref{thm:approx_main} that offers a bit worse dependence on the noise scale when approximating higher order derivatives.
Similar to Theorem \ref{th:estimation}, we establish a $\cO(n^{-2\beta / (2\beta + d)})$ rate of convergence for the score Jacobian matrix estimation.
Notably, this rate in terms of the sample size is as fast as that for the score function estimation. Similarly to Theorem \ref{th:estimation}, the reason for such phenomenon is that $s^*$ is an analytic function (see \cite[Lemma 4.1]{yakovlev2025generalization}). 
We also note that the condition on the $W^{\alpha, 2}(\sfp_t^*)$-norm in Definition \ref{def:dsm_score_class} can be removed if a Jacobian matrix estimation is not required.

The existing work on statistical efficiency of denoising score matching that avoids the curse of dimensionality \citep{chen2023score,tang2024adaptivity,azangulov2024convergence,yakovlev2025generalization} does not address score Jacobian matrix estimation problem.
Since these studies provide no approximation guarantees for the higher-order derivatives of the score function, a straightforward application of our weighted Gagliardo-Nirenberg inequality (see Lemma \ref{lem:bernstein_cond_div}) fails to establish an error bound for the score Jacobian matrix estimation.
Consequently, in \cite{li2024accelerating, li2024towards, li2024sharp, li2025faster} the score Jacobian matrix estimation error is no longer a limiting factor, since it is estimated at the same rate as the score function, as indicated by Theorem \ref{th:estimation_dsm}.

\section{Proof of Theorem \ref{th:estimation}}
\label{sec:estimation_proof}

We split the proof into several steps for convenience. We focus our attention on the upper bound $\|\hat s - s^*\|_{L_2(\sfp^*)}^2$ (see Steps 1--5) and postpone the proof of \eqref{eq:ism_jacobian_bound} to the very end (Step 6). A reader will observe that the second statement of the theorem easily follows from our findings presented on Steps 1--5.

As usually in learning theory, to control generalization error of the implicit score matching estimate we have to bound approximation and estimation errors. Fortunately, the upper bound on the former one is straightforward in view of the recent result of \cite[Theorem 3.2]{yakovlev2025simultaneous} (see Step 5). In contrast, the study of estimation error requires more efforts. First, we make some preparatory steps (Steps 1--3) to study properties of the excess loss class
\begin{equation}
    \label{eq:excess_loss_class}
    \cF = \big\{ \ell(s, \cdot) - \ell(s^*, \cdot) : s \in \cS(L, W, S, B) \big\}.
\end{equation}
In particular, on Step 1 we show that $\cF$ has a bounded $\psi_1$-diameter. On the second step, we use the weighted Gagliardo-Nirenberg inequality (Lemma \ref{lem:bernstein_cond_div}) to show that $\cF$ satisfies Bernstein's condition. Finally, on Step 3 we elaborate on the covering number of $\cF$. All these steps are necessary to apply a localization argument for unbounded empirical processes (Theorem \ref{th:tail_ineq_unb_new}).

\noindent
\textbf{Step 1: $\psi_1$-diameter of $\cF$.}
\quad
We start with an upper bound on the $\psi_1$-diameter of the excess loss class $\cF$.
The next lemma shows that it scales linearly with $D$ and $\sigma_{\min}^{-4}$.
\begin{Lem}
\label{lem:psi1_norm_bound}
    Grant Assumption \ref{asn:relax_man}.
    Then for $Y \sim \sfp^*$ it holds that
    \begin{align*}
        \left\|\sup_{s \in \cS(L, W, S, B)}|\ell(s, Y) - \ell(s^*, Y)|\right\|_{\psi_1}
        \lesssim \sigma_{\min}^{-4}D(C_0^2 \vee C_1 \vee 1).
    \end{align*}
\end{Lem}

We defer the proof of Lemma \ref{lem:psi1_norm_bound} to Appendix \ref{sec:lem_psi1_norm_bound_proof} and move further.

\noindent
\textbf{Step 2: verifying the Bernstein condition.}\quad
The goal of this step is to bound the variance of the excess loss $\ell(s, Y) - \ell(s^*, Y)$ by a function of the excess risk $\E \big[ \ell(s, Y) - \ell(s^*, Y) \big] = 0.5 \|s - s^*\|_{L^2(\sfp^*)}^2$. Recalling the definition of $\ell$ (see \eqref{eq:ism_loss}), we observe that
\[
    \ell(s, Y) - \ell(s^*, Y) = \frac{1}{2} \left( \|s(Y)\|^2 - \|s^*(Y)\|^2 \right) + \rdiv[s - s^*](Y)
    \quad \text{for any $s \in \cS(L, W, S, B)$.}
\]
To bound the variance of the first term in the right-hand side, we use properties of sub-Gaussian and sub-exponential random variables. An upper bound on the variance of the second term follows from Lemma \ref{lem:bernstein_cond_div}. A rigorous result is formulated below.
\begin{Lem}
\label{lem:bernstein_cond}
    For an arbitrary $s \in \cS(L, W, S, B)$, $\alpha \in \N$ with $\alpha \geq 2$ and for $Y \sim \sfp^*$, it holds that
    \begin{align*}
        (i) &\quad \Var[\ell(s, Y) - \ell(s^*, Y)]
        \lesssim \alpha^2 D^{1 + 1 / \alpha}(D + \alpha)\sigma^{-2}\sigma_{\min}^{-4 - 4 / \alpha}(C_0^{2 + 2 / \alpha} \vee C_\alpha^{2 / \alpha} \vee 1)\|s - s^*\|_{L^2(\sfp^*)}^{2 - 2 / \alpha}, \\
        (ii) &\quad \big\| \|\nabla(s - s^*)\|_F \big\|_{L^2(\sfp^*)}^2
        \lesssim \alpha^2 D^{1 + 1 / \alpha}(D + \alpha)\sigma^{-2}\sigma_{\min}^{-4 - 4 / \alpha}(C_0^{2 / \alpha} \vee C_\alpha^{2 / \alpha} \vee 1)\|s - s^*\|_{L^2(\sfp^*)}^{2 - 2 / \alpha} ,
    \end{align*}
    where the hidden constant is absolute and does not depend on $s$.
\end{Lem}
We defer the proof of Lemma \ref{lem:bernstein_cond} to Appendix \ref{sec:lem_bernstein_cond_proof} and focus on the covering number of the excess loss class $\cF$.

\noindent
\textbf{Step 3: covering number evaluation.}\quad
Our next goal is to show that the covering number $\cN(\tau, \cF, L^2(\sfP_n))$ of the excess loss class $\cF$ (see \eqref{eq:excess_loss_class}) grows polynomially with respect to $(1/\tau)$. Here and further in this paper, $L^2(\sfP_n)$ stands for the empirical $L^2$-norm. Let us note that $\cN(\cF, L^2(\sfP_n), \eps)$ is the same as the corresponding covering number of the loss class
\begin{equation}
    \label{eq:ism_loss_class}
    \cL = \left\{\ell(s, \cdot) \, : \, s\in\cS(L, W, S, B)\right\}.
\end{equation}
We start with the following result concerning the covering number of $\cL$ with respect to the $L^\infty$-norm on a compact set.

\begin{Lem}
\label{lem:loss_cov_num_bound_true}
    For any $\tau$ satisfying $0 < \tau \leq \cD(\cL, \|\cdot\|_{L^\infty([-R, R]^D)})$ and $R > 0$, the covering number of the class $\cL$ given by \eqref{eq:ism_loss_class} satisfies the inequality
    \[
        \log\cN(\tau, \cL, \|\cdot\|_{L^\infty([-R, R]^D)})
        \lesssim SL \log\left( \tau^{-1}\sigma_{\min}^{-2} D L (R \vee C_0 \vee C_1 \vee 1) (B \vee 1)(\|W\|_\infty + 1) \right).
    \]
\end{Lem}
The proof of Lemma \ref{lem:loss_cov_num_bound_true} is moved to Section \ref{sec:cov_num_bound}.
Let
\[
    R_n = \max\limits_{1 \leq i \leq n} \|Y_i\|.
\]
Then, for any $ 0 < \tau \leq \cD(\cF, L^2(\sfP_n))$, it holds that
\begin{align*}
    \log \cN(\tau, \cF, L^2(\sfP_n))
    = \log \cN(\tau, \cL, L^2(\sfP_n))
    \leq \log \cN\left(\tau, \cL, \|\cdot\|_{L^\infty([-R_n, R_n]^D)}\right),
\end{align*}
where $\cF$ and $\cL$ are defined in \eqref{eq:excess_loss_class} and \eqref{eq:ism_loss_class}, respectively.
Consequently, from Lemma \ref{lem:loss_cov_num_bound_true}, it follows that
\begin{align*}
    \log \cN(\tau, \cL, \|\cdot\|_{L^\infty([-R_n, R_n]^D)})
    &
    \lesssim SL \log(D_n / \tau),
\end{align*}
where we defined
\begin{equation}
    \label{eq:D_n}
    D_n = L D \sigma_{\min}^{-2}(R_n \vee C_0 \vee C_1 \vee 1) (B \vee 1)(\|W\|_\infty + 1) .
\end{equation}
Furthermore, the third moment of $\log D_n$ satisfies the inequality
\[
    (\E \log^3 D_n)^{1 / 3}
    \lesssim \log\big(D \sigma_{\min}^{-2}(C_0 \vee C_1 \vee 1) L (B \vee 1)(\|W\|_\infty + 1) \big)
    + \left( \E \max_{1 \leq i \leq n}\log^3 (\|X_i\| \vee 1) \right)^{1 / 3} .
\]
The second term in the right-hand side admits the following bound derived in Appendix \ref{sec:lem_exp_max_log_data_proof}.

\begin{Lem}
\label{lem:exp_max_log_data}
    Grant Assumption \ref{asn:relax_man} and assume $n \geq 2$.
    Then it holds that
    \begin{align*}
        \E\max_{1 \leq i \leq n}\log^3 (\|Y_i\| \vee 1) \lesssim \log^3(1 + \sigma^2 D)\sqrt{\log(2n)}.
    \end{align*}
\end{Lem}
Therefore, by Lemma \ref{lem:exp_max_log_data}, we have that
\begin{align*}
    (\E \log^3 D_n)^{1 / 3}
    \lesssim \log(D \sigma_{\min}^{-2}(C_0 \vee C_1 \vee 1) L (B \vee 1)(\|W\|_\infty + 1)) \log(2n).
\end{align*}
We will use this fact on the next step of the proof.

\noindent
\textbf{Step 4: tail inequality.}\quad
We are in position to apply Theorem \ref{th:tail_ineq_unb_new} to the excess loss class $\cF$. This theorem extends a localization argument (see, for instance, \citep{bartlett2005local}) to the case of unbounded empirical processes. In view of Steps 1--3, $\cF$ meets the conditions of Theorem \ref{th:tail_ineq_unb_new} with $\varkappa = 1 - 1/\alpha$ (see Step 2), $A = \cO(1)$, $\zeta = \cO(SL)$, and $D_n$ given by \eqref{eq:D_n} (see Step 3).
Then, in view of Remark \ref{rem:erm_localization_bound}, taking $\eps = 1/2$, we obtain that for any $\delta \in (0, 1)$, with probability at least $1 - \delta$, the empirical risk minimizer $\hat s$ satisfies the inequality
\begin{align*}
    \E_{Y \sim \sfp^*} \big[ \ell(\hat{s}, Y) - \ell(s^*, Y) \big]
    &
    \lesssim \inf_{s \in \cS(L, W, S, B)} \E[\ell(s, Y) - \ell(s^*, Y)] + \left( B_\alpha \Upsilon(n, \delta) \right)^{\alpha / (\alpha + 1)}
    \\&\quad 
    + \left( \left\| \sup_{s \in \cS(L, W, S, B)} \big|\ell(s, Y) - \ell(s^*, Y) \big| \right\|_{\psi_1} \vee 1 \right) \Upsilon(n, \delta) \log n ,
\end{align*}
where
\begin{align}
\label{eq:ism_upsilon_bound_gen_aux}
    \Upsilon(n, \delta) \lesssim \frac{SL}n \log\big( D\sigma_{\min}^{-2}(C_0 \vee C_1 \vee 1) L (B \vee 1)(\|W\|_\infty + 1) \big) \log(2n) + \frac{\log(e / \delta)}n
\end{align}
and the value of $B_\alpha$ is given by Lemma \ref{lem:bernstein_cond}:
\begin{align}
\label{eq:B_alpha_def}
    B_\alpha = \alpha^2 D^{1 + 1 / \alpha}(D + \alpha)\sigma^{-2}\sigma_{\min}^{-4 - 4 / \alpha}(C_0^{2 + 2 / \alpha} \vee C_\alpha^{2 / \alpha} \vee 1) .
\end{align}
Note that, according to Lemma \ref{lem:psi1_norm_bound}, we have
\[
    \left\| \sup_{s \in \cS(L, W, S, B)} \big|\ell(s, Y) - \ell(s^*, Y) \big| \right\|_{\psi_1}
    \lesssim \sigma_{\min}^{-4}D(C_0^2 \vee C_1 \vee 1).
\]
Taking this into account and recalling the identity
$\E[\ell(s, Y) - \ell(s^*, Y)] = 0.5 \|s - s^*\|^2_{L^2(\sfp^*)}$ for all $s \in \cS(L, W, S, B)$ (see \eqref{eq:ism_loss_diff_fisher}), we obtain that
\begin{align}
\label{eq:gen_bound_oracle}
    \notag
    \|\hat{s} - s^*\|^2_{L^2(\sfp^*)}
    &\lesssim \inf_{s \in \cS(L, W, S, B)}\|s - s^*\|^2_{L^2(\sfp^*)} + \left( B_\alpha \Upsilon(n, \delta) \right)^{1 / (2 - \varkappa)}
    \\&\quad
    + \sigma_{\min}^{-4}D(C_0^2 \vee C_1 \vee 1) \Upsilon(n, \delta) \log n
\end{align}
with probability at least $1 - \delta$.

\noindent
\textbf{Step 5: deriving the final generalization bound.}\quad
It remains to bound the approximation error in the right-hand side of \eqref{eq:gen_bound_oracle}.
It easily follows from a recent result of \cite{yakovlev2025simultaneous} (see Theorem \ref{thm:approx_main}).
Let $\eps \in (0, 1)$ stand for a the precision parameter $\eps \in (0, 1)$ to be determined a bit later and let us set the architecture $(L, W, S, B)$ as specified in Theorem \ref{thm:approx_main} applied with $m = \alpha$.
Therefore, it suffices to take
\begin{align}
\label{eq:C_0_1_alpha_def}
    \quad C_j \asymp \exp\left\{\cO \left(j^2 \log\left(\alpha D \log\left(\frac1{\eps} \right) \log \frac1{\sigma_{\min}^{2}} \right) \right) \right\}, \quad j \in \{0, 1, \alpha\}
\end{align}
in the definition of the score estimators class (see Definition \ref{def:score_class}) to ensure that
\begin{align*}
    \inf_{s \in \cS(L, W, S, B)} \|s - s^*\|_{L^2(\sfp^*)}^2
    \lesssim D^2\sigma^{-8}\eps^{2\beta}\log^2(1 / \eps) \log^2(\alpha D \sigma^{-2}) .
\end{align*}
In addition, from \eqref{eq:ism_upsilon_bound_gen_aux} and Theorem \ref{thm:approx_main} we deduce that
\begin{align*}
    n\Upsilon(n, \delta)
    &\lesssim SL \log\left( Dn\sigma_{\min}^{-2} L (B \vee 1)(\|W\|_\infty + 1)\alpha \log(1 / \eps) \right)  + \log(e / \delta) \\
    &\lesssim \frac{D^{24 + P(d, \beta)}\alpha^{217 + 17P(d, \beta)}}{\eps^{d} \sigma_{\min}^{48 + 4P(d, \beta)}} \left( \log\left(\frac{\alpha D}{\sigma_{\min}^{2}} \right) \log\frac1\eps \right)^{65 + 4P(d, \beta)} \log(2n)
    + \log(e / \delta) .
\end{align*}
Furthermore, $B_\alpha$ given in \eqref{eq:B_alpha_def} is evaluated as
\begin{align*}
    B_\alpha \lesssim \frac{\alpha^2 D^{1 + 1 / \alpha}}{\sigma^{2}\sigma_{\min}^{4 + 4 / \alpha}} \exp\left\{\cO \left(\alpha  \log\left(\alpha D \log\left(\frac1{\eps} \right) \log \frac1{\sigma_{\min}^{2}} \right) \right) \right\} .
\end{align*}
Therefore, we obtain from \eqref{eq:gen_bound_oracle} that, with probability at least $1 - \delta$,
\begin{align*}
    \|\hat{s} - s^*\|^2_{L^2(\sfp^*)}
    &\lesssim D^{\cO(1)}\left\{ \sigma^{-8}\eps^{2\beta}
    + \left(\frac{\eps^{-d} \log(e / \delta)}{n \cdot \sigma_{\min}^{52 + 4 / \alpha + 4P(d, \beta)}\sigma^{2}} \right)^{\alpha / (\alpha + 1)} \right\} \sL(\eps, \sigma_{\min}, \alpha, D, n),
\end{align*}
where we introduced
\begin{align*}
    \sL(\eps, \sigma_{\min}, \alpha, D, n)
    &
    = \left(\alpha \log\left(\frac1\eps \right) \log\left( \frac1{\sigma_{\min}^{2}} \right)
    \log n \right)^{\cO(1)}
    \\&\quad
    \cdot \exp \left\{ \cO\left( \alpha \log(\alpha D) + \alpha\log\log\left(\frac1\eps \right) + \alpha \log \log\left(\frac1{\sigma_{\min}^{2}} \right) \right) \right\} .
\end{align*}
Hence, setting $\alpha \asymp \sqrt{\log n} + \sqrt{\log(\sigma_{\min}^{-2})}$ ensures that, with probability at least $1 - \delta$,
\begin{align*}
    \|\hat{s} - s^*\|_{L^2(\sfp^*)}^2
    &\lesssim \left\{\sigma^{-8} \eps^{2\beta} + \frac{\eps^{-d} \log(e / \delta)}{n \cdot \sigma_{\min}^{52 + 4P(d, \beta)}\sigma^2}\right\} \sL(\eps, \sigma_{\min}, D, n),
\end{align*}
where
\begin{align*}
    \sL(\eps, \sigma_{\min}, D, n)
    = (D \log n \log(1 / \eps) \log(\sigma_{\min}^{-2}))^{\cO(\sqrt{\log n} + \sqrt{\log(\sigma_{\min}^{-2})})} .
\end{align*}
Choosing $\eps = (n \sigma_{\min}^{52 + 4P(d, \beta)})^{-1 / (2 \beta + d)} \in (0, 1)$ yields that, with probability at least $1 - \delta$,
\begin{align}
\label{eq:score_estim_final_proof}
    \|\hat{s} - s^*\|_{L^2(\sfp^*)}^2
    \lesssim \sigma^{-8} (n \sigma_{\min}^{52 + 4P(d, \beta)})^{-\frac{2\beta}{2\beta + d}} \log(e / \delta) \sL(\sigma_{\min}, D, n),
\end{align}
where
\begin{align}
\label{eq:log_fact_estim_final_proof}
    \sL(\sigma_{\min}, D, n)
    = (D \log n \log(\sigma_{\min}^{-2}))^{\cO\left(\sqrt{\log n} + \sqrt{\log(\sigma_{\min}^{-2})} \right)} .
\end{align}
Note that we also have to make sure that the choice of $\eps$ should be sufficiently small according to Theorem \ref{thm:approx_main}.
Thus, we have that
\begin{align*}
    n\sigma_{\min}^{52 + 4P(d, \beta)}
    \gtrsim 1 \vee \left\{ D \sigma^{-2} \alpha^2 \log(n\alpha D \sigma^{-2}) \right\}^{\frac{2\beta + d}{\beta}}
    \vee \left\{ D\sigma^{-2}\alpha^2\log(n\alpha D\sigma^{-2}) \right\}^{2\beta + d} .
\end{align*}
Taking into account the choice of $\alpha$, we deduce that it suffices to ensure that
\begin{align*}
    n\sigma_{\min}^{52 + 4P(d, \beta)}
    \gtrsim 1 \vee \left\{ D\sigma^{-2} (\log n + \log(\sigma_{\min}^{-2}) ) \log(nD\sigma^{-2}\log(\sigma_{\min}^{-2})) \right\}^{\frac{2\beta + d}{\beta \wedge 1}} .
\end{align*}
Finally, it remains to derive the parameters of the score estimators class $\cS(L, W, S, B)$.
Given the choice of $\eps$, we deduce from \eqref{eq:C_0_1_alpha_def} that
\begin{align}
\label{eq:C_0_1_alpha_final_def}
    C_0 \asymp 1,
    \quad C_1 \asymp (D \log n \log\sigma_{\min}^{-2})^{\cO(1)},
    \quad C_\alpha \asymp (D \log n \log\sigma_{\min}^{-2})^{\cO(\log(n\sigma_{\min}^{-2}))} .
\end{align}
Furthermore, we utilize the result from Theorem \ref{thm:approx_main} to specify the parameters of the neural network architecture.
Specifically, we have
\[
    L \lesssim \log \frac{nD}{\sigma_{\min}^{2}},
    \quad
    \log B \lesssim D^8 (\log n)^{65} (\log D)^{26} \left( \log\frac1{\sigma_{\min}^{2}} \right)^{69},
\]
and
\[
    \|W\|_\infty \vee S \lesssim D^{16 + P(d, \beta)} \left( n\sigma_{\min}^{52 + 4P(d, \beta)} \right)^{\frac{d}{2\beta + d}} \sigma_{\min}^{-48 - 4P(d, \beta)} \left( \log\frac{nD}{\sigma_{\min}^{2}} \right)^{142 + 17P(d, \beta)} .
\]
Here, $P(d, \beta) = \binom{d + \floor{\beta}}{d}$.

\noindent
\textbf{Step 6: Jacobian matrix estimation.}\quad
The generalization bound for the Jacobian estimator $\nabla\hat{s}$ is a straightforward implication of the above results.
Recall from Lemma \ref{lem:bernstein_cond} that
\begin{align*}
    \big\| \|\nabla(\hat{s} - s^*)\|_F \big\|^2_{L^2(\sfp^*)}
    \lesssim \frac{\alpha^2 D^{1 + 1 / \alpha}(D + \alpha)}{\sigma^{2}\sigma_{\min}^{4 + 4 / \alpha}} \left(C_0^{2 / \alpha} \vee C_\alpha^{2 / \alpha} \vee 1 \right) \|s - s^*\|_{L^2(\sfp^*)}^{2 - 2 / \alpha} .
\end{align*}
Therefore, from \eqref{eq:score_estim_final_proof}, \eqref{eq:log_fact_estim_final_proof}, and \eqref{eq:C_0_1_alpha_final_def} we conclude that, with probability at least $1 - \delta$,
\begin{align*}
    \big\| \|\nabla(\hat{s} - s^*)\|_F \big\|^2_{L^2(\sfp^*)}
    &
    \lesssim \sigma^{-2}\sigma_{\min}^{-4} \left(D \log n \log \sigma_{\min}^{-2} \right)^{\cO \left(\sqrt{\log n} + \sqrt{\log(\sigma_{\min}^{-2})} \right)} \|\hat{s} - s^*\|_{L^2(\sfp^*)}^{2 - 2 / \alpha} \\
    &
    \lesssim \sigma^{-10}\sigma_{\min}^{-4} \left( n \sigma_{\min}^{52 + 4P(d, \beta)} \right)^{-\frac{2\beta}{2\beta + d}} \log(e / \delta) \sL(\sigma_{\min}, D, n) .
\end{align*}
The proof is finished.

\myendproof

\bibliographystyle{abbrvnat}
\bibliography{references}

\begin{thebibliography}{67}
\providecommand{\natexlab}[1]{#1}
\providecommand{\url}[1]{\texttt{#1}}
\expandafter\ifx\csname urlstyle\endcsname\relax
  \providecommand{\doi}[1]{doi: #1}\else
  \providecommand{\doi}{doi: \begingroup \urlstyle{rm}\Url}\fi

\bibitem[Adamczak(2008)]{adamczak2008tail}
R.~Adamczak.
\newblock A tail inequality for suprema of unbounded empirical processes with
  applications to markov chains.
\newblock \emph{Electronic Journal of Probability}, 13:\penalty0 1000--1034,
  2008.

\bibitem[Anderson(1982)]{anderson1982reverse}
B.~D. Anderson.
\newblock Reverse-time diffusion equation models.
\newblock \emph{Stochastic Processes and their Applications}, 12\penalty0
  (3):\penalty0 313--326, 1982.

\bibitem[Anil et~al.(2019)Anil, Lucas, and Grosse]{anil19sorting}
C.~Anil, J.~Lucas, and R.~Grosse.
\newblock Sorting out {L}ipschitz function approximation.
\newblock In \emph{Proceedings of the 36th International Conference on Machine
  Learning}, volume~97 of \emph{Proceedings of Machine Learning Research},
  pages 291--301. PMLR, 2019.

\bibitem[Azangulov et~al.(2024)Azangulov, Deligiannidis, and
  Rousseau]{azangulov2024convergence}
I.~Azangulov, G.~Deligiannidis, and J.~Rousseau.
\newblock Convergence of diffusion models under the manifold hypothesis in
  high-dimensions.
\newblock Preprint. ArXiv:2409.18804, 2024.

\bibitem[Bartlett and Mendelson(2006)]{bartlett2006empirical}
P.~L. Bartlett and S.~Mendelson.
\newblock Empirical minimization.
\newblock \emph{Probability theory and related fields}, 135\penalty0
  (3):\penalty0 311--334, 2006.

\bibitem[Bartlett et~al.(2005)Bartlett, Bousquet, and
  Mendelson]{bartlett2005local}
P.~L. Bartlett, O.~Bousquet, and S.~Mendelson.
\newblock Local {R}ademacher complexities.
\newblock \emph{The Annals of Statistics}, 33\penalty0 (4):\penalty0
  1497--1537, 2005.

\bibitem[Belkin and Niyogi(2003)]{belkin03}
M.~Belkin and P.~Niyogi.
\newblock Laplacian eigenmaps for dimensionality reduction and data
  representation.
\newblock \emph{Neural Computation}, 15\penalty0 (6):\penalty0 1373--1396,
  2003.

\bibitem[Bengio et~al.(2013)Bengio, Courville, and
  Vincent]{bengio2013representation}
Y.~Bengio, A.~Courville, and P.~Vincent.
\newblock Representation learning: A review and new perspectives.
\newblock \emph{IEEE transactions on pattern analysis and machine
  intelligence}, 35\penalty0 (8):\penalty0 1798--1828, 2013.

\bibitem[Benton et~al.(2024)Benton, Bortoli, Doucet, and
  Deligiannidis]{benton2024nearly}
J.~Benton, V.~D. Bortoli, A.~Doucet, and G.~Deligiannidis.
\newblock Nearly $d$-linear convergence bounds for diffusion models via
  stochastic localization.
\newblock In \emph{The Twelfth International Conference on Learning
  Representations}, 2024.

\bibitem[Beyler and Bach(2025)]{beyler25conv}
E.~Beyler and F.~Bach.
\newblock Convergence of deterministic and stochastic diffusion-model samplers:
  A simple analysis in {W}asserstein distance.
\newblock Preprint. ArXiv:2508.03210, 2025.

\bibitem[Bongioanni and Torrea(2006)]{bongioanni2006sobolev}
B.~Bongioanni and J.~L. Torrea.
\newblock Sobolev spaces associated to the harmonic oscillator.
\newblock \emph{Proceedings of the Indian Academy of Sciences - Mathematical
  Sciences}, 116:\penalty0 337--360, 2006.

\bibitem[Bortoli(2022)]{de_bortoli2022convergence}
V.~D. Bortoli.
\newblock Convergence of denoising diffusion models under the manifold
  hypothesis.
\newblock \emph{Transactions on Machine Learning Research}, 2022.
\newblock Expert Certification.

\bibitem[Brown et~al.(2023)Brown, Caterini, Ross, Cresswell, and
  Loaiza-Ganem]{brown2023verifying}
B.~C. Brown, A.~L. Caterini, B.~L. Ross, J.~C. Cresswell, and G.~Loaiza-Ganem.
\newblock Verifying the union of manifolds hypothesis for image data.
\newblock In \emph{The Eleventh International Conference on Learning
  Representations}, 2023.

\bibitem[Chakraborty and Bartlett(2025)]{chakraborty2025statistical}
S.~Chakraborty and P.~L. Bartlett.
\newblock On the statistical properties of generative adversarial models for
  low intrinsic data dimension.
\newblock \emph{Journal of Machine Learning Research}, 26\penalty0
  (111):\penalty0 1--57, 2025.

\bibitem[Chen et~al.(2023{\natexlab{a}})Chen, Lee, and Lu]{chen2023improved}
H.~Chen, H.~Lee, and J.~Lu.
\newblock Improved analysis of score-based generative modeling: User-friendly
  bounds under minimal smoothness assumptions.
\newblock In A.~Krause, E.~Brunskill, K.~Cho, B.~Engelhardt, S.~Sabato, and
  J.~Scarlett, editors, \emph{Proceedings of the 40th International Conference
  on Machine Learning}, volume 202 of \emph{Proceedings of Machine Learning
  Research}, pages 4735--4763. PMLR, 2023{\natexlab{a}}.

\bibitem[Chen et~al.(2023{\natexlab{b}})Chen, Huang, Zhao, and
  Wang]{chen2023score}
M.~Chen, K.~Huang, T.~Zhao, and M.~Wang.
\newblock Score approximation, estimation and distribution recovery of
  diffusion models on low-dimensional data.
\newblock In A.~Krause, E.~Brunskill, K.~Cho, B.~Engelhardt, S.~Sabato, and
  J.~Scarlett, editors, \emph{Proceedings of the 40th International Conference
  on Machine Learning}, volume 202 of \emph{Proceedings of Machine Learning
  Research}, pages 4672--4712. PMLR, 2023{\natexlab{b}}.

\bibitem[Chen et~al.(2023{\natexlab{c}})Chen, Chewi, Li, Li, Salim, and
  Zhang]{chen2023sampling}
S.~Chen, S.~Chewi, J.~Li, Y.~Li, A.~Salim, and A.~Zhang.
\newblock Sampling is as easy as learning the score: theory for diffusion
  models with minimal data assumptions.
\newblock In \emph{The Eleventh International Conference on Learning
  Representations}, 2023{\natexlab{c}}.

\bibitem[Cheng et~al.(2018)Cheng, Chatterji, Bartlett, and
  Jordan]{cheng2017underdamped}
X.~Cheng, N.~S. Chatterji, P.~L. Bartlett, and M.~I. Jordan.
\newblock Underdamped {L}angevin {MCMC}: A non-asymptotic analysis.
\newblock In \emph{Proceedings of the 31st Conference On Learning Theory},
  volume~75 of \emph{Proceedings of Machine Learning Research}, pages 300--323.
  PMLR, 2018.

\bibitem[Coifman and Lafon(2006)]{coifman06}
R.~R. Coifman and S.~Lafon.
\newblock Diffusion maps.
\newblock \emph{Applied and Computational Harmonic Analysis. Time-Frequency and
  Time-Scale Analysis, Wavelets, Numerical Algorithms, and Applications},
  21\penalty0 (1):\penalty0 5--30, 2006.

\bibitem[Genovese et~al.(2014)Genovese, Perone-Pacifico, Verdinelli, and
  Wasserman]{genovese2014nonparametric}
C.~R. Genovese, M.~Perone-Pacifico, I.~Verdinelli, and L.~Wasserman.
\newblock Nonparametric ridge estimation.
\newblock \emph{The Annals of Statistics}, 42\penalty0 (4):\penalty0
  1511--1545, 2014.

\bibitem[Ghosh et~al.(2025)Ghosh, Ignatiadis, Koehler, and Lee]{ghosh2025stein}
S.~Ghosh, N.~Ignatiadis, F.~Koehler, and A.~Lee.
\newblock {S}tein's unbiased risk estimate and {H}yv\"arinen's score matching.
\newblock Preprint. ArXiv:2502.20123, 2025.

\bibitem[Gulrajani et~al.(2017)Gulrajani, Ahmed, Arjovsky, Dumoulin, and
  Courville]{gulrajani2017improved}
I.~Gulrajani, F.~Ahmed, M.~Arjovsky, V.~Dumoulin, and A.~C. Courville.
\newblock Improved training of {W}asserstein {GAN}s.
\newblock In \emph{Advances in Neural Information Processing Systems},
  volume~30, 2017.

\bibitem[Huang et~al.(2025)Huang, Huang, and Lin]{huang2025fast}
D.~Z. Huang, J.~Huang, and Z.~Lin.
\newblock Fast convergence for high-order {ODE} solvers in diffusion
  probabilistic models.
\newblock Preprint. ArXiv:2506.13061, 2025.

\bibitem[Huang et~al.(2024)Huang, Zou, Dong, Ma, and Zhang]{huang24faster}
X.~Huang, D.~Zou, H.~Dong, Y.-A. Ma, and T.~Zhang.
\newblock Faster sampling without isoperimetry via diffusion-based {M}onte
  {C}arlo.
\newblock In \emph{Proceedings of Thirty Seventh Conference on Learning
  Theory}, volume 247 of \emph{Proceedings of Machine Learning Research}, pages
  2438--2493. PMLR, 2024.

\bibitem[Hyv{{\"a}}rinen(2005)]{hyvarinen2005estimation}
A.~Hyv{{\"a}}rinen.
\newblock Estimation of non-normalized statistical models by score matching.
\newblock \emph{Journal of Machine Learning Research}, 6\penalty0
  (24):\penalty0 695--709, 2005.

\bibitem[Hyv{\"a}rinen(2007)]{hyvarinen2007extensions}
A.~Hyv{\"a}rinen.
\newblock Some extensions of score matching.
\newblock \emph{Computational Statistics \& Data Analysis}, 51\penalty0
  (5):\penalty0 2499--2512, 2007.

\bibitem[Koehler et~al.(2023)Koehler, Heckett, and
  Risteski]{koehler2023statistical}
F.~Koehler, A.~Heckett, and A.~Risteski.
\newblock Statistical efficiency of score matching: The view from isoperimetry.
\newblock In \emph{The Eleventh International Conference on Learning
  Representations}, 2023.

\bibitem[Ledoux and Talagrand(2013)]{ledoux2013probability}
M.~Ledoux and M.~Talagrand.
\newblock \emph{Probability in Banach Spaces: isoperimetry and processes}.
\newblock Springer Science \& Business Media, 2013.

\bibitem[Li and Yan(2024)]{li2024adapting}
G.~Li and Y.~Yan.
\newblock Adapting to unknown low-dimensional structures in score-based
  diffusion models.
\newblock In \emph{Advances in Neural Information Processing Systems},
  volume~37, pages 126297--126331, 2024.

\bibitem[Li et~al.(2024{\natexlab{a}})Li, Huang, Efimov, Wei, Chi, and
  Chen]{li2024accelerating}
G.~Li, Y.~Huang, T.~Efimov, Y.~Wei, Y.~Chi, and Y.~Chen.
\newblock Accelerating convergence of score-based diffusion models, provably.
\newblock In \emph{Proceedings of the 41st International Conference on Machine
  Learning}, volume 235 of \emph{Proceedings of Machine Learning Research},
  pages 27942--27954. PMLR, 2024{\natexlab{a}}.

\bibitem[Li et~al.(2024{\natexlab{b}})Li, Wei, Chen, and Chi]{li2024towards}
G.~Li, Y.~Wei, Y.~Chen, and Y.~Chi.
\newblock Towards non-asymptotic convergence for diffusion-based generative
  models.
\newblock In \emph{The Twelfth International Conference on Learning
  Representations}, 2024{\natexlab{b}}.

\bibitem[Li et~al.(2024{\natexlab{c}})Li, Wei, Chi, and Chen]{li2024sharp}
G.~Li, Y.~Wei, Y.~Chi, and Y.~Chen.
\newblock A sharp convergence theory for the probability flow {ODE}s of
  diffusion models.
\newblock \emph{arXiv preprint arXiv:2408.02320}, 2024{\natexlab{c}}.

\bibitem[Li et~al.(2025)Li, Zhou, Wei, and Chen]{li2025faster}
G.~Li, Y.~Zhou, Y.~Wei, and Y.~Chen.
\newblock Faster diffusion models via higher-order approximation.
\newblock \emph{arXiv preprint arXiv:2506.24042}, 2025.

\bibitem[Li et~al.(2005)Li, Gupta, and Liese]{li05}
J.~Li, S.~S. Gupta, and F.~Liese.
\newblock Convergence rates of empirical {B}ayes estimation in exponential
  family.
\newblock \emph{Journal of Statistical Planning and Inference}, 131\penalty0
  (1):\penalty0 101--115, 2005.

\bibitem[Li and Turner(2018)]{li2018gradient}
Y.~Li and R.~E. Turner.
\newblock Gradient estimators for implicit models.
\newblock In \emph{International Conference on Learning Representations}, 2018.

\bibitem[Liang et~al.(2015)Liang, Rakhlin, and Sridharan]{liang15}
T.~Liang, A.~Rakhlin, and K.~Sridharan.
\newblock Learning with square loss: Localization through offset {R}ademacher
  complexity.
\newblock In \emph{Proceedings of The 28th Conference on Learning Theory},
  volume~40, pages 1260--1285, 2015.

\bibitem[Neal et~al.(2011)]{neal2011mcmc}
R.~M. Neal et~al.
\newblock {MCMC} using hamiltonian dynamics.
\newblock \emph{Handbook of markov chain monte carlo}, 2\penalty0
  (11):\penalty0 2, 2011.

\bibitem[Nirenberg(1959)]{nirenberg1959elliptic}
L.~Nirenberg.
\newblock On elliptic partial differential equations.
\newblock \emph{Annali della Scuola Normale Superiore di Pisa-Scienze Fisiche e
  Matematiche}, 13\penalty0 (2):\penalty0 115--162, 1959.

\bibitem[Oko et~al.(2023)Oko, Akiyama, and Suzuki]{oko2023diffusion}
K.~Oko, S.~Akiyama, and T.~Suzuki.
\newblock Diffusion models are minimax optimal distribution estimators.
\newblock In \emph{Proceedings of the 40th International Conference on Machine
  Learning}, volume 202 of \emph{Proceedings of Machine Learning Research},
  pages 26517--26582. PMLR, 2023.

\bibitem[Pauli et~al.(2021)Pauli, Koch, Berberich, Kohler, and
  Allg{\"o}wer]{pauli2021training}
P.~Pauli, A.~Koch, J.~Berberich, P.~Kohler, and F.~Allg{\"o}wer.
\newblock Training robust neural networks using {L}ipschitz bounds.
\newblock \emph{IEEE Control Systems Letters}, 6:\penalty0 121--126, 2021.

\bibitem[Pope et~al.(2021)Pope, Zhu, Abdelkader, Goldblum, and
  Goldstein]{pope2021the}
P.~Pope, C.~Zhu, A.~Abdelkader, M.~Goldblum, and T.~Goldstein.
\newblock The intrinsic dimension of images and its impact on learning.
\newblock In \emph{International Conference on Learning Representations}, 2021.

\bibitem[Puchkin and Zhivotovskiy(2023)]{puchkin2023exploring}
N.~Puchkin and N.~Zhivotovskiy.
\newblock Exploring local norms in exp-concave statistical learning.
\newblock In \emph{The Thirty Sixth Annual Conference on Learning Theory},
  pages 1993--2013. PMLR, 2023.

\bibitem[Puchkin et~al.(2025)Puchkin, Suchkov, Naumov, and
  Belomestny]{puchkin25}
N.~Puchkin, D.~Suchkov, A.~Naumov, and D.~Belomestny.
\newblock Tight bounds for {S}chr\"odinger potential estimation in unpaired
  image-to-image translation problems.
\newblock Preprint. ArXiv:2508.07392, 2025.

\bibitem[Qi et~al.(2023)Qi, Wang, Chen, Shi, and Zhang]{qi2023lipsformer}
X.~Qi, J.~Wang, Y.~Chen, Y.~Shi, and L.~Zhang.
\newblock Lips{F}ormer: Introducing {L}ipschitz continuity to vision
  transformers.
\newblock In \emph{The Eleventh International Conference on Learning
  Representations}, 2023.

\bibitem[Sasaki et~al.(2018)Sasaki, Kanamori, Hyv{{\"a}}rinen, Niu, and
  Sugiyama]{sasaki18mode}
H.~Sasaki, T.~Kanamori, A.~Hyv{{\"a}}rinen, G.~Niu, and M.~Sugiyama.
\newblock Mode-seeking clustering and density ridge estimation via direct
  estimation of density-derivative-ratios.
\newblock \emph{Journal of Machine Learning Research}, 18\penalty0
  (180):\penalty0 1--47, 2018.

\bibitem[Schreuder et~al.(2021)Schreuder, Brunel, and
  Dalalyan]{schreuder2021statistical}
N.~Schreuder, V.-E. Brunel, and A.~Dalalyan.
\newblock Statistical guarantees for generative models without domination.
\newblock In \emph{Algorithmic Learning Theory}, pages 1051--1071. PMLR, 2021.

\bibitem[Shen et~al.(2024)Shen, Jiao, Lin, and Huang]{shen2024differentiable}
G.~Shen, Y.~Jiao, Y.~Lin, and J.~Huang.
\newblock Differentiable neural networks with {R}e{PU} activation: with
  applications to score estimation and isotonic regression, 2024.

\bibitem[Shi et~al.(2018)Shi, Sun, and Zhu]{shi2018spectral}
J.~Shi, S.~Sun, and J.~Zhu.
\newblock A spectral approach to gradient estimation for implicit
  distributions.
\newblock In \emph{International Conference on Machine Learning}, pages
  4644--4653. PMLR, 2018.

\bibitem[Song and Ermon(2019)]{song2019generative}
Y.~Song and S.~Ermon.
\newblock Generative modeling by estimating gradients of the data distribution.
\newblock \emph{Advances in neural information processing systems}, 32, 2019.

\bibitem[Song et~al.(2020)Song, Garg, Shi, and Ermon]{song20a}
Y.~Song, S.~Garg, J.~Shi, and S.~Ermon.
\newblock Sliced score matching: A scalable approach to density and score
  estimation.
\newblock In R.~P. Adams and V.~Gogate, editors, \emph{Proceedings of The 35th
  Uncertainty in Artificial Intelligence Conference}, volume 115 of
  \emph{Proceedings of Machine Learning Research}, pages 574--584. PMLR, 22--25
  Jul 2020.

\bibitem[Song et~al.(2021)Song, Sohl-Dickstein, Kingma, Kumar, Ermon, and
  Poole]{song2021scorebased}
Y.~Song, J.~Sohl-Dickstein, D.~P. Kingma, A.~Kumar, S.~Ermon, and B.~Poole.
\newblock Score-based generative modeling through stochastic differential
  equations.
\newblock In \emph{International Conference on Learning Representations}, 2021.

\bibitem[Srebro et~al.(2010)Srebro, Sridharan, and
  Tewari]{srebro2010smoothness}
N.~Srebro, K.~Sridharan, and A.~Tewari.
\newblock Smoothness, low noise and fast rates.
\newblock \emph{Advances in neural information processing systems}, 23, 2010.

\bibitem[Sriperumbudur et~al.(2017)Sriperumbudur, Fukumizu, Gretton,
  Hyv\"{a}rinen, and Kumar]{sriperumbudur17}
B.~Sriperumbudur, K.~Fukumizu, A.~Gretton, A.~Hyv\"{a}rinen, and R.~Kumar.
\newblock Density estimation in infinite dimensional exponential families.
\newblock \emph{Journal of Machine Learning Research}, 18\penalty0
  (57):\penalty0 1--59, 2017.

\bibitem[Stempak and Torrea(2003)]{stempak2003poisson}
K.~Stempak and J.~L. Torrea.
\newblock Poisson integrals and {R}iesz transforms for hermite function
  expansions with weights.
\newblock \emph{Journal of functional analysis}, 202\penalty0 (2):\penalty0
  443--472, 2003.

\bibitem[St{\'e}phanovitch et~al.(2024)St{\'e}phanovitch, Aamari, and
  Levrard]{stephanovitch2024wasserstein}
A.~St{\'e}phanovitch, E.~Aamari, and C.~Levrard.
\newblock Wasserstein generative adversarial networks are minimax optimal
  distribution estimators.
\newblock \emph{The Annals of Statistics}, 52\penalty0 (5):\penalty0
  2167--2193, 2024.

\bibitem[Sutherland et~al.(2018)Sutherland, Strathmann, Arbel, and
  Gretton]{sutherland2018efficient}
D.~J. Sutherland, H.~Strathmann, M.~Arbel, and A.~Gretton.
\newblock Efficient and principled score estimation with {N}ystr{\"o}m kernel
  exponential families.
\newblock In \emph{International Conference on Artificial Intelligence and
  Statistics}, pages 652--660. PMLR, 2018.

\bibitem[Szeg\"{o}(1975)]{szego1975orthogonal}
G.~Szeg\"{o}.
\newblock \emph{Orthogonal polynomials}, volume Vol. XXIII of \emph{American
  Mathematical Society Colloquium Publications}.
\newblock American Mathematical Society, Providence, RI, fourth edition, 1975.

\bibitem[Tang and Yang(2024)]{tang2024adaptivity}
R.~Tang and Y.~Yang.
\newblock Adaptivity of diffusion models to manifold structures.
\newblock In \emph{Proceedings of The 27th International Conference on
  Artificial Intelligence and Statistics}, volume 238 of \emph{Proceedings of
  Machine Learning Research}, pages 1648--1656. PMLR, 2024.

\bibitem[Vershynin(2018)]{vershynin2018high}
R.~Vershynin.
\newblock \emph{High-dimensional probability: An introduction with applications
  in data science}, volume~47.
\newblock Cambridge university press, 2018.

\bibitem[Vincent(2011)]{vincent2011connection}
P.~Vincent.
\newblock A connection between score matching and denoising autoencoders.
\newblock \emph{Neural computation}, 23\penalty0 (7):\penalty0 1661--1674,
  2011.

\bibitem[Wibisono et~al.(2024)Wibisono, Wu, and Yang]{wibisono2024optimal}
A.~Wibisono, Y.~Wu, and K.~Y. Yang.
\newblock Optimal score estimation via empirical {B}ayes smoothing.
\newblock In \emph{Proceedings of Thirty Seventh Conference on Learning
  Theory}, volume 247 of \emph{Proceedings of Machine Learning Research}, pages
  4958--4991. PMLR, 2024.

\bibitem[Yakovlev and Puchkin(2025{\natexlab{a}})]{yakovlev2025generalization}
K.~Yakovlev and N.~Puchkin.
\newblock Generalization error bound for denoising score matching under relaxed
  manifold assumption.
\newblock In \emph{Proceedings of Thirty Eighth Conference on Learning Theory},
  volume 291 of \emph{Proceedings of Machine Learning Research}, pages
  5824--5891. PMLR, 2025{\natexlab{a}}.

\bibitem[Yakovlev and Puchkin(2025{\natexlab{b}})]{yakovlev2025simultaneous}
K.~Yakovlev and N.~Puchkin.
\newblock Simultaneous approximation of the score function and its derivatives
  by deep neural networks.
\newblock Preprint. ArXiv:2512.23643, 2025{\natexlab{b}}.

\bibitem[Yoshida and Miyato(2017)]{yoshida2017spectral}
Y.~Yoshida and T.~Miyato.
\newblock Spectral norm regularization for improving the generalizability of
  deep learning.
\newblock Preprint. ArXiv:1705.10941, 2017.

\bibitem[Zhang et~al.(2024)Zhang, Yin, Liang, and Liu]{zhang2024minimax}
K.~Zhang, H.~Yin, F.~Liang, and J.~Liu.
\newblock Minimax optimality of score-based diffusion models: Beyond the
  density lower bound assumptions.
\newblock In \emph{Proceedings of the 41st International Conference on Machine
  Learning}, volume 235 of \emph{Proceedings of Machine Learning Research},
  pages 60134--60178. PMLR, 2024.

\bibitem[Zhang et~al.(2025)Zhang, Huan, Huang, Boffi, Chen, and
  Chewi]{zhang2025sublinear}
M.~S. Zhang, S.~Huan, J.~Huang, N.~M. Boffi, S.~Chen, and S.~Chewi.
\newblock Sublinear iterations can suffice even for {DDPM}s.
\newblock Preprint. ArXiv:2511.04844, 2025.

\bibitem[Zhou et~al.(2020)Zhou, Shi, and Zhu]{zhou2020nonparametric}
Y.~Zhou, J.~Shi, and J.~Zhu.
\newblock Nonparametric score estimators.
\newblock In \emph{Proceedings of the 37th International Conference on Machine
  Learning}, volume 119 of \emph{Proceedings of Machine Learning Research},
  pages 11513--11522. PMLR, 2020.

\end{thebibliography}

\appendix

\section{Proof of Lemma \ref{lem:bernstein_cond_div}}
\label{sec:lem_bernstein_cond_div_proof}

Let $f : (z, u) \mapsto h\big(\sigma z + g^*(u)\big)$ for $z \in \R^D$ and $u \in [0, 1]^d$.
Next, we note that for all $\bk\in\Z_+^D$ with $0 \leq |\bk| \leq \alpha$ it holds that
\begin{align*}
    \integral{[0,1]^d} \left\|\partial^\bk f(\cdot, u)\right\|^2_{L^2(\phi_D)} \dd u =  \sigma^{2|\bk|}  \left\|\partial^\bk h\right\|^2_{L^2(\sfp^*)} .
\end{align*}
Thus, we conclude that
\begin{equation}
\label{eq:norms_relations}
\begin{split}
    &
    \integral{[0,1]^d}\|f(\cdot, u)\|^{2}_{L^2(\phi_D)} \dd u
    = \|h\|^{2}_{L^2(\sfp^*)},
    \quad 
    \integral{[0,1]^d} \|\rdiv[f](\cdot, u)\|^{2}_{L^2( \phi_D)} \dd u = \sigma^2\|\rdiv[h]]\|^{2}_{L^2(\sfp^*)},
    \\&
    \integral{[0,1]^d} |f(\cdot, u)|^{2}_{W^{\alpha, 2}(\phi_D)} \dd u = \sigma^{2\alpha} |h|^{2}_{W^{\alpha, 2}(\sfp^*)},
    \quad
    \integral{[0, 1]^d} \big\| \|\nabla f(\cdot, u)\|_F \big\|_{L^2(\phi_D)}^2 \dd u = \sigma^2 \big\| \|\nabla h\|_F \big\|_{L^2(\sfp^*)}^2 .
\end{split}
\end{equation}
We now reduce the problem to establishing the desired property for the standard Gaussian weight function.

\begin{Lem}
\label{lem:bernstein_cond_helper}
    For an arbitrary $f: \R^D \to \R^D$ such that $|f|_{W^{\alpha, 2}\left(\phi_D\right)} < \infty$  for some integer $\alpha \geq 2$, it holds that
    \begin{align*}
        (i)
        &\quad
        \|\rdiv[f]\|_{L^2(\phi_D)}^2
        \leq \alpha D \left(\|f\|_{L^2(\phi_D)}^2 + |f|_{W^{\alpha, 2}(\phi_D)}^2\right)^{1 / \alpha} \|f\|_{L^2(\phi_D)}^{2(\alpha - 1) / \alpha},
        \\
        (ii)
        &\quad
        \big\| \|\nabla f\|_F \big\|_{L^2(\phi_D)}^2
        \leq \alpha D^{1 - 1/\alpha} \left( D\|f\|_{L^2(\phi_D)}^2 + |f|_{W^{\alpha, 2}(\phi_D)}^2 \right)^{1 / \alpha} \|f\|_{L^2(\phi_D)}^{2(\alpha - 1) / \alpha} .
    \end{align*}
\end{Lem}
The proof of Lemma \ref{lem:bernstein_cond_helper} can be found in Appendix \ref{sec:lem_bernstein_cond_helper_proof}.
By Lemma \ref{lem:bernstein_cond_helper} and Hölder's inequality, we obtain
\begin{align*}
    \integral{[0,1]^d} \|\rdiv[f](\cdot, u)\|^2_{L^2(\phi_D)} \dd u
    &
    \leq \alpha D \integral{[0, 1]^d} \left(\|f(\cdot, u)\|_{L^2(\phi_D)}^2 + |f(\cdot, u)|_{W^{\alpha, 2}(\phi_D)}^2 \right)^{1 / \alpha} \|f(\cdot, u)\|_{L^2(\phi_D)}^{2(1 - \alpha) / \alpha} \dd u .
\end{align*}
Now the Hölder inequality yields
\begin{align*}
    &
    \integral{[0,1]^d} \|\rdiv[f](\cdot, u)\|^2_{L^2(\phi_D)} \dd u
    \\&
    \leq \alpha D \left( \; \integral{[0,1]^d} \left( \|f(\cdot, u)\|_{L^2(\phi_D)}^2 + |f(\cdot, u)|^{2}_{W^{\alpha, 2}(\phi_D)} \right) \dd u \right)^{1 / \alpha} \left( \; \integral{[0,1]^d} \|f(\cdot, u)\|^{2}_{L^2(\phi_D)} \dd u \right)^{1 - 1 / \alpha} .
\end{align*}
Combining this with \eqref{eq:norms_relations}, we deduce that
\begin{align*}
    \|\rdiv[h]\|_{L^2(\sfp^*)}^2 \leq \alpha D \sigma^{-2} \left( \|h\|_{L^2(\sfp^*)}^2 + \sigma^{2\alpha}|h|_{W^{\alpha, 2}(\sfp^*)}^2 \right)^{1 / \alpha} \|h\|_{L^2(\sfp^*)}^{2 - 2 / \alpha}.
\end{align*}
Hence, statement $(i)$ is established.
As for the second claim, Lemma \ref{lem:bernstein_cond_helper} suggests that
\begin{align*}
    &
    \integral{[0, 1]^d} \big\| \|\nabla f(\cdot, u)\|_F \big\|_{L^2(\phi_D)}^2 \dd u
    \\&
    \leq \alpha D^{1 - 1 / \alpha}\integral{[0, 1]^d} \left( D\|f(\cdot, u)\|_{L^2(\phi_D)}^2 + |f(\cdot, u)|_{W^{\alpha, 2}(\phi_D)}^2 \right)^{1 / \alpha} \|f(\cdot, u)\|_{L^2(\phi_D)}^{2(\alpha - 1) / \alpha} \dd u
    \\&
    \leq \alpha D^{1 - 1 / \alpha} \left( \, \integral{[0, 1]^d}\left( D\|f(\cdot, u)\|_{L^2(\phi_D)}^2 + |f(\cdot, u)|_{W^{\alpha, 2}(\phi_D)}^2 \right) \dd u \right)^{1 / \alpha}
    \left( \, \integral{[0, 1]^d} \|f(\cdot, u)\|_{L^2(\phi_D)}^2 \dd u \right)^{1 - 1 / \alpha} ,
\end{align*}
where the last inequality follows from the H\"older inequality.
In view of \eqref{eq:norms_relations}, we have that
\begin{align*}
    \big\| \|\nabla h\|_F \big\|_{L^2(\sfp^*)}^2
    \leq \alpha D^{1 - 1 / \alpha}\sigma^{-2} \left( D \|h\|_{L^2(\sfp^*)}^2 + \sigma^{2\alpha} |h|_{W^{\alpha, 2}(\sfp^*)}^2 \right)^{1 / \alpha} \|h\|_{L^2(\sfp^*)}^{2 - 2 / \alpha} .
\end{align*}
Thus, the proof is finished.

\myendproof

\subsection{Proof of Lemma \ref{lem:bernstein_cond_helper}}
\label{sec:lem_bernstein_cond_helper_proof}

The proof relies on the expansion of $f$ using Hermite polynomials, enabling control of the divergence norm via the function norm.
We split the proof into two parts, each corresponding to one of the claims.

\noindent
\textbf{Step 1: proof of statement $(i)$.}\quad
Since Hermite polynomials form an orthogonal basis in Sobolev space $W^{\alpha, 2}\left(\R^D, \phi_D\right)$, $f$ can be rewritten in terms of a convergent series
\begin{align*}
    f(z) = \sum\limits_{\bk\in\Z_+^D}a_{\bk}\cH_{\bk}(z), \quad z \in \R^D,
\end{align*}
where $a_\bk = (a_\bk^1, a_\bk^2, \dots, a_\bk^D)$ for every $\bk \in \Z_+^D$.
Hence, we have that
\begin{align}
    \label{eq:f_L2_norm_hermite}
    \|f\|^2_{L^2(\phi_D)}
    = \sum_{\bk \in \Z_+^D}\|a_\bk\|^2\langle \cH_\bk, \cH_\bk\rangle_{\phi_D}
    = \sum_{\bk \in \Z_+^D}\|a_\bk\|^2\bk! 
\end{align}
and, similarly, for any $\bp \in \Z_+^D$ with $|\bp| \geq 1$ and $1 \leq i \leq D$,
\begin{align}
    \label{eq:D_p_f_aux}
    \|\partial^\bp f_i\|_{L^2(\phi_D)}^2
    = \left\|\sum_{\substack{\bk\in \Z_+^D \\ \bk \geq \bp}} a_\bk^i \frac{\bk!}{(\bk - \bp)!}\cH_{\bk - \bp}\right\|_{L^2(\phi_D)}^2
    = \sum_{\substack{\bk \in \Z_+^D \\ \bk \geq \bp}} (a_\bk^i)^2\frac{(\bk!)^2}{(\bk - \bp)!},
\end{align}
where we used the identity that $\partial^{\bp}\cH_\bk = (\bk! / (\bk - \bp)!) \cdot\cH_{\bk - \bp}$ for all $\bk \geq \bp$.
Therefore, the weighted Sobolev seminorm of $f$ (see Definition \ref{def:sobolev_space}) is given by
\begin{align}
    \label{eq:weighted_sob_norm_hermite}
    |f|_{W^{\alpha, 2}(\phi_D)}^2 
    = \sum_{\substack{\bk \in \Z_+^D \\ |\bk| = \alpha}} \sum_{i = 1}^D\|\partial^\bk f_i\|_{L^2(\phi_D)}^2
    = \sum_{\substack{\bp \in \Z_+^D \\ |\bp| = \alpha}} \sum_{\substack{\bk \in \Z_+^D \\ \bk \geq \bp}} \|a_\bk\|^2 \frac{(\bk)!}{(\bk - \bp)!} .
\end{align}
Now Cauchy-Schwartz inequality implies that
\begin{align*}
    \|\rdiv[f]\|^2_{L^2(\phi_D)}
    = \left\|\sum_{i = 1}^D \partial^{\be_i}f_i\right\|_{L_2(\phi_D)}^2
    \leq \left(\sum_{i=1}^D\|\partial^{\be_i}f_i\|_{L^2(\phi_D)}\right)^2
    \leq D \sum_{i = 1}^D\|\partial^{\be_i}f_i\|_{L^2(\phi_D)}^2 ,
\end{align*}
where $\be_i$ denotes the $i$-th unit basis vector in $\R^D$.
Therefore, \eqref{eq:D_p_f_aux} yields
\begin{align*}
    \|\rdiv[f]\|^2_{L^2(\phi_D)}
    &\leq D \sum_{i=1}^D\sum\limits_{\bk\in\Z_+^D}\left(a^i_{\bk+\be_i}\right)^2 (k_i + 1)(\bk+\be_i)! .
\end{align*}
By Hölder's inequality with $p=\alpha$ and $q=\alpha/(\alpha-1)$, we have
\begin{align}
    \label{eq:berns_div_holder}
    \notag
    \|\rdiv[f]\|_{L^2(\phi_D)}^2
    &
    \leq D\left(\sum\limits_{i=1}^D \sum\limits_{\bk\in\Z_+^D}\left(a^i_{\bk+\be_i}\right)^2 (\bk+\be_i)!(k_i + 1)^{\alpha}\right)^{1/\alpha}
    \\&\quad
    \cdot \left(\sum\limits_{i=1}^D \sum\limits_{\bk\in\Z_+^D}\left(a^i_{\bk+\be_i}\right)^2 (\bk+\be_i)!\right)^{1-1/\alpha}.
\end{align}
Evidently, the second term on the right-hand-side is bounded by
\begin{align}
    \label{eq:berns_herm_second_term}
    \sum_{i=1}^D \sum_{\bk\in\Z_+^D}(a^i_{\bk+\be_i})^2 (\bk+\be_i)!
    \leq \sum_{\bk \in \Z_+^D} \|a_\bk\|^2 \bk !
    = \|f\|_{L^2(\phi_D)}^2 ,
\end{align}
as suggested by \eqref{eq:f_L2_norm_hermite}.
We now turn to the evaluation of the first term.
Shifting the multi-index and partitioning the series leads to
\begin{align*}
    \sum_{i=1}^D \sum_{\bk\in\Z_+^D} (a^i_{\bk+\be_i})^2 (\bk + \be_i)!(k_i + 1)^{\alpha}
    &= \sum_{i=1}^D\sum_{\substack{\bk\in\Z_+^D\\1 \leq k_i \leq \alpha}}(a^i_{\bk})^2 \bk!k_i^{\alpha} + \sum\limits_{i=1}^D \sum\limits_{\substack{\bk\in\Z_+^D\\ k_i > \alpha}} (a^i_{\bk})^2 \bk!k_i^{\alpha} \\
    &\leq \alpha^\alpha \sum_{\bk \in \Z_+^D} \|a_\bk\|^2 \bk ! + \alpha^\alpha \sum_{i = 1}^D \sum_{\substack{\bk \in \Z_+^D \\ k_i > \alpha}} (a_\bk^i)^2 \bk! \frac{k_i!}{(k_i - \alpha)!},
\end{align*}
where the last inequality uses the observation that if $k_i \geq \alpha$, then $k_i^\alpha(k_i - \alpha)! \leq \alpha^\alpha k_i!$.
In view of \eqref{eq:f_L2_norm_hermite} and \eqref{eq:weighted_sob_norm_hermite}, we deduce that
\begin{align*}
    \sum_{i=1}^D \sum_{\bk\in\Z_+^D} (a^i_{\bk+\be_i})^2 (\bk + \be_i)!(k_i + 1)^{\alpha}
    &\leq \alpha^\alpha \sum_{\bk \in \Z_+^D} \|a_\bk\|^2 \bk ! + \alpha^\alpha \sum_{\substack{\bp \in \Z_+^D \\ |\bp| = \alpha}} \sum_{\substack{\bk \in \Z_+^D \\ \bk \geq \bp}} \|a_\bk\|^2 \frac{(\bk!)^2}{(\bk - \bp)!} \\
    &\leq \alpha^\alpha\left(\|f\|_{L^2(\phi_D)}^2 + |f|_{W^{\alpha, 2}(\phi_D)}^2\right).
\end{align*}
Combining this with \eqref{eq:berns_div_holder} and \eqref{eq:berns_herm_second_term}, we conclude that
\begin{align*}
    \|\rdiv[f]\|_{L^2(\phi_D)}^2
    \leq \alpha D \left(\|f\|_{L^2(\phi_D)}^2 + |f|_{W^{\alpha, 2}(\phi_D)}^2\right)^{1 / \alpha} \|f\|_{L^2(\phi_D)}^{2(\alpha - 1) / \alpha} .
\end{align*}
The proof of claim $(i)$ is complete.

\noindent
\textbf{Step 2: proof of statement $(ii)$.}\quad
The proof follows almost identically to that of $(i)$.
First, we note that
\begin{align*}
    \big\| \|\nabla f\|_F \big\|_{L^2(\phi_D)}^2
    = \sum_{i = 1}^D \sum_{j = 1}^D \|\partial^{\be_j}f_i\|_{L^2(\phi_D)}^2
    = \sum_{j = 1}^D \sum_{\bk \in \Z_+^D} \|a_{\bk + \be_j}\|^2(k_j + 1)(\bk + \be_j)! .
\end{align*}
Next, applying the Hölder inequality with parameters $p = \alpha$ and $q = \alpha / (\alpha - 1)$, we obtain
\begin{align*}
    \big\| \|\nabla f\|_F \big\|_{L^2(\phi_D)}^2
    \leq \left( \sum_{j = 1}^D\sum_{\bk \in \Z_+^D}\|a_{\bk + \be_j}\|^2(\bk + \be_j)! (k_j + 1)^\alpha \right)^{1 / \alpha} \left( \sum_{j = 1}^D\sum_{\bk \in \Z_+^D}\|a_{\bk + \be_j}\|^2(\bk + \be_j)! \right)^{1 - 1 / \alpha} .
\end{align*}
From \eqref{eq:f_L2_norm_hermite} we find that the second term in the right-hand-side of the above bound is bounded as
\begin{align*}
    \sum_{j = 1}^D\sum_{\bk \in \Z_+^D}\|a_{\bk + \be_j}\|^2(\bk + \be_j)!
    \leq D\sum_{\bk \in \Z_+^D} \|a_\bk\|^2\bk!
    = D\|f\|_{L^2(\phi_D)}^2 .
\end{align*}
For the first term, we apply the same reasoning used in the previous step.
Formally,
\begin{align*}
     \sum_{j = 1}^D\sum_{\bk \in \Z_+^D}\|a_{\bk + \be_j}\|^2(\bk + \be_j)! (k_j + 1)^\alpha
     &\leq \alpha^\alpha\sum_{j = 1}^D \sum_{\bk \in \Z_+^D} \|a_\bk\|^2 \bk! + \alpha^\alpha \sum_{\substack{\bp \in \Z_+^D \\ |\bp| = \alpha}}\sum_{\substack{\bk \in \Z_+^D \\ \bk \geq \bp}}\|a_\bk\|^2\frac{(\bk!)^2}{(\bk - \bp)!} \\
     &= \alpha^\alpha\left( D\|f\|_{L^2(\phi_D)}^2 + |f|_{W^{\alpha, 2}(\phi_D)}^2 \right) ,
\end{align*}
where the last equality uses \eqref{eq:f_L2_norm_hermite} and \eqref{eq:weighted_sob_norm_hermite}.
Combining both bounds, we conclude that
\begin{align*}
    \big\| \|\nabla f\|_F \big\|_{L^2(\phi_D)}^2
    \leq \alpha D^{1 - 1/\alpha} \left( D\|f\|_{L^2(\phi_D)}^2 + |f|_{W^{\alpha, 2}(\phi_D)}^2 \right)^{1 / \alpha} \|f\|_{L^2(\phi_D)}^{2(\alpha - 1) / \alpha} .
\end{align*}
Thus, claim $(ii)$ holds, and the proof is complete.

\myendproof

\section{Proofs of supplemental results from Section \ref{sec:estimation_proof}}

\subsection{Proof of Lemma \ref{lem:bernstein_cond}}
\label{sec:lem_bernstein_cond_proof}

We divide the proof into several steps to improve readability.

\noindent
\textbf{Step 1: decomposing the variance.}\quad
We begin by noticing that
\begin{align}
\label{eq:berns_var_decomp}
    \notag
    \Var[\ell(s, Y) - \ell(s^*, Y)]
    &= \Var\left[\rdiv[s - s^*](Y) + \frac12\|s(Y)\|^2 - \frac12\|s^*(Y)\|^2\right] \\
    &\leq 2 \underbrace{ \E \big( \rdiv[s - s^*](Y) \big)^2 }_{(A)} + \frac12\underbrace{\E\left(\|s(Y)\|^2 - \|s^*(Y)\|^2\right)^2}_{(B)} .
\end{align}

\noindent
\textbf{Step 2: Bounding term $(A)$.}\quad
First, we rewrite term $(A)$:
\begin{align*}
    \E \big( \rdiv[s - s^*](Y) \big)^2
    = \|\rdiv[s - s^*]\|^2_{L^2(\sfp^*)}.
\end{align*}
Applying Lemma \ref{lem:bernstein_cond_div} with $\alpha \geq 2$, we obtain
\begin{align}
    \label{eq:div_s_s_star_aux}
    \|\rdiv[s - s^*]\|^2_{L^2(\sfp^*)} \leq \frac{\alpha D}{\sigma^{2}} \left\{ \|s - s^*\|^2_{L^2(\sfp^*)} + \sigma^{2\alpha}|s - s^*|^2_{W^{\alpha, 2}(\sfp^*)} \right\}^{1 / \alpha}\|s - s^*\|^{2 - 2/\alpha}_{L^2(\sfp^*)} .
\end{align}
Next, we evaluate
\begin{align*}
    \|s - s^*\|^2_{L^2(\sfp^*)} \lesssim \E \left\|-aX + af(Y) \right\|^2 + \E \left\| -\frac{Y}{\sigma^{2}} + \frac{f^*(Y)}{\sigma^{2}} \right\|^2,
\end{align*}
where $X \sim \sfp^*$, $s(y) = -a y + af(y)$, and $s^*(y) = (-y + f^*(y)) / \sigma^{2}$ according to \eqref{eq:s_star_f_star_def}.
Therefore, we have
\begin{align*}
    \|s - s^*\|_{L^2(\sfp^*)}^2
    \lesssim \sigma_{\min}^{-4} \, \E \|Y\|^2 + \sigma_{\min}^{-4} \, \E \|f(Y)\|^2 + \sigma^{-4} \, \E\|f^*(Y)\|^2.
\end{align*}
Taking into account that $|f_l|_{W^{0, \infty}(\R^D)} \leq C_0$ for each $1 \leq l \leq D$, and $\E \|Y\|^2 \lesssim 1 + \sigma^2 D$ according to Assumption \ref{asn:relax_man}, we arrive at
\begin{align}
\label{eq:s_s_star_l2_bound}
    \|s - s^*\|^2_{L^2(\sfp^*)}
    \lesssim \sigma_{\min}^{-4}D(C_0^2 \vee 1) .
\end{align}
Now, using the fact that $\alpha \geq 2$, we bound
\begin{align*}
    |s - s^*|^2_{W^{\alpha, 2}(\sfp^*)}
    = \sum_{\substack{\bk \in \Z_+^D \\ |\bk| = \alpha}} \|a \partial^\bk f - \partial^\bk s^* \|^2_{L^2(\sfp^*)}
    \lesssim \sum_{\substack{\bk \in \Z_+^D \\ |\bk| = \alpha}} \left(\sigma_{\min}^{-4}\|\partial^\bk f\|^2_{L^2(\sfp^*)} + \|\partial^\bk s^*\|^2_{L^2(\sfp^*)} \right) .
\end{align*}
Recall that $|f_l|_{W^{\alpha, 2}(\sfp^*)} \leq C_\alpha \sigma_{\min}^{-2\alpha}$ for each $1 \leq l \leq D$.
Moreover, applying Lemma \ref{lem:score_derivatives}, we deduce that
\begin{align*}
    |s - s^*|^2_{W^{\alpha, 2}(\sfp^*)}
    \lesssim \sum_{\substack{\bk \in \Z_+^D \\ |\bk| = \alpha}} \left( C_\alpha^2D\sigma_{\min}^{-4\alpha - 4} + 4^\alpha \alpha! \sigma^{-4\alpha - 4} \right)
    \leq (D + \alpha)^\alpha (C_\alpha^2 \vee 1)\sigma_{\min}^{-4\alpha - 4}(D + 4^\alpha\alpha!) .
\end{align*}
Substituting this and \eqref{eq:s_s_star_l2_bound} into \eqref{eq:div_s_s_star_aux} yields
\begin{align*}
    &\|\rdiv[s - s^*]\|^2_{L^2(\sfp^*)} \\
    & \lesssim \alpha D\sigma^{-2}\left\{ (\sigma_{\min}^{-4}D(C_0^2 \vee 1))^{1 / \alpha} + \sigma^2(D + \alpha)(C_\alpha^{2 / \alpha} \vee 1)\sigma_{\min}^{-4 - 4 / \alpha} (D + 4^\alpha\alpha!)^{1/\alpha} \right\} \|s - s^*\|_{L^2(\sfp^*)}^{2 - 2 / \alpha} .
\end{align*}
Finally, from Stirling's approximation we find that
\begin{align}
\label{eq:berns_A_bound}
    \|\rdiv[s - s^*]\|^2_{L^2(\sfp^*)}
    \lesssim \alpha^2 D^{1 + 1 / \alpha}(D + \alpha)\sigma^{-2}\sigma_{\min}^{-4 - 4 / \alpha}(C_0^{2 / \alpha} \vee C_\alpha^{2 / \alpha} \vee 1)\|s - s^*\|_{L^2(\sfp^*)}^{2 - 2 / \alpha} .
\end{align}
Lemma \ref{lem:bernstein_cond_div} also implies that the above bound holds for the Frobenius norm of the Jacobian matrix $\nabla(s - s^*)$, therefore, verifying claim $(ii)$ of the current lemma.

\noindent
\textbf{Step 3: Bounding term $(B)$.}\quad
To evaluate term $(B)$, the Cauchy-Schwartz inequality together with the Hölder inequality, applied with $p = \alpha / (\alpha - 1)$ and $q = \alpha$ for some $\alpha \in \N$ with $\alpha \geq 2$, yields
\begin{align*}
    (B)
    &\leq \E \left[ \|s(Y) - s^*(Y)\|^2 \, \|s(Y) - s^*(Y)\|^2 \right] \\
    &\leq \left( \E \|s(Y) - s^*(Y)\|^2 \right)^{(\alpha - 1) / \alpha} \left( \E \big[ \|s(Y) - s^*(Y)\|^2 \, \|s(Y) + s^*(Y)\|^{2\alpha} \big] \right)^{1 / \alpha} .
\end{align*}
The second term is bounded by
\begin{align*}
    \left( \E \big[ \|s(Y) - s^*(Y)\|^2 \|s(Y) + s^*(Y)\|^{2\alpha} \big] \right)^{1 / \alpha}
    &\lesssim \left( \E[(\|s(Y)\|^2 + \|s^*(Y)\|^2)^{1 + \alpha}] \right)^{1 / \alpha} \\
    &\lesssim \left( \E \left[ \left( \frac{\|Y\|^2 + D(C_0^2 \vee 1)}{\sigma_{\min}^{4}} \right)^{1 + \alpha} \right] \right)^{1 / \alpha} .
\end{align*}
Next, by Assumption \ref{asn:relax_man}, we have $\|Y\|^2 \leq 2 + 2\sigma^2\|Z\|^2$ almost surely, where $Z \sim \cN(0, I_D)$.
Therefore, we have that
\begin{align*}
    \left( \E \big[ \|s(Y) - s^*(Y)\|^2\|s(Y) + s^*(Y)\|^{2\alpha} \big] \right)^{1 / \alpha}
    &\lesssim \sigma_{\min}^{-4 - 4/\alpha} \left\{ \E \left( \|Z\|^2 + D(C_0^2 \vee 1) \right)^{1 + \alpha} \right\}^{1 / \alpha} \\
    &\leq \sigma_{\min}^{-4 - 4 / \alpha} \left\{ \left(\E \|Z\|^{2 + 2\alpha} \right)^{1 / (1 + \alpha)} + D(C_0^2 \vee 1) \right\}^{1 + 1 / \alpha} ,
\end{align*}
where the last line follows from Minkowski’s inequality.
Next, since $\big\| \|Z\|^2 \big\|_{\psi_1} \lesssim D$, then from \cite[Proposition 2.7.1]{vershynin2018high} we deduce that
\begin{align*}
     \left( \E \left[\|s(Y) - s^*(Y)\|^2 \, \|s(Y) + s^*(Y)\|^{2\alpha} \right] \right)^{1 / \alpha}
     &
     \lesssim \sigma_{\min}^{-4 - 4 / \alpha} \left\{ \alpha D + D(C_0^2 \vee 1) \right\}^{1 + 1 / \alpha} \\
     &
     \lesssim \alpha \left( \frac{D(C_0^2 \vee 1)}{\sigma_{\min}^{4}} \right)^{1 + 1 / \alpha} .
\end{align*}
Consequently, we conclude that
\begin{align}
    \label{eq:berns_B_bound}
    \E\left[\left(\|s(Y)\|^2 - \|s^*(Y)\|^2\right)^2\right]
    \lesssim \alpha \{\sigma_{\min}^{-4}D(C_0^2 \vee 1)\}^{1 + 1 / \alpha} \{\E[\|s(Y) - s^*(Y)\|^2]\}^{1 - 1 / \alpha} .
\end{align}

\noindent
\textbf{Step 4: combining the bound for terms $(A)$ and $(B)$.}\quad
Putting together \eqref{eq:berns_var_decomp}, \eqref{eq:berns_A_bound}, and \eqref{eq:berns_B_bound}, we obtain that
\begin{align*}
    \Var[\ell(s, Y) - \ell(s^*, Y)]
    \lesssim \frac{\alpha^2 D^{1 + 1 / \alpha}(D + \alpha)}{\sigma^{2}\sigma_{\min}^{4 + 4 / \alpha}} \left(C_0^{2 + 2 / \alpha} \vee C_\alpha^{2 / \alpha} \vee 1 \right) \|s - s^*\|_{L^2(\sfp^*)}^{2 - 2 / \alpha} .
\end{align*}
Note that the hidden constant in the above inequality does not depend on $s$.
The proof is now complete.

\myendproof

\subsection{Proof of Lemma \ref{lem:psi1_norm_bound}}
\label{sec:lem_psi1_norm_bound_proof}
First, fix an arbitrary $x \in \R^D$.
Recall that all $s \in \cS(L, W, S, B)$ are in the form of $s(x) = -a x + af(x)$ with $a \leq \sigma_{\min}^{-2}$ (see Definition \ref{def:score_class}).
Therefore, we have that
\begin{align*}
    \sup_{s \in \cS(L, W, S, B)}|\ell(s, x)|
    &\lesssim \sup_{s \in \cS(L, W, S, B)}\|s(x)\|^2 + \sup_{s \in \cS(L, W, S, B)}|\rdiv[s](x)|
    \\&
    \lesssim \sigma_{\min}^{-4}\|x\|^2 + \sigma_{\min}^{-4}\|f(x)\|^2 + \sup_{s \in \cS(L, W, S, B)}|\rdiv[s](x)| .
\end{align*}
We further find from Definition \ref{def:score_class} that $\rdiv[s](x) = -a D + a \cdot \rdiv[f](x)$ and, thus,
\begin{align*}
    \sup_{s \in \cS(L, W, S, B)}|\ell(s, x)| \lesssim \sigma_{\min}^{-4}\|x\|^2 + \sigma_{\min}^{-4}D(C_0^2 \vee C_1 \vee 1) .
\end{align*}
Similarly, from \eqref{eq:s_star_f_star_def} and Lemma \ref{lem:score_derivatives} we find that
\begin{align*}
    |\ell(s^*, x)|
    \lesssim \|s^*(x)\|^2 + |\rdiv[s^*](x)|
    \lesssim \sigma^{-4}\|x\|^2 + \sigma^{-4}D .
\end{align*}
Therefore, we conclude that
\begin{align*}
    \sup_{s \in \cS(L, W, S, B)}|\ell(s, x) - \ell(s^*, x)| \lesssim \sigma_{\min}^{-4}\|x\|^2 + \sigma_{\min}^{-4}D(C_0^2 \vee C_1 \vee 1) .
\end{align*}
This and Assumption \ref{asn:relax_man} yield
\begin{align*}
    \left\|\sup_{s \in \cS(L, W, S, B)}|\ell(s, Y) - \ell(s^*, Y)|\right\|_{\psi_1}
    &\lesssim \sigma_{\min}^{-4}\left\|\|Y\|^2\right\|_{\psi_1} + \sigma_{\min}^{-4}D(C_0^2 \vee C_1 \vee 1) \\
    &\lesssim \sigma_{\min}^{-4}(1 + \sigma^2\left\| \|Z\|^2 \right\|_{\psi_1}) + \sigma_{\min}^{-4}D(C_0^2 \vee C_1 \vee 1) ,
\end{align*}
where $Z \sim \cN(0, I_D)$.
Since $\|Z\|^2 \sim \chi^2(D)$, we arrive at
\begin{align*}
    \left\|\sup_{s \in \cS(L, W, S, B)}|\ell(s, Y) - \ell(s^*, Y)|\right\|_{\psi_1}
    \lesssim \frac{D(C_0^2 \vee C_1 \vee 1)}{\sigma_{\min}^{4}}.
\end{align*}
The proof is finished.

\myendproof

\subsection{Proof of Lemma \ref{lem:loss_cov_num_bound_true}}\label{sec:cov_num_bound}

The proof is quite technical, so we divide it into several steps.

\noindent
\textbf{Step 1: evaluating the proximity of the loss functions.}\quad
Consider $s^{(1)}, s^{(2)} \in \cS(L, W, S, B)$ (see Definition \ref{def:score_class}) of the form
\begin{align*}
    s^{(j)}(x) = -a_j x + a_j f^{(j)}(x), \quad x \in \R^D, \; j \in \{1, 2\} .
\end{align*}
Now, for any arbitrary $x \in [-R, R]^D$, it holds that
\begin{align}
    \label{eq:l_1_l_2_x_diff}
    |\ell(s^{(1)}, x) - \ell(s^{(2)}, x)| \leq \frac{1}{2}\left\|s^{(1)}(x) - s^{(2)}(x)\right\| \cdot \left\|s^{(1)}(x) + s^{(2)}(x)\right\| + \left|\rdiv[s^{(1)} - s^{(2)}](x)\right| .
\end{align}
From Definition \ref{def:score_class} we find that
\begin{align*}
    \left\|s^{(1)}(x) + s^{(2)}(x)\right\|
    \leq \left\|s^{(1)}(x)\right\| + \left\|s^{(2)}(x)\right\|
    \leq 2\|x\|\sigma_{\min}^{-2} + 2\sqrt{D}C_0\sigma_{\min}^{-2}
    \leq 2\sqrt{D}(R + C_0)\sigma_{\min}^{-2} .
\end{align*}
In addition, we have that
\begin{align*}
    \left\|s^{(1)}(x) - s^{(2)}(x)\right\|
    &\leq |a_1 - a_2| \cdot \|x\| + |a_1 - a_2| \cdot \left\|f^{(1)}(x)\right\| + a_2 \cdot \left\|f^{(1)}(x) - f^{(2)}(x)\right\| \\
    &\leq |a_1 - a_2| \sqrt{D} (R + C_0) + \sigma_{\min}^{-2} \left\|f^{(1)}(x) - f^{(2)}(x)\right\| 
\end{align*}
and, similarly,
\begin{align*}
    \left|\rdiv[s^{(1)} - s^{(2)}](x)\right|
    &\leq |a_1 - a_2|D + |a_1 - a_2| \cdot \left|\rdiv[f^{(1)}](x)\right| + a_2\left|\rdiv[f^{(1)} - f^{(2)}](x)\right| \\
    &\leq |a_1 - a_2|D (1 + C_1\sigma_{\min}^{-2}) + \sigma_{\min}^{-2}\left|\rdiv[f^{(1)} - f^{(2)}](x)\right| .
\end{align*}
Substituting the derived bounds into \eqref{eq:l_1_l_2_x_diff}, we conclude that, for any $x \in [-R, R]^D$, the following holds:
\begin{align}
    \label{eq:l_s1_l_s2_prox}
    \notag
    &\left|\ell(s^{(1)}, x) - \ell(s^{(2)}, x)\right|
    \\&
    \lesssim D\sigma_{\min}^{-4}(R^2 \vee C_0^2 \vee C_1 \vee 1) \left(|a_1 - a_2| \vee \left\|f^{(1)}(x) - f^{(2)}(x)\right\| \vee \left|\rdiv[f^{(1)} - f^{(2)}](x)\right| \right).
\end{align}

\noindent
\textbf{Step 2: covering number evaluation.}\quad
We now formulate a result that quantifies the proximity of $f^{(1)}$ and $f^{(2)}$ together with its derivatives provided that their weights are close.
\begin{Lem}
\label{lem:gelu_func_div_prox}
    Let $\eps > 0$ and $R > 0$ be arbitrary, and let $f^{(1)}, f^{(2)} \in \NN(L, W, S, B)$ be of the form
    \begin{align*}
        f^{(j)}(x) = -b_L^{(j)} + A_L^{(j)} \circ \gelu_{b_{L - 1}^{(j)}} \circ A_{L - 1}^{(j)} \circ \gelu_{b_{L - 2}^{(j)}} \circ \dots \circ A_2^{(j)} \circ \gelu_{b_1^{(j)}} \circ A_1^{(j)} \circ x,
    \end{align*}
    where $x, f^{(j)} \in \R^D$, $j \in \{1, 2\}$, $\max_{1 \leq l \leq L} (\|A_l^{(1)} - A_l^{(2)}\|_\infty \vee \|b_l^{(1)} - b_l^{(2)}\|) \leq \eps$.
    Then, it holds that
    \begin{align*}
        (i) &\quad \|f_1 - f_2\|_{L^\infty([-R, R]^D)} \leq \sqrt{D} \cdot 4^L(B \vee 1)^L(\|W\|_\infty + 1)^L (R \vee 1) \eps , \\
        (ii) &\quad \|\rdiv[f_1 - f_2]\|_{L^\infty([-R, R]^D)} \leq D \cdot 16^L  (\|W\|_\infty + 1)^{2L - 1} (B \vee 1)^{2L - 1} (R \vee 1)\eps .
    \end{align*}
    Furthermore, for all $0 < \tau \leq \cD(\NN, \|\cdot\|_{L^\infty([-R, R]^D)})$, the following inequality holds:
    \begin{align*}
        (iii) &\quad \log \cN\left(\tau, \NN, \|\cdot\|_{L^\infty([-R, R]^D)}\right)
        \lesssim SL \log\left(\tau^{-1}L(B \vee 1)(\|W\|_\infty + 1)(R \vee 1)\right) .
    \end{align*}
\end{Lem}

We postpone the proof of Lemma \ref{lem:gelu_func_div_prox} to Appendix \ref{sec:lem_gelu_func_div_prox_proof}.
Define the class of neural network weights as
\begin{align*}
    &\cW(L, W, S, B) \\
    &= \left\{ \{(A_l, b_l)\}_{l = 1}^L : A_l \in \R^{W_l \times W_{l - 1}}, b_l \in \R^{W_l}, \max_{1 \leq l \leq L}(\|A_l\|_\infty \vee \|b\|_\infty) \leq B, \sum_{l = 1}^L(\|A_l\|_0 + \|b_l\|_0) \leq S \right\} .
\end{align*}
We also define $\cW_\eps$ as a minimal $\eps$-net of $\cW(L, W, S, B)$ with respect to $\|\cdot\|_\infty$-norm given by
\begin{align*}
    \|\{A_l, b_l\}_{l = 1}^L\|_\infty = \max_{1 \leq l \leq L}(\|A_l\|_\infty \vee \|b_l\|_\infty),
    \quad \{A_l, b_l\}_{l = 1}^L \in \cW(L, W, S, B) .
\end{align*}
Here, $0 < \eps \leq 2B$ will be determined later in the proof.
Let $\cA_\eps$ be a minimal $\eps$-net of $[1, \sigma_{\min}^{-2}]$ with respect to $\|\cdot\|_\infty$-norm.
We also define
\begin{align}
\label{eq:s_eps_net_def}
    \cS_\eps = \left\{s(x) = -ax + af(x) : a \in \cA_\eps, \, \text{$f$ of the form \eqref{eq:feed_forward_nn_def} with weights $\{(A_l, b_l)\}_{l = 1}^L \in \cW_\eps$} \right\} . 
\end{align}
By lemma \ref{lem:gelu_func_div_prox} and \eqref{eq:l_s1_l_s2_prox}, we have that for any $s \in \cS(L, W, S, B)$ there exists $s_\eps \in \cS_\eps$ such that
\begin{align*}
    \|\ell(s, \cdot) - \ell(s_\eps, \cdot)\|_{L^{\infty}([-R, R]^D)}
    \lesssim D^2 \sigma_{\min}^{-4}(R^2 \vee C_0^2 \vee C_1 \vee 1) \cdot 16^L(\|W_\infty\| + 1)^{2L}(B \vee 1)^{2L} (R \vee 1) \eps .
\end{align*}
Hence, choosing
\begin{align*}
    \log(1 / \eps)
    \asymp \log(1 / \tau) + L \log(D\sigma_{\min}^{-2}(R \vee C_0 \vee C_1 \vee 1)(B \vee 1)(\|W\|_\infty + 1)) 
\end{align*}
ensures that $\{\ell(s, \cdot) : s \in \cS_\eps\}$ is a $\tau$-net of $\cL$ with respect to $\|\cdot\|_{L^\infty([-R, R]^D)}$-norm.
In addition, we obtain that
\begin{align}
\label{eq:ism_log_diam_bound}
    \log \cD(\cL, \|\cdot\|_{L^{\infty}([-R, R]^D)})
    \lesssim L \log(D\sigma_{\min}^{-2}(R \vee C_0 \vee C_1 \vee 1)(B \vee 1)(\|W\|_\infty + 1)) .
\end{align}
Now it remains to evaluate $\log |\cS_\eps|$.
From \eqref{eq:s_eps_net_def} we find that
\begin{align}
    \label{eq:s_eps_net_prod}
    \log |\cS_\eps| \leq \log |\cA_\eps| + \log|\cW_\eps|. 
\end{align}
As for $|\cA_\eps|$, it is evident that
\begin{align}
    \label{eq:A_eps_net_bound}
    \log |\cA_\eps|
    \leq \log(\sigma_{\min}^{-2} / \eps)
    \lesssim \log(1 / \tau) + L \log(D\sigma_{\min}^{-2}(R \vee C_0 \vee C_1 \vee 1)(B \vee 1)(\|W\|_\infty + 1)) .
\end{align}
As for $|\cW_\eps|$, we first select the location of $S$ non-zero parameters given at most $L(\|W\|_\infty^2 + \|W\|_\infty)$ and, second, build an $\eps$-net of a cube $[-B, B]^S$.
Therefore, it holds that
\begin{align*}
    \log|\cW_\eps|
    &\leq \log\left( \binom{L(\|W\|_\infty^2 + \|W\|_\infty)}{S} \right) + S\log((B \vee 1) / \eps) \\
    &\lesssim SL \log \left(\tau^{-1} D \sigma_{\min}^{-2}(R \vee C_0 \vee C_1 \vee 1) L (B \vee 1)(\|W\|_\infty + 1) \right).
\end{align*}
Combining this, \eqref{eq:ism_log_diam_bound}, \eqref{eq:s_eps_net_prod},  and \eqref{eq:A_eps_net_bound} finishes the proof.

\myendproof

\subsection{Proof of Lemma \ref{lem:exp_max_log_data}}
\label{sec:lem_exp_max_log_data_proof}
    From \cite[Exercise 2.5.10]{vershynin2018high} we find that
    \begin{align}
        \label{eq:exp_max_log_aux}
        \E\max_{1 \leq i \leq n}\log^3(\|Y_i\| \vee 1)
        \lesssim \big\|\log^3(\|Y_1\| \vee 1)\big\|_{\psi_2} \sqrt{\log(2n)} .
    \end{align}
    Now \cite[Lemma 2.7.6]{vershynin2018high} suggests that
    \begin{align*}
        \big\| \|Y_1\| \vee 1 \big\|_{\psi_2}^2 = \big\|\|Y_1\|^2 \vee 1\big\|_{\psi_1}
        \lesssim 1 + \sigma^2\big\| \|Z\|^2 \big\|_{\psi_1},
    \end{align*}
    where $Z \sim \cN(0, I_D)$.
    Since $\|Z\|^2 \sim \chi^2(D)$, then it follows that
    \begin{align*}
       \big \| \|Y_1\| \vee 1 \big\|^2_{\psi_2} \lesssim 1 + \sigma^2 D .
    \end{align*}
    Therefore, using the definition of $\psi_1$-norm (see \eqref{eq:psi_1_norm_def}), we obtain that for any $\delta \in (0, 1)$, with probability at least $1 - \delta$, we have  $\|Y_1\|^2 \vee 1 \lesssim (1 + \sigma^2 D)\log(1 / \delta)$.
    Hence, on the same event, it holds that
    \begin{align*}
        \log^6(\|Y_1\|\vee 1) \lesssim \log^6(1 + \sigma^2 D) + \log^6\log(1 / \delta)
        \lesssim \log^6(1 + \sigma^2 D)(1 + \log(1 / \delta)) ,
    \end{align*}
    which implies that
    \begin{align*}
        \left\|\log^3(\|Y_1\| \vee 1)\right\|_{\psi_2}
        = \|\log^6(\|Y_1\| \vee 1)\|_{\psi_1}^{1/2}
        \lesssim \log^3(1 + \sigma^2 D).
    \end{align*}
    Therefore, substituting the derived bound into \eqref{eq:exp_max_log_aux} finishes the proof.

\myendproof

\subsection{Proof of Lemma \ref{lem:gelu_func_div_prox}}
\label{sec:lem_gelu_func_div_prox_proof}
To enhance clarity, we divide the proof into several steps.

\noindent
\textbf{Step 1: proving statement $(i)$.}\quad
First, for any $f \in \NN(L, W, S, B)$ of the form given in \eqref{eq:feed_forward_nn_def} and $l \in \N$ with $0 \leq l \leq L$, define
\begin{align*}
    \cF_l[f](x) = -b_l + A_l \circ \gelu_{b_{l - 1}} \circ A_{l - 1} \circ \gelu_{b_{l - 2}} \circ \dots \circ A_2 \circ \gelu_{b_1} \circ A_1 \circ x, \quad 1 \leq l \leq L,
\end{align*}
with $\cF_0[f](x) = x$.
Next, fix an arbitrary $x \in [-R, R]^D$, $1 \leq l \leq L$ and note that
$$\cF_l[f](x) = -b_l + A_l \circ \gelu \circ \cF_{l - 1}(f)(x).$$
Therefore, we obtain
\begin{align*}
    \left\|\cF_l[f](x)\right\|_\infty
    &\leq B + \|W\|_\infty B \cdot \left\|\gelu \circ \cF_{l - 1}(f)(x)\right\|_\infty .
\end{align*}
Next note from \eqref{eq:gelu_def} that $|\gelu(y)| \leq |y|$ for all $y \in \R$ and, thus,
\begin{align*}
    \|\cF_l[f](x)\|_\infty \leq B(\|W\|_\infty + 1)(1 \vee \| \cF_{l - 1}(f)(x) \|_\infty) .
\end{align*}
By unrolling the recursion and taking into account that $\|\cF_0(f)(x)\|_\infty \leq R$, we deduce that
\begin{align}
    \label{eq:gelu_bound_compact}
    \sup_{x \in [-R, R]^D}\|\cF_l[f](x)\|_\infty \leq R(B \vee 1)^l (\|W\|_\infty + 1)^l,
    \quad \text{for all $0 \leq l \leq L$} .
\end{align}
Now we are ready to elaborate on proximity of $f^{(1)}$ and $f^{(2)}$.
Formally, for an arbitrary $x \in [-R, R]^D$ and $1 \leq l \leq L$, we have that
\begin{align*}
    &\left\|\cF_l[f^{(1)}](x) - \cF_l[f^{(2)}](x)\right\|_\infty
    \\&
    \leq \eps + \|W\|_\infty B \left\|\cF_{l - 1}[f^{(1)}](x)\right\|_\infty \eps + \|W\|_\infty B \left\|\gelu \circ \cF_{l - 1}[f^{(1)}](x) - \gelu \circ \cF_{l - 1}[f^{(2)}](x)\right\|_\infty .
\end{align*}
Next, we establish that $\gelu$ and its derivative are $2$-Lipschitz continuous.
\begin{Lem}
\label{lem:gelu_and_der_2_lip}    
    For any $x, y \in \R$, it holds that
    \begin{align*}
        |\gelu(x) - \gelu(y)| \vee |\nabla \gelu(x) - \nabla\gelu(y)| \leq 2|x - y| .
    \end{align*}
    
\end{Lem}

The proof of Lemma \ref{lem:gelu_and_der_2_lip} is moved to Appendix \ref{sec:lem_gelu_and_der_2_lip_proof}.
Using \eqref{eq:gelu_bound_compact} and Lemma \ref{lem:gelu_and_der_2_lip}, we arrive at
\begin{align*}
    &\left\|\cF_l[f^{(1)}](x) - \cF_l[f^{(2)}](x)\right\|_\infty
    \\&
    \leq 2R(B \vee 1)^l(\|W\|_\infty + 1)^l \eps + 2\|W\|_\infty B \left\|\cF_{l - 1}[f^{(1)}](x) - \cF_{l - 1}[f^{(2)}](x)\right\|_\infty .
\end{align*}
Thus, by unrolling the recursion and using the fact that $\left\|\cF_0(f^{(1)})(x) - \cF_0(f^{(2)})(x)\right\|_\infty = 0$, we conclude that
\begin{align}
\label{eq:nn_l_lip_bound}
    \sup_{x \in [-R, R]^D} \left\|\cF_l[f^{(1)}](x) - \cF_l[f^{(2)}](x)\right\|_\infty
    \leq 4^l(B \vee 1)^l(\|W\|_\infty + 1)^l (R \vee 1) \eps,
    \quad \text{for all $0 \leq l \leq L$}.
\end{align}
Finally, noticing that $\|\cdot\| \leq \sqrt{D}\|\cdot\|_\infty$ and setting $l = L$ establishes claim $(i)$.

\noindent
\textbf{Step 2: proving statement $(ii)$.}\quad
As in the previous step, for any $1 \leq l \leq L$ and $f \in \NN(L, W, S, B)$ of the form given in \eqref{eq:feed_forward_nn_def}, we begin by evaluating
\begin{align*}
    \|\nabla \cF_l[f](x)\|_\infty
    = \|A_l \nabla\gelu(\cF_{l - 1}(f)(x))\nabla \cF_{l - 1}(f)(x) \|_\infty, \quad x \in \R^D .
\end{align*}
Since $\nabla\gelu(\cF_{l - 1}(f)(x))$ is a diagonal matrix with values from $[-2, 2]$, we deduce that
\begin{align*}
    \|\nabla \cF_l[f](x)\|_\infty
    \leq 2\|W\|_\infty B \|\nabla \cF_{l - 1}(f)(x)\|_\infty
\end{align*}
and, therefore, unrolling the recursion and noticing that $\|\nabla \cF_0(f)(x)\|_\infty = 1$ for all $x \in \R^D$ leads to
\begin{align}
\label{eq:jac_nn_l_bound}
    \sup_{x \in \R^D}\|\nabla \cF_l[f](x)\|_\infty \leq 2^l\|W\|_\infty^l B^l,
    \quad \text{for all $0 \leq l \leq L$}.
\end{align}
Next, for any $x \in [-R, R]^D$, we bound
\begin{align*}
    &\left\|\nabla \cF_l[f^{(1)}](x) - \nabla\cF_l[f^{(2)}](x)\right\|_\infty
    \\&
    = \left\|A_l^{(1)} \nabla \gelu(\cF_{l - 1}[f^{(1)}](x)) \nabla \cF_{l - 1}[f^{(1)}](x) - A_l^{(2)} \nabla \gelu(\cF_{l - 1}[f^{(2)}](x)) \nabla \cF_{l - 1}[f^{(2)}](x)\right\|_\infty .
\end{align*}
The triangle inequality implies that
\begin{align*}
    &\left\|\nabla \cF_l[f^{(1)}](x) - \nabla\cF_l[f^{(2)}](x)\right\|_\infty
    \\&
    \leq \left\|(A_l^{(1)} - A_l^{(2)}) \nabla \gelu(\cF_{l - 1}[f^{(1)}](x))\nabla \cF_{l - 1}[f^{(1)}](x)\right\|_\infty \\
    &\qquad + \left\| A_l^{(2)} ( \nabla\gelu(\cF_{l - 1}[f^{(1)}](x)) - \nabla\gelu( \cF_{l - 1}[f^{(2)}](x) ) ) \nabla\cF_{l - 1}[f^{(1)}](x) \right\|_\infty \\
    &\qquad + \left\| A_l^{(2)} \nabla\gelu(\cF_{l - 1}[f^{(2)}](x)) (\nabla\cF_{l - 1}[f^{(1)}](x) - \nabla\cF_{l - 1}[f^{(2)}](x) ) \right\|_\infty .
\end{align*}
Using the fact that, for each $j \in \{1, 2\}$, $\nabla\gelu(\cF_{l - 1}(f^{(j)})(x))$ is a diagonal matrix with entries from $[-2, 2]$, and the derivative of $\gelu$ is $2$-Lipschitz continuous due to Lemma \ref{lem:gelu_and_der_2_lip}, we deduce that
\begin{align*}
    \left\|\nabla \cF[f^{(1)}](x) - \nabla\cF[f^{(2)}](x)\right\|_\infty
    &\leq 2\|W\|_\infty B \left\|\nabla\cF_{l - 1}[f^{(1)}](x)\right\|_\infty \eps \\
    &\quad + 2\|W\|_\infty B \left\|\nabla \cF_{l - 1}[f^{(1)}](x)\right\|_\infty \cdot \left\|\cF_{l - 1}[f^{(1)}](x) - \cF_{l - 1}[f^{(2)}](x)\right\|_\infty \\
    &\quad + 2 \|W\|_\infty B \left\| \nabla\cF_{l - 1}[f^{(1)}](x) - \nabla\cF_{l - 1}[f^{(2)}](x) \right\|_\infty .
\end{align*}
In view of \eqref{eq:nn_l_lip_bound} and \eqref{eq:jac_nn_l_bound}, it follows that
\begin{align*}
    \left\|\nabla \cF_l[f^{(1)}](x) - \nabla\cF_l[f^{(2)}](x)\right\|_\infty
    &\leq (2\|W\|_\infty B)^l \eps
    + (2\|W\|_\infty B)^l \cdot 4^{l - 1}(B \vee 1)^{l - 1}(R \vee 1)\eps \\
    &\quad + 2B \|W\|_\infty \left\| \nabla\cF_{l - 1}[f^{(1)}](x) - \nabla \cF_{l - 1}[f^{(2)}](x) \right\|_\infty .
\end{align*}
By unrolling the recursion and noticing that $\left\|\nabla \cF_0(f^{(1)})(x) - \nabla\cF_0(f^{(2)})(x)\right\|_\infty = 0$, we conclude that, for all $0 \leq l \leq L$, it holds that
\begin{align*}
    \sup_{x \in [-R, R]^D} \left\|\nabla \cF_l[f^{(1)}](x) - \nabla\cF_l[f^{(2)}](x)\right\|_\infty
    \leq 16^l (\|W\|_\infty + 1)^{2l - 1} (B \vee 1)^{2l - 1}(R \vee 1)\eps .
\end{align*}
Therefore, statement $(ii)$ follows from the observation that
\begin{align*}
    \left|\rdiv\left[\cF_l[f^{(1)}] - \cF_l[f^{(2)}]\right](x) \right|
    &= \left|\tr(\nabla\cF_l[f^{(1)}](x) - \nabla\cF_l[f^{(2)}](x))\right|
    \\&
    \leq D \left\|\nabla\cF_l[f^{(1)}](x) - \nabla\cF_l[f^{(2)}](x)\right\|_\infty .
\end{align*}

\noindent
\textbf{Step 3: proving statement $(iii)$.}\quad
Note that there are at most
\begin{align*}
    \binom{L(\|W\|_\infty^2 + \|W\|_\infty)}{S} \leq 
    L^S (\|W_\infty\|^2 + \|W\|_\infty)^S
\end{align*}
ways to locate non-zero weights.
Therefore, statement $(i)$ of the current lemma together with the observation that $0 < \eps \leq 2B$ implies that
\begin{align*}
    &\log \cN\left(\tau, \NN(L, W, S, B), \|\cdot\|_{L^\infty([-R, R]^D)}\right)
    \\&
    \lesssim S \log(L (\|W\|_\infty + 1)) + SL \log(\tau^{-1}(\|W\|_\infty + 1)(B \vee 1)(R \vee 1))
    \\&
    \lesssim SL \log(\tau^{-1}L(B \vee 1)(\|W\|_\infty + 1)(R \vee 1)) ,
\end{align*}
for every $0 < \tau \leq \cD\left(\NN(L, W, S, B), \|\cdot\|_{L^\infty([-R, R]^D)}\right)$.
The proof is finished.

\myendproof

\subsection{Proof of Lemma \ref{lem:gelu_and_der_2_lip}}
\label{sec:lem_gelu_and_der_2_lip_proof}
Note that due to the mean value theorem it suffices to show that the absolute value of the first and the second derivative of $\gelu$ are uniformly bounded by $2$.
Using \eqref{eq:gelu_def}, we deduce that
\begin{align*}
    \partial^1 \gelu(x) = \frac{x}{\sqrt{2\pi}} e^{-x^2 / 2} + \Phi(x),
    \quad \partial^2\gelu(x) = \frac{1}{\sqrt{2\pi}}e^{-x^2 / 2}(2 - x^2), \quad x \in \R.
\end{align*}
Next, taking into account the fact that $t \, e^{-t}$ for any $t \geq 0$, we conclude that
\begin{align*}
    |\partial^2 \gelu(x)|
    \leq \sqrt{\frac{2}{\pi}}\left(1 + \sup_{t \geq 0}(t \, e^{-t})\right)
    \leq 2 , \quad \text{for all $x \in \R$} .
\end{align*}
As for the first derivative, it suffices to evaluate it at the points, where the second derivative is zero, that is $x = \pm \sqrt{2}$.
Therefore, we obtain
\begin{align*}
    |\partial^1 \gelu(x)| \leq \frac{e^{-1}}{\sqrt{\pi}} + 1 \leq 2,
    \quad \text{for all $x \in \R$.}
\end{align*}
The proof is complete.

\myendproof

\section{Proof of Theorem \ref{th:estimation_dsm}}
\label{sec:th_estimation_dsm_proof}

The proof structure shares similarities with the one of Theorem \ref{th:estimation}.
We are going to apply the tail inequality for unbounded empirical process outlined in Theorem \ref{th:tail_ineq_unb_new} and Remark \ref{rem:erm_localization_bound}.
Now ensure that the conditions of the above results are met.

\noindent
\textbf{Step 1: bounding the $\psi_1$-diameter of the excess loss class.}
\quad
First, define the excess loss class
\begin{align*}
    \cF = \big\{ \ell_t(s, \cdot) - \ell_t(s_t^*, \cdot) \, : \, s \in \cS_{DSM}(L, W, S, B) \big\} .
\end{align*}
The following lemma provides a bound on the $\psi_1$ diameter of $\cF$.
\begin{Lem}
\label{lem:psi1_norm_bound_dsm}
    Under Assumption \ref{asn:relax_man}, the following holds:
    \begin{align*}
        \left\|\sup_{s \in \cS_{DSM}(L, W, S, B)} |\ell_t(s, X_0) - \ell_t(s_t^*, X_0)| \right\|_{\psi_1}
        \lesssim \frac{D}{\sigma_t^2} + \frac{m_t^2 D(C_0^2 \vee 1)}{\sigma_t^4} .
    \end{align*}
\end{Lem}

The proof of Lemma \ref{lem:psi1_norm_bound_dsm} is deferred to Appendix \ref{sec:lem_psi1_norm_bound_dsm_proof}.

\noindent
\textbf{Step 2: verifying the Bernstein condition.}
\quad
The next lemma demonstrates that the Bernstein condition holds, a crucial property for achieving fast convergence rates.
Furthermore, it establishes that the score Jacobian matrix approximation error is controlled by the score matching error.

\begin{Lem}
\label{lem:bernstein_cond_dsm}
    For any $s \in \cS_{DSM}(L, W, S, B)$, $\alpha \in \N$ such that $\alpha \geq 2$, it holds that
    \begin{align*}
        (i) &\quad \Var[\ell_t(s, X_0) - \ell_t(s_t^*, X_0)]
        \lesssim \left\{ \frac{m_t^2 D (C_0^2 + \alpha)}{\sigma_t^4} + \frac{D \alpha}{\sigma_t^2} \right\}^{1 + 1 / \alpha} \|s - s_t^*\|_{L^2(\sfp_t^*)}^{2 - 2 / \alpha}, \\
        (ii) &\quad \big\| \|\nabla(s - s_t^*)\|_F \big\|_{L^2(\sfp_t^*)}^2 \lesssim \frac{\alpha^2 D(D + \alpha)(C_0^{2 / \alpha} \vee C_\alpha^{2 / \alpha} \vee 1)}{\sigma_t^{4 + 4 / \alpha} (m_t^2\sigma^2 + \sigma_t^2)} \|s - s_t^*\|_{L^2(\sfp_t^*)}^{2 - 2 / \alpha} .
    \end{align*}
    We emphasize that the hidden constant does not depend on $s$.
\end{Lem}

The proof of Lemma \ref{lem:bernstein_cond_dsm} is moved to Appendix \ref{sec:lem_bernstein_cond_dsm_proof}.

\noindent
\textbf{Step 3: covering number evaluation.}
\quad
Now we evaluate the covering number of the excess loss class $\cF$ with respect to the empirical $L^2$-norm.
This is shown in the following lemma.

\begin{Lem}
\label{lem:loss_cov_num_bound_dsm}
Define the loss function class as
\begin{align*}
        \cL = \left\{\ell_t(s, \cdot) \, : \, s\in\cS_{DSM}(L, W, S, B)\right\}
    \end{align*}
    Then, for every $R > 0$ and $0 < \tau \leq \cD(\cL, \|\cdot\|_{L^\infty([-R, R]^D)})$, it holds that
    \begin{align*}
        \log\cN(\tau, \cL, \|\cdot\|_{L^\infty([-R, R]^D)})
        \lesssim SL\log(\tau^{-1}L(B \vee 1)(\|W\|_\infty + 1)\sigma_t^{-2}D(C_0 \vee R \vee 1) ) .
    \end{align*}
\end{Lem}

We move the proof of Lemma \ref{lem:loss_cov_num_bound_dsm} to Appendix \ref{sec:lem_loss_cov_num_bound_dsm_proof}.
Note that for any $ 0 < \tau \leq \cD(\cF, L^2(\sfP_n))$, it holds that
\begin{align*}
    \log \cN(\tau, \cF, L^2(\sfP_n))
    \leq \log \cN(\tau, \cL, L^2(\sfP_n))
    \leq \log \cN(\tau, \cL, \|\cdot\|_{L^\infty([-R_n, R_n]^D)}),
\end{align*}
where $R_n = \max_{1 \leq i \leq n}\|Y_i\|$.
Therefore, by Lemma \ref{lem:loss_cov_num_bound_dsm}, we have that
\begin{align*}
    \log(\tau, \cF, L^2(\sfP_n))
    \lesssim SL \log(\tau^{-1}D_n),
\end{align*}
where we defined
\begin{align}
\label{eq:dsm_dn_def}
    D_n = L(B \vee 1) (\|W\|_\infty + 1)\sigma_t^{-2}D(C_0 \vee R_n \vee 1) .
\end{align}
Moreover, it holds that
\begin{align*}
    (\E \log^3 D_n)^{1/3}
    \lesssim \log(L(B \vee 1) (\|W\|_\infty + 1)\sigma_t^{-2}D(C_0 \vee 1))
    + \left(\E \max_{1 \leq i \leq n}\log^3(\|Y_i\| \vee 1) \right)^{1/3}
\end{align*}
Applying Lemma \ref{lem:exp_max_log_data}, we conclude that
\begin{align}
\label{eq:dsm_dn_bound}
    (\E \log^3 D_n)^{1/3}
    \lesssim \log(L(B \vee 1) (\|W\|_\infty + 1)\sigma_t^{-2}D(C_0 \vee 1)) \log(2n)
\end{align}

\noindent
\textbf{Step 4: applying the tail inequality for unbounded empirical processes.}
\quad
We now apply the tail inequality, as detailed in Theorem \ref{th:tail_ineq_unb_new} and Remark \ref{rem:erm_localization_bound}.
In fact, the conditions of Theorem \ref{th:tail_ineq_unb_new} are satisfied with $\varkappa = 1 - 1 / \alpha$, $A = \cO(1)$, $\zeta = \cO(SL)$, and $D_n$ given by \eqref{eq:dsm_dn_def}.
Therefore, applying Remark \ref{rem:erm_localization_bound} with $\eps = 1/2$, we conclude that
\begin{align*}
    \E[\ell_t(\widehat{s}, X_0) - \ell_t(s_t^*, X_0)]
    &
    \lesssim \inf_{s \in \cS(L, W, S, B)}\E[\ell_t(s, X_0) - \ell_t(s_t^*, X_0)]
    + (B_\alpha \Upsilon(n, \delta))^{\alpha / (\alpha + 1)}
    \\&\quad
    + (\Psi \vee 1)\Upsilon(n, \delta) \log n ,
\end{align*}
with probability at least $1 - \delta$.
Here,
\begin{align*}
    \Upsilon(n, \delta)
    &= \frac{1}{n}\left(\log A + \zeta \log n + \zeta(\E \log^3 D_n)^{1/3} + \log(e / \delta) \right)
\end{align*}
and
\begin{align*}
    \Psi = \left\|\sup_{s \in \cS_{DSM}(L, W, S, B)}|\ell_t(s, X_0) - \ell_t(s_t^*, X_0)| \right\|_{\psi_1} .
\end{align*}
In addition, $B_\alpha$ is given by Lemma \ref{lem:bernstein_cond_dsm}:
\begin{align}
\label{eq:B_a_dsm_def}
    B_\alpha = \left\{ \frac{m_t^2 D (C_0^2 + \alpha)}{\sigma_t^4} + \frac{D \alpha}{\sigma_t^2} \right\}^{1 + 1 / \alpha} .
\end{align}
From \eqref{eq:dsm_exp_loss_diff_fisher} we deduce that
\begin{align*}
    \E[\ell_t(s, X_0) - \ell_t(s_t^*, X_0)] = \|s - s_t^*\|^2_{L^2(\sfp_t^*)},
    \quad
    \text{for all } s \in \cS_{DSM}(L, W, S, B) .
\end{align*}
Therefore, from \eqref{eq:dsm_dn_bound} and Lemma \ref{lem:psi1_norm_bound_dsm} we conclude that
\begin{align}
\label{eq:dsm_gen_bound_almost_subst}
    \notag
    \|\widehat{s} - s_t^*\|^2_{L^2(\sfp_t^*)}
    &
    \lesssim \inf_{s \in \cS_{DSM}(L, W, S, B)}\|s - s_t^*\|^2_{L^2(\sfp_t^*)} + \left\{\frac{D\alpha}{\sigma_t^2} + \frac{m_t^2D(C_0^2 + \alpha)}{\sigma_t^4}\right\} 
    \\&\quad
    \cdot \frac{SL \log(L(B \vee 1) (\|W\|_\infty + 1)\sigma_t^{-2}D(C_0 \vee 1) ) \log^2(2n) \log(e / \delta)}{n^{\alpha / (\alpha + 1)}} .
\end{align}
with probability at least $1 - \delta$.

\noindent
\textbf{Step 5: deriving the final generalization bound.}\quad
Now, we will apply the approximation result outlined in Theorem \ref{thm:approx_main}.
Specifically, for the precision parameter $\eps \in (0, 1)$, which will be determined later in the proof, we configure the architecture $(L, W, S, B)$ as presented in Theorem \ref{thm:approx_main} for the value of $m = \alpha$.
First, we set in $\cS_{DSM}(L, W, S, B)$ (see Definition \ref{def:dsm_score_class})
\begin{align}
\label{eq:C_0_1_alpha_def_dsm}
    \quad C_j \asymp \exp\left\{\cO \left(j^2 \log(\alpha D) + j^2 \log\log\frac1{\eps} + j^2 \log \log \frac1{\sigma_t^{2}}  \right) \right\}, \quad j \in \{0, \alpha\} .
\end{align}
Since $\sfp_t^*$ satisfies Assumption \ref{asn:relax_man} with the generator $m_tg^*$ and the variance parameter $m_t^2\sigma^2 + \sigma_t^2$, we have that
\begin{align*}
    \inf_{s \in \cS_{DSM}(L, W, S, B)} \|s - s_t^*\|_{L^2(\sfp_t^*)}^2
    \lesssim \frac{D^2 \eps^{2\beta}}{ (m_t^2\sigma^2 + \sigma_t^2)^{4}} \log^2(1 / \eps) \log^2\left(\frac{\alpha D}{\sigma_t^{2}} \right) .
\end{align*}
In addition, Theorem \ref{thm:approx_main} implies that
\[
    SL\log(L(B\vee 1)(\|W\|_\infty + 1))
    \lesssim \frac{D^{24 + P(d, \beta)} \alpha^{217 + 17P(d, \beta)}}{ \eps^{d} \sigma_t^{48 + 4P(d, \beta)}} \left(\log(\alpha D \sigma_t^{-2})\log(1 / \eps)\right)^{65 + 4P(d, \beta)} .
\]
Therefore, we deduce from \eqref{eq:dsm_gen_bound_almost_subst} that if $\alpha \asymp \sqrt{\log n} + \sqrt{\log(\sigma_t^{-2})}$, then
\begin{align*}
    \|\hat{s} - s_t^*\|^2_{L^2(\sfp_t^*)}
    &
    \lesssim \frac{D^2\eps^{2\beta}\log^2(1 / \eps) \log^2(D \sigma_t^{-2} \log n)}{(m_t^2\sigma^2 + \sigma_t^2)^4}
    + \frac{D^{25 + P(d, \beta)}\alpha^{217 + 17P(d, \beta)}}{\sigma_t^{52 + 4P(d, \beta)} n \eps^{d}}
    \\&\quad
    \cdot \left( \log(D \sigma_t^{-2}n)\log(1 / \eps) \right)^{65 + 4P(d, \beta)} \log(e / \delta) 
    \exp\left\{\cO \left( \sqrt{\log n} + \sqrt{\log\frac1{\sigma_t^{2}}} \right) \right\} .
\end{align*}
with probability at least $1 - \delta$.
Thus, setting $\eps = (n\sigma_t^{52 + 4P(d, \beta)})^{-1 / (2\beta + d)} \in (0, 1)$ ensures that, with probability at least $1 - \delta$,
\begin{align}
\label{eq:dsm_gen_in_proof}
    \|\hat{s} - s_t^*\|^2_{L^2(\sfp_t^*)}
    \lesssim \frac{D^{25 + P(d, \beta)}}{(m_t^2\sigma^2 + \sigma_t^2)^{4}} \, \left( n\sigma_t^{52 + 4P(d, \beta)} \right)^{-\frac{2\beta}{2\beta + d}} \log(e / \delta) \sL(\sigma_t, D, n),
\end{align}
where we defined
\begin{align*}
    \sL(\sigma_t, D, n) = (\log(Dn\sigma_t^{-2}))^{239 + 17 P(d, \beta)} \exp\left\{\cO\left( \sqrt{\log n} + \sqrt{\log(\sigma_t^{-2})} \right) \right\} .
\end{align*}
Note that Theorem \ref{thm:approx_main} requires that the sample size be sufficiently large so that is satisfies
\begin{align*}
    n\sigma_{\min}^{52 + 4P(d, \beta)}
    \gtrsim 1 \vee \left\{ \frac{D \alpha^2 \log(n\alpha D \sigma_t^{-2})}{m_t^2 \sigma^2 + \sigma_t^2} \right\}^{\frac{2\beta + d}{\beta}}
    \vee \left\{ \frac{ D \alpha^2\log(n\alpha D\sigma_t^{-2}) }{m_t^2 \sigma^2 + \sigma_t^2}\right\}^{2\beta + d} .
\end{align*}
Taking into account the choice of $\alpha$, the bound simplifies to
\begin{align*}
    n\sigma_{\min}^{52 + 4P(d, \beta)}
    \gtrsim \left\{ D(m_t^2\sigma^2 + \sigma_t^2)^{-1} \log^2(nD\sigma_t^{-2}) \right\}^{\frac{2\beta + d}{\beta \wedge 1}} .
\end{align*}
Finally, using Theorem \ref{thm:approx_main}, we specify the parameters of the class $\cS_{DSM}(L, W, S, B)$ given in Definition \ref{def:dsm_score_class}:
\[
    L \lesssim \log(nD\sigma_t^{-2}),
    \quad \log B \lesssim D^8 \left( \log\frac{Dn}{\sigma_t^{2}} \right)^{90},
\]
and
\[
    \|W\|_\infty \vee S \lesssim \frac{D^{16 + P(d, \beta)}}{\sigma_t^{48 + 4P(d, \beta)}} \left( n\sigma_t^{52 + 4P(d, \beta)} \right)^{\frac{d}{2\beta + d}} \left( \log\frac{nD}{\sigma_t^{2}} \right)^{142 + 17 P(d, \beta)} .
\]
We also obtain from \eqref{eq:C_0_1_alpha_def_dsm} that
\begin{align}
\label{eq:c_0_c_alpha_dsm}
    C_0 \asymp 1,
    \quad C_\alpha \asymp \left(D\log(n) \log\frac1{\sigma_t^{2}} \right)^{\cO \left (\log(n\sigma_t^{-2}) \right)} .
\end{align}

\noindent
\textbf{Step 6: Jacobian matrix estimation.}\quad
From \eqref{eq:c_0_c_alpha_dsm} and Lemma \ref{lem:bernstein_cond_dsm} we deduce that
\begin{align*}
    \big\| \|\nabla(\hat{s} - s_t^*)\|_F \big\|^2_{L^2(\sfp_t^*)}
    \lesssim \frac{(D\log n \log(\sigma_t^{-2}))^{\cO(\sqrt{\log n} + \sqrt{\log(\sigma_t^{-2})})}}{\sigma_t^4(m_t^2\sigma^2 + \sigma_t^2)} \|\hat{s} - s_t^*\|^{2 - 2 / \alpha}_{L^2(\sfp_t^*)} 
\end{align*}
with unit probability.
Therefore, using \eqref{eq:dsm_gen_in_proof}, we conclude that
\begin{align*}
    \big\| \|\nabla(\hat{s} - s_t^*)\|_F \big\|^2_{L^2(\sfp_t^*)}
    \lesssim \frac{(D\log n \log(\sigma_t)^{-2})^{\cO(\sqrt{\log n} + \sqrt{\log(\sigma_t^{-2})})} \log(e / \delta)}{\sigma_t^4(m_t^2\sigma^2 + \sigma_t^2)^5} \left( n\sigma_t^{52 + 4P(d, \beta)} \right)^{-\frac{2\beta}{2\beta + d}}
\end{align*}
with probability at least $1 - \delta$.
The proof is finished.

\myendproof

\subsection{Proof of Lemma \ref{lem:bernstein_cond_dsm}}
\label{sec:lem_bernstein_cond_dsm_proof}

\noindent
\textbf{Proving statement $(i)$.}\quad
First, we note that
\begin{align*}
    &
    \Var[\ell_t(s, X_0) - \ell_t(s_t^*, X_0)]
    \\&
    \leq \E[(\ell_t(s, X_0) - \ell_t(s_t^*, X_0))^2] \\&
    = \E\left[\left(\E[(s(X_t) - s_t^*(X_t))^\top(s(X_t) + s_t^*(X_t) - 2\nabla_{X_t}\log\sfp_t^*(X_t | X_0)) \, \vert \, X_0]\right)^2\right] .
\end{align*}
Using Jensen's inequality in conjunction with Cauchy-Schwarz inequality, we have
\begin{align*}
    \Var[\ell_t(s, X_0) - \ell_t(s_t^*, X_0)]
    \leq \E \left[ \|s(X_t) - s_t^*(X_t)\|^2 \; \big\|s(X_t) + s_t^*(X_t) - 2\nabla_{X_t}\log\sfp_t^*(X_t | X_0)) \big\|^2 \right] .
\end{align*}
The Hölder inequality, when applied with parameters $p = \alpha / (\alpha - 1)$ and $q = \alpha$, yields
\begin{align}
\label{eq:var_lt_aux}
    \notag
    &\Var[\ell_t(s, X_0) - \ell_t(s_t^*, X_0)] \\
    &
    \leq \|s - s_t^*\|_{L^2(\sfp_t^*)}^{2 - 2 / \alpha} \left\{ \E[\|s(X_t) - s_t^*(X_t)\|^2 \|s(X_t) + s_t^*(X_t) - 2\nabla_{X_t}\log\sfp_t^*(X_t | X_0)\|^{2\alpha}] \right\}^{1 / \alpha} .
\end{align}
According to Definition \ref{def:dsm_score_class}, the element $s \in \cS_{DSM}(L, W, S, B)$ has the following form:
\begin{align*}
    s(x) = -\frac{x}{m_t^2\tilde{\sigma}^2 + \sigma_t^2} + \frac{m_t f(x)}{m_t^2\tilde{\sigma}^2 + \sigma_t^2},
    \quad x \in \R^D .
\end{align*}
Hence, it holds that
\begin{align}
\label{eq:bern_cond_dsm_s_diff}
    \notag
    \|s(X_t) - s_t^*(X_t)\|^2
    &\leq \frac{2\|X_t\|^2}{\sigma_t^4} + \frac{4m_t^2\|f(X_t)\|^2}{(m_t^2\tilde{\sigma} + \sigma_t^2)^2} + \frac{4m_t^2\|f^*(X_t)\|^2}{(m_t^2\tilde{\sigma} + \sigma_t^2)^2} \\
    &\leq \frac{4m_t^2\|X_0\|^2}{\sigma_t^4} + \frac{4\|X_t - m_t X_0\|^2}{\sigma_t^4} + \frac{4m_t^2 (DC_0^2 + 1)}{\sigma_t^4},
\end{align}
where the last inequality uses the fact that $|f_l(X_t)| \leq C_0$ for any $1 \leq l \leq D$ with unit probability.
Similarly, we obtain
\begin{align*}
    \left\|s(X_t) + s_t^*(X_t) - 2\nabla_{X_t}\log\sfp_t^*(X_t | X_0)\right\|^2
    &\leq \frac{8\|X_t\|^2}{\sigma_t^4} + \frac{4 m_t^2 \|f(X_t)\|^2}{\sigma_t^4} + \frac{4m_t^2\|f^*(X_t)\|^2}{\sigma_t^4} \\
    &\leq \frac{16m_t^2\|X_0\|^2}{\sigma_t^4} + \frac{16\|X_t - m_t X_0\|^2}{\sigma_t^4} + \frac{4m_t^2 (DC_0^2 + 1)}{\sigma_t^4} .
\end{align*}
Substituting the derived bounds into \eqref{eq:var_lt_aux}, we deduce that
\begin{align*}
    &\Var[\ell_t(s, X_0) - \ell_t(s_t^*, X_0)] \\
    &
    \leq 256 
    \left\{ \E\left[ \left( \frac{m_t^2\|X_0\|^2}{\sigma_t^4} + \frac{\|X_t - m_t X_0\|^2}{\sigma_t^4} + \frac{m_t^2 (DC_0^2 + 1)}{\sigma_t^4} \right)^{\alpha + 1}\right] \right\}^{1 / \alpha}
    \|s - s_t^*\|_{L^2(\sfp_t^*)}^{2 - 2 / \alpha} .
\end{align*}
From Minkowski's inequality we find that
\begin{align*}
    &\Var[\ell_t(s, X_0) - \ell_t(s_t^*, X_0)]
    \leq
    256 \, \sigma_t^{-4 - 4 / \alpha} \|s - s_t^*\|_{L^2(\sfp_t^*)}^{2 - 2 / \alpha}
    \\&
    \cdot \left\{ m_t^2 \left( \E \|X_0\|^{2\alpha + 2} \right)^{1 / (\alpha + 1)} + \left( \E \|X_t - m_t X_0\|^{2\alpha + 2} \right)^{1 / (\alpha + 1)} + m_t^2 (DC_0^2 + 1) \right\}^{1 + 1 / \alpha} .
\end{align*}
Next, recall from Assumption \ref{asn:relax_man} that $\|X_0\|^2 \leq 2 + 2\sigma^2\|Z\|^2$, where $Z \sim \cN(0, I_D)$.
Moreover, we have that $X_t - m_t x \sim \cN(0, \sigma_t^2 I_D)$ for any $x \in \R^D$.
Thus, using \cite[Proposition 2.7.1]{vershynin2018high} and taking into account that $\|\|Z\|^2\|_{\psi_1} \lesssim D$, we conclude that
\begin{align*}
    \Var[\ell_t(s, X_0) - \ell_t(s_t^*, X_0)]
    \lesssim \left\{ \frac{m_t^2 D (C_0^2 + \alpha)}{\sigma_t^4} + \frac{D \alpha}{\sigma_t^2} \right\}^{1 + 1 / \alpha} \|s - s_t^*\|_{L^2(\sfp_t^*)}^{2 - 2 / \alpha} .
\end{align*}
Furthermore, the hidden constant is absolute and, thus, it does not depend on $s$.

\noindent
\textbf{Proving statement $(ii)$.}\quad
Note that $\sfp_t^*$ satisfies Assumption \ref{asn:relax_man} with the generator function $m_t g^*$ and the variance parameter $m_t^2\sigma^2 + \sigma_t^2$ due to \eqref{eq:cond_law_ou}.
Therefore, by Lemma \ref{lem:bernstein_cond_div}, we have that
\begin{align}
\label{eq:jac_dsm_aux}
    \big\| \|\nabla(s - s_t^*)\|_F \big\|_{L^2(\sfp_t^*)}^2
    & \notag
    \leq \frac{\alpha D^{1 - 1 / \alpha}}{m_t^2\sigma^2 + \sigma_t^2} \; \|s - s_t^*\|_{L^2(\sfp_t^*)}^{2 - 2 / \alpha} 
    \\&\quad
    \cdot \left\{ D\|s - s_t^*\|^2_{L^2(\sfp_t^*)} + (m_t^2 \sigma^2 + \sigma_t^2)^\alpha |s - s_t^*|^2_{W^{\alpha, 2}(\sfp_t^*)} \right\}^{1 / \alpha} .
\end{align}
From \eqref{eq:bern_cond_dsm_s_diff} we deduce that
\begin{align}
\label{eq:s_s_star_l2_dsm_eval}
    \notag
    \E \|s(X_t) - s_t^*(X_t)\|^2
    &\leq \frac{4m_t^2 \E \|X_0\|^2}{\sigma_t^4} + \frac{4 \E \|X_t - m_t X_0\|^2}{\sigma_t^4} + \frac{4m_t^2(DC_0^2 + 1)}{\sigma_t^4} \\
    &\leq \frac{4m_t^2(2 + 2\sigma^2 D)}{\sigma_t^4} + \frac{4D}{\sigma_t^2} + \frac{4m_t^2(DC_0^2 + 1)}{\sigma_t^4} ,
\end{align}
where the last inequality stems from Assumption \ref{asn:relax_man} and the conditional law given in \eqref{eq:cond_law_ou}.
Now we evaluate the weighted Sobolev norm
\begin{align*}
    \left|s - s_t^*\right|_{W^{\alpha, 2}(\sfp_t^*)}^2
    = \sum_{\substack{\bk \in \Z_+^D \\ |\bk| = \alpha}} \left\| \frac{m_t \partial^\bk f}{m_t^2\tilde{\sigma}^2 + \sigma_t^2} - \partial^\bk s_t^* \right\|^2_{L^2(\sfp_t^*)}
    \lesssim \sum_{\substack{\bk \in \Z_+^D \\ |\bk| = \alpha}} \left( \frac{m_t^2}{\sigma_t^4}\|\partial^\bk f\|^2_{L^2(\sfp_t^*)} + \|\partial^\bk s_t^*\|^2_{L^2(\sfp_t^*)} \right) .
\end{align*}
Now applying Lemma \ref{lem:score_derivatives} for $\sfp_t^*$, which satisfies Assumption \ref{asn:relax_man} with the generator $m_t g^*$ and the noise variance $m_t\sigma^2 + \sigma_t^2$, and using the fact that $|f_l|_{W^{\alpha, 2}(\sfp_t^*)} \leq C_\alpha \sigma_t^{-2\alpha}$ (see Definition \ref{def:dsm_score_class}), we arrive at
\begin{align*}
    |s - s_t^*|^2_{W^{\alpha, 2}(\sfp_t^*)}
    &\lesssim \sum_{\substack{\bk \in \Z_+^D \\ |\bk| = \alpha}}\left( \frac{m_t^2}{\sigma_t^4}DC_\alpha^2\sigma_t^{-4\alpha} + 4^\alpha \alpha!\sigma_t^{-4\alpha - 4} m_t^2 \right) \\
    &\lesssim (D + \alpha)^\alpha(C_\alpha^2 \vee 1)\sigma_t^{-4\alpha - 4}m_t^2(D + 4^\alpha \alpha!) .
\end{align*}
Putting the derived bound, \eqref{eq:jac_dsm_aux}, and \eqref{eq:s_s_star_l2_dsm_eval} together and using Stirling's approximation leads to
\begin{align*}
    &
    \big\| \|\nabla(s - s_t^*)\|_F \big\|^2_{L^2(\sfp_t^*)} \\
    &
    \lesssim \frac{\alpha D(D + \alpha)}{m_t^2\sigma^2 + \sigma_t^2} \left\{ \left(\frac{1}{\sigma_t^2} + \frac{m_t^2 (C_0^2 \vee 1)}{\sigma_t^4}\right)^{1 / \alpha} + \frac{C_\alpha^{2 / \alpha} \vee 1}{\sigma_t^{4 + 4 / \alpha}}m_t^{2 / \alpha} \alpha \right\} \|s - s_t^*\|_{L^2(\sfp_t^*)}^{2 - 2 / \alpha} \\
    &
    \lesssim \frac{\alpha^2 D(D + \alpha)}{(m_t^2\sigma^2 + \sigma_t^2) \sigma_t^{4 + 4 / \alpha}} \; \left( C_0^{2 / \alpha} \vee C_\alpha^{2 / \alpha} \vee 1 \right) \, \|s - s_t^*\|_{L^2(\sfp_t^*)}^{2 - 2 / \alpha} .
\end{align*}
Thus, claim $(ii)$ holds true, and the proof is finished.

\myendproof

\subsection{Proof of Lemma \ref{lem:psi1_norm_bound_dsm}}
\label{sec:lem_psi1_norm_bound_dsm_proof}
For any $s \in \cS_{DSM}(L, W, S, B)$ of the form
\[
    s(y) = -\frac{y}{m_t^2\tilde{\sigma}^2 + \sigma_t^2} + \frac{m_t f(y)}{m_t^2\tilde{\sigma}^2 + \sigma_t^2},
    \quad y \in \R^D,
\]
it holds that
\begin{align*}
    \ell_t(s, X_0)
    &
    = \E \left[ \left\|s(X_t) - \nabla_{X_t}\log\sfp_t^*(X_t | X_0)\right\|^2 \, \big\vert \, X_0 \right]
    \\&
    \leq 2 \E \left[ \left\| -\frac{X_t}{m_t^2\tilde{\sigma}^2 + \sigma_t^2} + \frac{m_t f(X_t)}{m_t^2\tilde{\sigma}^2 + \sigma_t^2} \right\|^2 \, \bigg\vert \, X_0\right]
    + \frac{2}{\sigma_t^4} \, \E\left[ \|X_t - m_t X_0\|^2 \, \big\vert \, X_0 \right]
    \\&
    \leq \frac{4}{\sigma_t^4} \E\left[ \|X_t\|^2 \, \big\vert \, X_0 \right] + \frac{2m_t^2 C_0^2 D}{\sigma_t^4} + \frac{2}{\sigma_t^4} \, \E\left[ \|X_t - m_t X_0\|^2 \, \big\vert \, X_0 \right],
\end{align*}
where the last inequality uses the fact $|f_l(X_t)| \leq C_0$ for each $1 \leq l \leq D$ with unit probability due to Definition \ref{def:dsm_score_class}.
Using the conditional distribution given in \eqref{eq:cond_law_ou}, we arrive at
\begin{align*}
    \ell_t(s, X_0) \leq \frac{8m_t^2 \|X_0\|^2}{\sigma_t^4} + \frac{10 D}{\sigma_t^2} + \frac{2m_t^2 D C_0^2}{\sigma_t^4} .
\end{align*}
Similarly, for $s_t^*$ given in \eqref{eq:true_score_def} we have that
\begin{align*}
    \ell_t(s_t^*, X_0) \leq \frac{8m_t^2 \|X_0\|^2}{\sigma_t^4} + \frac{10 D}{\sigma_t^2} + \frac{2m_t^2}{\sigma_t^4} .
\end{align*}
Therefore, it follows that
\begin{align*}
    \left\|\sup_{s \in \cS_{DSM}(L, W, S, B)} |\ell_t(s, X_0) - \ell_t(s_t^*, X_0)| \right\|_{\psi_1}
    \lesssim \frac{m_t^2}{\sigma_t^4} \left\| \|X_0\|^2 \right\|_{\psi_1} + \frac{D}{\sigma_t^2} + \frac{m_t^2(D C_0^2 + 1)}{\sigma_t^4} .
\end{align*}
Using Assumption \ref{asn:relax_man} and \cite[Proposition 2.7.1]{vershynin2018high}, we conclude that
\begin{align*}
    \left\|\sup_{s \in \cS_{DSM}(L, W, S, B)} |\ell_t(s, X_0) - \ell_t(s_t^*, X_0)| \right\|_{\psi_1}
    \lesssim \frac{D}{\sigma_t^2} + \frac{m_t^2 D(C_0^2 \vee 1)}{\sigma_t^4} .
\end{align*}
The proof is complete.

\myendproof

\subsection{Proof of Lemma \ref{lem:loss_cov_num_bound_dsm}}
\label{sec:lem_loss_cov_num_bound_dsm_proof}

\noindent
\textbf{Step 1: evaluating the proximity of the loss functions.}\quad
Consider $s^{(1)}, s^{(2)} \in \cS_{DSM}(L, W, S, B)$ (see Definition \ref{def:dsm_score_class}) of the form
\begin{align*}
    s^{(j)}(x) = -\frac{x}{m_t^2 \tilde{\sigma}_j^2 + \sigma_t^2} + \frac{m_t}{m_t^2\tilde{\sigma}_j^2 + \sigma_t^2} f^{(j)}(x), \quad x \in \R^D, \; j \in \{1, 2\} .
\end{align*}
Also fix $\tilde{R} > 0$ that will be determined later in the proof.
Hence, for any arbitrary $x \in [-\tilde{R}, \tilde{R}]^D$, it holds that
\begin{align}
\label{eq:dsm_loss_diff_aux}
    &\notag
    \left|\ell_t(s^{(1)}, x) - \ell_t(s^{(2)}, x)\right|
    \\&\notag
    = \E \left[ \left| \|s^{(1)}(X_t) - \nabla_{X_t}\log \sfp_t^*(X_t | X_0)\|^2 - \|s^{(2)}(X_t) - \nabla_{X_t}\log \sfp_t^*(X_t | X_0)\|^2 \right| \, \Bigg\vert \, X_0 = x \right] \\&
    \leq \left( \E \left[ \|s^{(1)}(X_t) - s^{(2)}(X_t)\|^2 \, \Big\vert \, X_0 = x \right] \right)^{1 / 2}
    \\&\quad \notag
    \cdot \left( \E \left[ \| s^{(1)}(X_t) + s^{(2)}(X_t) - 2\nabla_{X_t}\log\sfp_t^*(X_t | X_0) \|^2 \, \Big\vert \, X_0 = x \right] \right)^{1 / 2} ,
\end{align}
where the last inequality uses the Cauchy-Schwarz inequality.
Now let us evaluate
\begin{align}
\label{eq:score_diff_dsm_cond_exp}
    &\notag
    \E[\|s^{(1)}(X_t) - s^{(2)}(X_t)\|^2 \, \vert \, X_0 = x]
    \\
    \notag
    &
    \leq 2 \left| \frac{1}{m_t^2\tilde{\sigma}_1^2 + \sigma_t^2} - \frac{1}{m_t^2\tilde{\sigma}_2^2 + \sigma_t^2} \right|^2 \E[\|X_t\|^2 \, \vert \, X_0 = x]
    \\&\quad
    + 2\E\left[\left\| \frac{m_t f^{(1)}(X_t)}{m_t^2\tilde\sigma_1^2 + \sigma_t^2} - \frac{m_t f^{(2)}(X_t)}{m_t^2\tilde\sigma_2^2 + \sigma_t^2} \right\|^2 \, \Bigg\vert \, X_0 = x\right]
    \\
    \notag
    &
    \leq \frac{2 m_t^2 |\tilde{\sigma}_1^2 - \tilde\sigma_2^2|^2}{\sigma_t^8} \; \E \left[ \|X_t\|^2 \, \big\vert \, X_0 = x \right]
    + 4C_0^2D \, \frac{m_t^2|\tilde{\sigma}_1^2 - \tilde\sigma_2^2|^2}{\sigma_t^8}
    \\&\quad\notag
    + \frac{4 m_t^2}{\sigma_t^4} \, \E \left[ \left\| f^{(1)}(X_t) - f^{(2)}(X_t) \right\|^2 \, \Big\vert \, X_0 = x \right] ,
\end{align}
where the last inequality uses Definition \ref{def:dsm_score_class}.
Next, using the union bound, we deduce that
\begin{align*}
    &\E \left[ \left\|f^{(1)}(X_t) - f^{(2)}(X_t) \right\|^2 \, \Big\vert \, X_0 = x \right] \\
    &
    \leq D\max_{1 \leq l \leq D} \left\| f^{(1)}_l - f^{(2)}_l \right\|^2_{W^{0, \infty}([-\tilde{R}, \tilde{R}]^D)} + 4C_0^2D \, \p\left( \|X_t\|_\infty \geq \tilde{R} \,\big|\, X_0 = x \right) \\
    &
    \leq D\max_{1 \leq l \leq D} \left\|f^{(1)}_l - f^{(2)}_l \right\|^2_{W^{0, \infty}([-\tilde{R}, \tilde{R}]^D)} + 4C_0^2 D^2 \, \p\left(|Z| \geq \frac{\tilde{R} - m_t\|x\|_\infty}{\sigma_t}\right),
\end{align*}
where $Z \sim \cN(0, 1)$.
Hence, we conclude that
\begin{align*}
    &
    \E \left[ \left\|f^{(1)}(X_t) - f^{(2)}(X_t) \right\|^2 \, \Big\vert \, X_0 = x \right] \\
    &
    \leq D \max_{1 \leq l \leq D} \left\|f^{(1)}_l - f^{(2)}_l \right\|^2_{W^{0, \infty}([-\tilde{R}, \tilde{R}]^D)} + 8C_0^2 D^2 \exp\left\{-\frac{\tilde{R}^2(1 - m_t)^2}{2\sigma_t^2}\right\} .
\end{align*}
Substituting this bound into \eqref{eq:score_diff_dsm_cond_exp} we find that
\begin{align}
\label{eq:cov_dsm_score_diff_cond}
    \E \left[ \left\|s^{(1)}(X_t) - s^{(2)}(X_t) \right\|^2 \, \Big\vert \, X_0 = x \right]
    & \notag
    \leq \frac{2m_t^2 D|\tilde{\sigma}_1^2 - \tilde{\sigma}_2^2| ( \sigma_t^2 + m_t^2 \tilde{R}^2 + 2C_0^2 )}{\sigma_t^8}
    \\
    &\quad
    + \frac{32 m_t^2D^2(C_0^2 \vee 1)}{\sigma_t^4} \max_{1 \leq l \leq D} \left\|f^{(1)}_l - f^{(2)}_l \right\|^2_{W^{0, \infty}([-\tilde{R}, \tilde{R}]^D)}
    \\&\quad \notag
    + \frac{32 m_t^2D^2(C_0^2 \vee 1)}{\sigma_t^4} \exp\left\{-\frac{\tilde{R}^2(1 - m_t)^2}{2\sigma_t^2}\right\} .
\end{align}
Similarly, using Definition \ref{def:dsm_score_class}  we obtain
\begin{align}
\label{eq:cov_dsm_mult_aux}
    &\notag
    \E \left[ \left\| s^{(1)}(X_t) + s^{(2)}(X_t) - 2\nabla_{X_t}\log\sfp_t^*(X_t | X_0) \right\|^2 \, \Big\vert \, X_0 = x \right] \\
    \notag
    &
    \leq 3 \E \left[ \left\|s^{(1)}(X_t) \right\|^2 \, \Big\vert \, X_0 = x \right]
    + 3 \E \left[ \left\|s^{(2)}(X_t) \right\|^2 \, \Big\vert \, X_0 = x \right]
    \\&\quad
    + 12 \E \left[ \left\|\nabla_{X_t} \log\sfp_t^*(X_t | X_0) \right\|^2 \, \Big\vert \, X_0 = x \right] 
    \\
    \notag
    &
    \leq \frac{12 \E \left[\|X_t\|^2 \, \big\vert \, X_0 = x \right] }{\sigma_t^4} + \frac{12 m_t^2 D C_0^2}{\sigma_t^4}
    + \frac{12 \E \left[ \|X_t - m_t x\|^2 \, \big\vert \, X_0 = x \right]}{\sigma_t^4} .
\end{align}
Thus, \eqref{eq:cond_law_ou} yields
\begin{align*}
    \E \left[ \left\| s^{(1)}(X_t) + s^{(2)}(X_t) - 2\nabla_{X_t}\log\sfp_t^*(X_t | X_0) \right\|^2 \, \Big\vert \, X_0 = x \right]
    \leq \frac{12 m_t^2 D(\tilde{R}^2 + C_0^2)}{\sigma_t^4} + \frac{24 D}{\sigma_t^2} .
\end{align*}
Putting together \eqref{eq:dsm_loss_diff_aux}, \eqref{eq:cov_dsm_score_diff_cond}, and \eqref{eq:cov_dsm_mult_aux} yields that for any $[-\tilde{R}, \tilde{R}]^D$, we have
\begin{align}
\label{eq:loss_prox_dsm_aux}
    &
    \left| \ell_t\left(s^{(1)}, x \right) - \ell_t\left(s^{(2)}, x \right) \right|
    \leq \left( \frac{768 m_t^2 D^3(C_0^2 \vee 1) (\tilde{R}^2 + C_0^2 + 1)}{\sigma_t^{12}}\right)^{1/2}
    \\
    &\notag
    \cdot \left( |\tilde{\sigma}_1^2 - \tilde{\sigma}_2^2|(1 + \tilde{R}^2 + 2C_0^2) +  \max_{1 \leq l \leq D} \left\|f^{(1)}_l - f^{(2)}_l \right\|^2_{W^{0, \infty}([-\tilde{R}, \tilde{R}]^D)} + \exp\left\{-\frac{\tilde{R}^2(1 - m_t)^2}{2\sigma_t^2}\right\} \right)^{1/2} .
\end{align}

\noindent
\textbf{Step 2: covering number evaluation.}\quad
We first evaluate the covering number of the GELU neural networks class.
Using claim $(iii)$ of Lemma \ref{lem:gelu_func_div_prox} we deduce that
\begin{align}
\label{eq:nn_cov_num_comp}
    \log \cN\left(\eps, \NN(L, W, S, B), \|\cdot\|_{L^\infty([-\tilde{R}, \tilde{R}]^D)}\right)
    \lesssim SL\log\left(\eps^{-1} L (B \vee 1)(\tilde{R} \vee 1)(\|W\|_\infty + 1)\right) ,
\end{align}
for any $0 < \eps \leq \cD\left(\NN(L, W, S, B), \|\cdot\|_{L^\infty([-\tilde{R}, \tilde{R}]^D)}\right)$.
Now let $\cF_\eps$ be a minimal $\eps$-net of $\NN(L, W, S, B)$ with respect to $\|\cdot\|_{L^\infty([-\tilde{R}, \tilde{R}]^D)}$-norm.
Let also $\cG_\eps$ be a minimal $\eps$-net of $[0, 1]$ with respect to $\|\cdot\|_\infty$-norm.
The value of $\eps$ will be determined later in the proof.
Define
\begin{align*}
    \cS_\eps = \left\{s(y) = -\frac{y}{m_t^2\tilde{\sigma}^2 + \sigma_t^2} + \frac{m_t}{m_t^2\tilde{\sigma}^2 + \sigma_t^2}f(y) \, : \, \tilde\sigma \in \cG_\eps, \, f \in \cF_\eps, \, y \in \R^D \right\} .
\end{align*}
Hence, we have that
\begin{align}
\label{eq:dsm_cov_product_basic}
    \log|\cS_\eps|
    = \log|\cF_\eps| + \log|\cG_\eps|
    \leq \log |\cF_\eps| + \log(1 / \eps)  .
\end{align}
Furthermore, \eqref{eq:loss_prox_dsm_aux} suggests that taking
\begin{align*}
    \tilde{R}
    = \frac{\sigma_t^2 \sqrt{2\log(1 / \eps)}}{(1 - m_t)^2} \vee R
    = \frac{(1 + e^{-t})^2 \sqrt{2\log(1 / \eps)}}{\sigma_t^2} \vee R
\end{align*}
and
\begin{align*}
    \log(1 / \eps) \asymp \log(1 / \tau) + \log(\sigma_t^{-2}D(C_0 \vee 1)) + \log(\tilde{R} \vee 1)
\end{align*}
ensures that
\begin{align*}
    \left\| \ell(s^{(1)}, \cdot) - \ell(s^{(2)}, \cdot) \right\|_{L^\infty([-R, R]^D)} \leq \tau .
\end{align*}
Therefore, $\{\ell(s, \cdot) : s \in \cS_\eps\}$ is the desired $\tau$-net of $\cL$ with respect to $\|\cdot\|_{L^{\infty}([-R, R]^D)}$-norm.
We also find that
\begin{align*}
    \log(1 / \eps) \lesssim \log(1 / \tau) + \log(\sigma_t^{-2}D(C_0 \vee R \vee 1)) .
\end{align*}
Hence, combining results from \eqref{eq:nn_cov_num_comp} and \eqref{eq:dsm_cov_product_basic}, we conclude that
\begin{align*}
    \log|\cS_\eps|
    \lesssim SL\log(\tau^{-1}L(B \vee 1)(\|W\|_\infty + 1)\sigma_t^{-2}D(C_0 \vee R \vee 1) ) ,
\end{align*}
thereby finishing the proof.

\myendproof

\section{Elements of learning theory}
\label{sec:stat_learn_th}

Let $\xi, \xi_1, \dots, \xi_n$ be i.i.d. random elements in $\R^D$ drawn from a distribution $\sfP$ and let $\cF = \left\{ f : \R^D \rightarrow \R \right\}$ be a class of Borel functions. Following the standard terminology of the empirical processes theory, we denote
\[
    \sfP f = \E f(\xi)
    \quad \text{and} \quad
    \sfP_n f = \frac1n \sum\limits_{i = 1}^n f(\xi_i)
    \quad \text{for any $f \in \cF$.}
\]
Given some $\eps > 0$, we are interested in high-probability upper bounds on the suprema of the empirical processes
\begin{equation}
    \label{eq:offset_processes}
    \sfP f - (1 + \eps) \sfP_n f
    \quad \text{and} \quad
    \sfP_n f - (1 + \eps) \sfP f,
    \quad \text{where $f \in \cF$.}
\end{equation}
When $\cF$ is bounded with respect to the $L^\infty$-norm and satisfies the Bernstein condition, sharp high-probability bounds on the suprema of the processes \eqref{eq:offset_processes} can be obtained via local \citep{bartlett2005local} or offset (see, for instance, \citep{liang15, puchkin2023exploring} Rademacher complexities. Let us recall that a sample Rademacher complexity is defined as
\[
    \cR_n(\cF) = \E_\sigma \sup\limits_{f \in \cF} \left| \frac1n \sum\limits_{i = 1}^n \sigma_i f(\xi_i) \right|,
\]
where $\E_\sigma$ stands for the  expectation with respect to $\sigma_1, \dots, \sigma_n$ (conditionally on $\xi_1, \dots, \xi_n$). Unfortunately, this does not suit the score estimation setup where we have to deal with unbounded empirical processes. For this reason, we have to slightly extend the localization technique. Throughout this section, we assume that $\cF$ has a finite $\psi_1$-diameter and its covering number $\cN\big(\eps, \cF, L^2(\sfP_n) \big)$ grows polynomially with respect to $(1 / \eps)$\footnote{Our technique naturally extends to Donsker and even nonparametric classes $\cF$, where the metric entropy $\log\cN\big(\eps, \cF, L^2(\sfP_n) \big)$ grows at a polynomial rate. However, for our purposes it will be enough to consider classes satisfying Assumption \ref{as:covering_number}.}. 

\begin{As}
    \label{as:covering_number}
    There exist $A \geq 1$, $\zeta \geq 0$, and a random variable $D_n$ such that the covering number of $\cF$ with respect to the empirical $L^2$-norm satisfies the inequality
    \[
        \cN\big(\eps, \cF, L^2(\sfP_n) \big) \leq A \left( \frac{D_n}{\eps} \right)^\zeta
        \quad \text{for all $0 < \eps \leq \cD\big(\cF, L_2(\sfP_n) \big)$ almost surely,}
    \]
    where $\cD\big(\cF, L^2(\sfP_n) \big)$ denotes the empirical $L^2$-diameter of $\cF$:
    \[
        \cD^2\big(\cF, L^2(\sfP_n) \big) = \sup\limits_{f, g \in \cF} \left\{ \frac1n \sum\limits_{i = 1}^n \big( f(\xi_i) - g(\xi_i) \big)^2 \right\}.
    \]
\end{As}
When dealing with unbounded empirical processes, the main obstacle is to relate $L^2(\sfP)$-radius $\rho = \sup_{f \in \cF} \sqrt{\sfP f^2}$ of $\cF$ with its empirical counterpart $\rho_n = \sup_{f \in \cF} \sqrt{\sfP_n f^2}$. In the bounded case, due to the contraction principle, one simply has
\[
    \E \rho_n^2 \leq \rho^2 + 2 \, \E \cR(\cF) \cdot \sup\limits_{f \in \cF} \|f\|_{L^\infty}.
\]
In the unbounded case, we slightly modify this approach and take into account sub-exponential tails of $f(\xi_1), \dots, f(\xi_n)$. As a result, we obtain the following bound on Rademacher complexity.

\begin{Lem}
    \label{lem:exp_local_rad}
    Grant Assumption \ref{as:covering_number}. Let us fix an arbitrary $q \in (1, 2)$ and suppose that
    \[
        \E (\log D_n)^{q / (2 - q)} < +\infty.
    \]
    Then it holds that
    \[
        \left( \E \cR_n^q(\cF) \right)^{1 / q}
        \leq \frac{2 \chi(q, \rho) \, \rho}{\sqrt n} + \frac{16 \chi^2(q, \rho) (2q - 1)(1 + \log_2 n)}{(q - 1) n} \left\| \sup\limits_{f \in \cF} |f(\xi)| \right\|_{\psi_1}.
    \]
    where $\rho = \sup_{f \in \cF} \sqrt{\sfP f^2}$ and
    \[
        \chi(q, \rho)
        = 12\sqrt{\frac{\pi \zeta}2} + 12\sqrt{2 \log A} + 12\sqrt{2 \zeta} \left( \E \left( \log \frac{D_n}\rho \right)^{q / (2 - q)} \right)^{(2 - q) / (2q)}.
    \]
\end{Lem}

We postpone the proof of Lemma \ref{lem:exp_local_rad} to Appendix \ref{sec:lem_exp_local_rad_proof} and proceed with a high-probability upper bound on the suprema of the offset processes \eqref{eq:offset_processes}.

\begin{Th}
\label{th:tail_ineq_unb_new}
    Let $\xi, \xi_1, \dots, \xi_n$ be i.i.d. random elements in $\R^D$. Let $\cF$ be a class of measurable functions $f : \R^D \to \R$ with a finite $\psi_1$-diameter and denote
    \[
        \Psi = \left\| \sup\limits_{f \in \cF} |f(\xi)| \right\|_{\psi_1}.
    \]
    Grant Assumption \ref{as:covering_number} and suppose that there exist $\varkappa \in (0, 1]$ and $B \geq 1$ such that $\sfP f^2 \leq B (\sfP f)^\varkappa$ for all $f \in \cF$.
    Then, for any $\delta \in (0, 1)$ and $\eps > 0$, with probability at least $1 - \delta$, simultaneously for all $f \in \cF$, we have
    \begin{align*}
        &
        \max\big\{ \sfP_n f - (1 + \eps) \sfP f, \sfP f - (1 + \eps) \sfP_n f \big\}
        \\&
        \lesssim \left( \frac{(1 + \eps)^2 B}{\eps^\varkappa} \Upsilon(n, \delta) \right)^{1 / (2 - \varkappa)} + (1 + \eps) (\Psi \vee 1) \Upsilon(n, \delta) \log n, 
    \end{align*}
    where
    \begin{equation}
        \label{eq:ups}
        \Upsilon(n, \delta) = \frac1n \left( \log A + \zeta \log n + \zeta \left( \E \log^3 D_n \right)^{1/3} + \log(e / \delta) \right),
    \end{equation}
    and $\lesssim$ stands for the inequality up to an absolute constant.
\end{Th}

The proof of Theorem \ref{th:tail_ineq_unb_new} mostly repeats the standard localization argument, except for the fact that we use Lemma \ref{lem:exp_local_rad} to bound local Rademacher complexities. We provide rigorous derivations in Appendix \ref{sec:th_tail_ineq_unb_new_proof}. Theorem \ref{th:tail_ineq_unb_new} yields the following upper bound on the generalization error of an empirical risk minimizer.

\begin{Rem}
    \label{rem:erm_localization_bound}
    Theorem \ref{th:tail_ineq_unb_new} yields an upper bound on the performance of an empirical risk minimizer. Indeed, assume the conditions of Theorem \ref{th:tail_ineq_unb_new} and let
    \[
        \hat f \in \argmin\limits_{f \in \cF} \sfP_n f,
        \quad
        f^\circ \in \argmin\limits_{f \in \cF} \sfP f.
    \]
    Then, for any $\eps > 0$ and $\delta \in (0, 1)$, with probability at least $(1 - \delta)$, it holds that
    \begin{align*}
        \sfP \hat f - (1 + \eps)^2 \sfP f^\circ
        &
        \leq \left( \sfP \hat f - (1 + \eps) \sfP_n \hat f \right) + (1 + \eps) \big( \sfP_n f^\circ - (1 + \eps) \sfP f^\circ \big)
        \\&
        \lesssim \left( \frac{(1 + \eps)^2 B}{\eps^\varkappa} \, \Upsilon(n, \delta) \right)^{1 / (2 - \varkappa)} + (1 + \eps) (\Psi \vee 1) \Upsilon(n, \delta) \log n,
    \end{align*}
    where $\Upsilon(n, \delta)$ is defined in \eqref{eq:ups} and $\lesssim$ stands for the inequality up to an absolute constant.
\end{Rem}

\subsection{Proof of Lemma \ref{lem:exp_local_rad}}
\label{sec:lem_exp_local_rad_proof}

Let us introduce an empirical counterpart of $\rho$:
\[
    \rho_n = \sup\limits_{f \in \cF} \sqrt{\sfP_n f^2}.
\]
Using the standard chaining technique \cite[Lemma A.3]{srebro2010smoothness}, we obtain that
\[
    \cR_n(\cF)
    \leq \frac{12}{\sqrt n} \int\limits_0^{\rho_n} \sqrt{\log \cN\big( \eps, \cF, L^2(\sfP_n) \big)} \, \dd \eps
    \leq \frac{12}{\sqrt n} \int\limits_0^{\rho_n \vee \rho} \sqrt{\log A + \zeta \log \frac{D_n}\eps } \, \dd \eps.
\]
Making a substitution $\eps = (\rho_n \vee \rho) e^{-u}$, $u \in (0, +\infty)$, we deduce that
\begin{align*}
    \cR_n(\cF)
    &
    \leq \frac{12 (\rho_n \vee \rho)}{\sqrt n} \int\limits_0^{+\infty} \sqrt{\log A + \zeta \log \frac{B_n}{\rho_n \vee \rho} + \zeta u} \; e^{-u} \dd u
    \\&
    \leq \frac{12 (\rho_n \vee \rho)}{\sqrt n} \int\limits_0^{+\infty} \left( \sqrt{\log A + \zeta \log \frac{B_n}{\rho_n \vee \rho}} + \sqrt{\zeta u} \right) \, e^{-u} \dd u
    \\&
    \leq \frac{12 (\rho_n \vee \rho)}{\sqrt n} \left( \sqrt{\log A} + \sqrt{\zeta \log \frac{B_n}{\rho}} +  \frac{\sqrt{\pi \zeta}}2 \right).
\end{align*}
In the last line, we used the fact that
\[
    \int\limits_0^{+\infty} \sqrt{u} \, e^{-u} \dd u
    = \Gamma\left( \frac32 \right)
    = \frac{\sqrt{\pi}}2.
\]
Due to the triangle inequality and the symmetrization trick, it holds that
\[
    \E \rho_n^2
    \leq \rho^2 + \E \sup\limits_{f \in \cF} \left( \sfP_n f^2 -\sfP f^2 \right)
    \leq \rho^2 + 2 \E_\xi \E_\sigma \sup\limits_{f \in \cF}\left| \frac1n \sum\limits_{i = 1}^n \sigma_i f^2(\xi_i) \right|,
\]
where $\sigma_1, \dots, \sigma_n$ are i.i.d. Rademacher random variables, which are independent of $\xi_1, \dots, \xi_n$.
Let us note that for each $R > 0$ the map $x \mapsto x^2$ is $2R$-Lipschitz on $[-R, R]$. In view of the Talagrand contraction principle \cite[Theorem 4.12]{ledoux2013probability}, we have
\[
    \E_\sigma \sup\limits_{f \in \cF} \left|\frac1n \sum\limits_{i = 1}^n \sigma_i f^2(\xi_i)\right|
    \leq 4 \max\limits_{1 \leq i \leq n} \sup\limits_{f \in \cF} |f(\xi_i)| \, \cR_n(\cF).
\]
Thus, applying H\"older's inequality, we obtain that
\[
    \E \rho_n^2
    \leq \rho^2 + 8 \left( \E \max\limits_{1 \leq i \leq n} \sup\limits_{f \in \cF} |f(\xi_i)|^{q / (q - 1)} \right)^{(q - 1) / q} \left( \E \cR_n^q(\cF) \right)^{1 / q}.
\]
Let us note that, due to Lemma \ref{lem:max_psi_norm} and \cite[Lemma F.7]{puchkin25}, it holds that
\begin{align*}
    \left( \E \max\limits_{1 \leq i \leq n} \sup\limits_{f \in \cF} |f(\xi_i)|^{q / (q - 1)} \right)^{(q - 1) / q}
    &
    \leq \frac{2^{(q - 1) / q}  (2q - 1)}{q - 1}  \left\| \max\limits_{1 \leq i \leq n} \sup\limits_{f \in \cF} |f(\xi_i)| \right\|_{\psi_1}
    \\&
    \leq \frac{2 \Psi (2q - 1)(1 + \log_2 n)}{q - 1},
\end{align*}
where we introduced
\[
    \Psi = \left\| \sup\limits_{f \in \cF} |f(\xi)| \right\|_{\psi_1}.
\]
Substituting this bound into the previous inequality, we obtain that
\[
    \E \rho_n^2
    \leq \rho^2 + \frac{16 \Psi (2q - 1)(1 + \log_2 n)}{q - 1} \left( \E \cR_n^q(\cF) \right)^{1 / q},
\]
On the other hand, according to H\"older's inequality, we have
\begin{align*}
    \left( \E \cR_n^q(\cF) \right)^{1 / q}
    &
    \leq \frac{12}{\sqrt n} \left( \E (\rho_n^q \vee \rho^q) \left( \frac{\sqrt{\pi \zeta}}2 + \sqrt{\log A + \zeta \log \frac{D_n}\rho} \right)^{q} \right)^{1 / q}
    \\&
    \leq \frac{12}{\sqrt n} \sqrt{ \E (\rho_n^2 \vee \rho^2) } \left( \E \left( \frac{\sqrt{\pi \zeta}}2 + \sqrt{\log A + \zeta \log \frac{D_n}\rho} \right)^{2q / (2 - q)} \right)^{(2 - q) / (2q)}.
\end{align*}
This yields that
\begin{align*}
    \left( \E \cR_n^q(\cF) \right)^{1 / q}
    &
    \leq \frac{12 \sqrt{2}}{\sqrt n} \left( \rho^2 + \frac{16 \Psi (2q - 1)(1 + \log_2 n)}{q - 1} \left( \E \cR_n^q(\cF) \right)^{1 / q} \right)^{1 / 2}
    \\&\quad
    \cdot \left[ \frac{\sqrt{\pi \zeta}}2 + \sqrt{\log A} + \sqrt{\zeta} \left( \E \left( \log \frac{D_n}\rho \right)^{q / (2 - q)} \right)^{(2 - q) / (2q)} \right]
    \\&
    \leq \frac{\chi(q, \rho)}{\sqrt n} \left( \rho + \sqrt{\frac{16 \Psi (2q - 1)(1 + \log_2 n)}{q - 1} \left( \E \cR_n^q(\cF) \right)^{1 / q}} \right),
\end{align*}
where
\[
    \chi(q, \rho)
    = 12\sqrt{\frac{\pi \zeta}2} + 12\sqrt{2 \log A} + 12\sqrt{2 \zeta} \left( \E \left( \log \frac{D_n}\rho \right)^{q / (2 - q)} \right)^{(2 - q) / (2q)}.
\]
Hence, we showed that
\begin{align*}
    \left( \E \cR_n^q(\cF) \right)^{1 / q}
    \leq \frac{\chi(q, \rho) \, \rho}{\sqrt n} + 4 \chi(q, \rho) \sqrt{\frac{\Psi (2q - 1)(1 + \log_2 n)}{(q - 1) n} \left( \E \cR_n^q(\cF) \right)^{1 / q}}.
\end{align*}
It only remains to note that the inequality $x \leq a + b\sqrt{x}$ yields $x \leq b^2 + 2a$. This allows us to conclude that
\[
    \left( \E \cR_n^q(\cF) \right)^{1 / q}
    \leq \frac{2 \chi(q, \rho) \, \rho}{\sqrt n} + \frac{16 \chi^2(q, \rho) (2q - 1)(1 + \log_2 n) \, \Psi}{(q - 1) n}.
\]

\subsection{Proof of Theorem \ref{th:tail_ineq_unb_new}}
\label{sec:th_tail_ineq_unb_new_proof}

Similarly to the standard localization argument (see, for instance, \citep{bartlett2005local}), the proof of Theorem \ref{th:tail_ineq_unb_new} uses reweighting and peeling techniques. For any $f \in \cF$ and $r > 0$ we define
\[
    \sfk(f) = \min\left\{ k \in \Z_+ : \sfP f \leq 4^k r \right\}
    \quad \text{and} \quad
    \cF_r = \left\{ f \in \cF : \sfP f \leq r \right\} .
\]
Then the triangle inequality yields that
\begin{align*}
    \E \sup\limits_{f \in \cF} \left\{ 4^{-\sfk(f)} \left| \sfP f - \sfP_n f \right| \right\}
    &
    \leq \E \sup\limits_{f \in \cF_r} \left| \sfP f - \sfP_n f \right|
    + \sum\limits_{k = 1}^\infty 4^{-k} \; \E \sup\limits_{f \in \cF_{4^k r} \backslash \cF_{4^{k - 1} r}} \left| \sfP f - \sfP_n f \right|
    \\&
    \leq \sum\limits_{k = 0}^\infty 4^{-k} \; \E \sup\limits_{f \in \cF_{4^k r}} \left| \sfP f - \sfP_n f \right|.
\end{align*}
According to the standard symmetrization argument the summands in the right-hand side are bounded by the corresponding double Rademacher complexities:
\begin{align}
\label{eq:exp_scaled_sup_emp_proc}
    \E \sup\limits_{f \in \cF} \left\{ 4^{-\sfk(f)} \left| \sfP f - \sfP_n f \right| \right\}
    \leq \sum\limits_{k = 0}^\infty 4^{-k} \; \E \sup\limits_{f \in \cF_{4^k r}} \left| \sfP f - \sfP_n f \right|
    \leq 2 \sum\limits_{k = 0}^\infty 4^{-k} \; \E \cR_n \left( \cF_{4^k r} \right).
\end{align}
Applying Lemma \ref{lem:exp_local_rad} with $q = 3/2$, we deduce that
\begin{align}
\label{eq:exp_Rn_local}
    \notag
    \E \cR_n \left( \cF_{4^k r} \right)
    \leq \left( \E \cR_n^{3/2} \left( \cF_{4^k r} \right) \right)^{2/3}
    &
    \lesssim \frac{\rho_k}{\sqrt n} \left( \sqrt{\log A} + \sqrt{\zeta} + \sqrt{\zeta \log(1 / \rho_k)} + \left( \E \log^3 D_n \right)^{1/6} \right)
    \\&\quad
    +\frac{\Psi \log n}n \left( \log A + \zeta + \zeta \log(1 / \rho_k) + \left( \E \log^3 D_n \right)^{1/3} \right),
\end{align}
where $\rho_k = \sup_{f \in \cF_{4^k r}} \sqrt{\sfP f^2}$.
Next, using the Bernstein condition, that is, $\sfP f^2 \leq B(\sfP f)^\varkappa$ for every $f \in \cF$, we obtain that
\begin{align}
\label{eq:rho_k_bound}
    \rho_k \leq \sup\limits_{f \in \cF_{4^k r}} \sqrt{B} \left( \sfP f \right)^{\varkappa / 2}
    \leq 2^{\varkappa k} r^{\varkappa / 2} \sqrt{B}.
\end{align}
Introducing
\[
    \Phi_n(r, \delta) = \frac1n \left(\log A + \zeta + \zeta \log(1 / r) + \zeta \left( \E \log^3 D_n \right)^{1/3} + \log(1 / \delta) \right)
\]
and summing up the inequalities\eqref{eq:exp_scaled_sup_emp_proc}, \eqref{eq:exp_Rn_local}, and \eqref{eq:rho_k_bound}, we conclude that
\begin{align*}
    \E \sup\limits_{f \in \cF} \left\{ 4^{-\sfk(f)} \left| \sfP f - \sfP_n f \right| \right\}
    &
    \lesssim \sum\limits_{k = 0}^\infty \left( \frac{2^{\varkappa k}}{4^k} \sqrt{B r^\varkappa \Phi_n(r, 1)} + 4^{-k} \Psi \Phi_n(r, 1) \log n \right)
    \\&
    \lesssim \sqrt{B r^\varkappa \Phi_n(r, 1)} + \Psi \Phi_n(r, 1) \log n
\end{align*}
Note that for all $f \in \cF$ it holds that $\Var[4^{-\sfk(f)}f] \leq Br^\varkappa$.
Indeed, due to the conditions of the theorem, we have
\[
    \Var\left[4^{-\sfk(f)} f \right] \leq 8^{-\sfk(f)}\Var[f] \leq 8^{-\sfk(f)}B (\sfP f)^\varkappa
    \leq 8^{-\sfk(f)} B(4^{\sfk(f)} r)^\varkappa \leq B r^\varkappa.
\]
According to the concentration inequality for suprema of unbounded empirical processes \citep{adamczak2008tail}, there exists a universal constant $C > 0$ and an event $\cE$ of probability measure at least $(1 - \delta)$ such that
\[
    \sup\limits_{f \in \cF} \left\{ 4^{\sfk(f)} \left| \sfP f - \sfP_n f \right| \right\}
    \leq C \sqrt{B r^\varkappa \Phi_n(r, \delta)} + C \Psi \Phi_n(r, \delta) \log n
    \quad \text{on $\cE$.}
\]
From now on, we restrict our attention on the event $\cE$. Let us fix an arbitrary $f \in \cF$. There are two scenarios: either $\sfk(f) = 0$ or $\sfk(f) > 0$.
If $\sfk(f) = 0$, then
\begin{align}
    \label{eq:k_f_zero_conc}
    \left| \sfP f - \sfP_n f \right| \leq C \sqrt{B r^\varkappa \Phi_n(r, \delta)} + C \Psi \Phi_n(r, \delta) \log n \quad \text{on $\cE$.}
\end{align}
Otherwise, it holds that $4^{-\sfk(f)} \sfP f \geq r / 4$ and, therefore, the following two inequalities hold on $\cE$:
\begin{align}
\label{eq:pf_1peps_pn}
    \notag
    \sfP f - (1 + \eps) \sfP_n f
    &
    \leq -\eps \, \sfP f + 4^{\sfk(f)} (1 + \eps) C \left( \sqrt{B r^\varkappa \Phi_n(r, \delta)} + \Psi \Phi_n(r, \delta) \log n \right)
    \\&
    \leq 4^{\sfk(f)}  \left( -\frac{\eps r}4 + (1 + \eps) C \sqrt{B r^\varkappa \Phi_n(r, \delta)} + (1 + \eps) C \Psi \Phi_n(r, \delta) \log n \right)
\end{align}
and
\begin{align}
\label{eq:pn_1peps_p}
    \notag
    \sfP_n f - (1 + \eps) \sfP f
    &
    \leq -\eps \, \sfP f + 4^{\sfk(f)} C \left( \sqrt{B r^\varkappa \Phi_n(r, \delta)} + \Psi \Phi_n(r, \delta) \log n \right)
    \\&
    \leq 4^{\sfk(f)}  \left( -\frac{\eps r}4 + C \sqrt{B r^\varkappa \Phi_n(r, \delta)} + C \Psi \Phi_n(r, \delta) \log n \right).
\end{align}
Let us choose the smallest $r > 0$ satisfying the condition
\[
    \frac{\eps r}4 \geq (1 + \eps) C \sqrt{B r^\varkappa \Phi_n(r, \delta)} + (1 + \eps) C \Psi \Phi_n(r, \delta) \log n.
\]
In view of \citep[Lemma E.9]{puchkin25}, such $r$ fulfills
\[
    r \lesssim \left( \frac{(1 + \eps)^2 B C^2}{\eps^2} \Upsilon(n, \delta) \right)^{1 / (2 - \varkappa)} \vee \frac{(1 + \eps) C (\Psi \vee 1) \Upsilon(n, \delta) \log n}\eps,
\]
where we introduced
\[
    \Upsilon(n, \delta) = \frac1n \left( \log A + \zeta \log n + \zeta \left( \E \log^3 D_n \right)^{1/3} + \log(1 / \delta) \right).
\]
Hence, due to the inequalities \eqref{eq:k_f_zero_conc}, \eqref{eq:pf_1peps_pn}, and \eqref{eq:pn_1peps_p}, on the event $\cE$ any function $f \in \cF$ satisfies either
\[
    \left| \sfP f - \sfP_n f \right|
    \leq C \sqrt{B r^\varkappa \Phi_n(r, \delta)} + C \Psi \Phi_n(r, \delta) \log n
    \leq \frac{\eps r}{4 (1 + \eps)}
\]
or
\[
    \max\left\{ \sfP_n f - (1 + \eps) \sfP f, \sfP f - (1 + \eps) \sfP_n f \right\} \leq 0.
\]
This immediately implies that with probability at least $(1 - \delta)$ simultaneously for all $f \in \cF$, it holds that
\begin{align*}
    &
    \max\left\{ \sfP_n f - (1 + \eps) \sfP f, \sfP f - (1 + \eps) \sfP_n f \right\}
    \lesssim \eps r
    \\&
    \lesssim \left( \frac{(1 + \eps)^2 B}{\eps^\varkappa} \, \Upsilon(n, \delta) \right)^{1 / (2 - \varkappa)} + (1 + \eps) (\Psi \vee 1) \Upsilon(n, \delta) \log n.
\end{align*}
\myendproof

\section{Auxiliary results}

\begin{Th}[approximation of the true score function (\cite{yakovlev2025simultaneous}, Theorem 3.2)]
\label{thm:approx_main}
    Grant Assumption \ref{asn:relax_man}.
    Also assume that $\eps \in (0, 1)$ is sufficiently small in the sense that it satisfies
    \begin{align*}
        \eps^\beta
        \leq \frac{\floor{\beta}!}{H d^\floor{\beta} \sqrt{D}}\left(1 \wedge \frac{C_1 \sigma^2}{\sqrt{D}m^2(\log(1 / \eps) + \log(mD\sigma^{-2}))} \right)
    \end{align*}
    and
    \begin{align*}
        (H \vee 1)^2 P(d, \beta)^2 D m^2(\log(1 / \eps) + \log(mD\sigma^{-2}))\eps \leq C_2 \sigma^2 ,
    \end{align*}
    where $C_1$ and $C_2$ are absolute positive constants.
    Then for any $m \in \N$ there exists a score function approximation $\bar{s} \in \cS(L, W, S, B)$ of the form
    \begin{align*}
        \bar{s}(y) = -\frac{y}{\sigma^2} + \frac{\bar{f}(y)}{\sigma^2}, \quad y \in \R^D,
    \end{align*}
    which satisfies
    \begin{align*}
        (i) &\quad
        \max_{1 \leq l \leq D}\max_{\bk \in \Z_+^D, \, |\bk| \leq m} \|\partial^\bk[\bar{s}_l - s^*_l] \|^2_{L^2(\sfp^*)}
        \lesssim \sigma^{-4|\bk| - 8} e^{\cO(|\bk|\log |\bk|)} D^2\eps^{2\beta} \log^2(1 / \eps) \log^2(mD\sigma^{-2}), \\
        (ii) &\quad \max_{1 \leq l \leq D} |\bar{f}_l|_{W^{k, \infty}(\R^D)} \leq \sigma^{-2k}\exp\{\cO( k^2\log(mD\log(1 / \eps)\log(\sigma^{-2})) )\}, \quad \text{for all } 0 \leq k \leq m .
    \end{align*}
    Moreover, $\bar{f}$ has the following configuration:
    \begin{align*}
        &L \lesssim \log(mD\sigma^{-2}\log(1 / \eps)),
        \quad \log B \lesssim m^{85}D^8 \log^{26}(mD\sigma^{-2})\log^{21}(1 / \eps), \\
        &\|W\|_\infty \vee S
        \lesssim \eps^{-d} D^{16 + P(d, \beta)} m^{132 + 17 P(d, \beta)}\sigma^{-48 - 4P(d, \beta)} \left( \log(mD\sigma^{-2}) \log(1 / \eps) \right)^{38 + 4 P(d, \beta)},
    \end{align*}
    where $P(d, \beta) = \binom{d + \floor{\beta}}{d}$.
\end{Th}

\begin{Lem}[\cite{yakovlev2025generalization}, Lemma 4.1]
\label{lem:score_derivatives}
    Grant Assumption \ref{asn:relax_man}.
    Then, for all $k \in \N$, it holds that
    \begin{align*}
        \left\|\nabla^k\left( \log \sfp^*(y) - \frac{\|y\|^2}{2\sigma^2} \right)\right\| \leq \frac{2^{k - 1}(k - 1)!}{\sigma^{2k}} \max_{u \in [0, 1]^d}\|g^*(u)\|^k ,
    \end{align*}
    where $\|\cdot\|$ denotes the operator norm of a tensor $\cT$ of order $k$, defined as
    \begin{align*}
        \|\cT\| = \sup_{\|u_1\| = \|u_2\| = \dots = \|u_k\| = 1}\left\{\sum_{i_1, \dots, i_k} \cT_{i_1, \dots, i_k}u_{1, i_1}u_{2, i_2} \cdot\cdot\cdot u_{k, i_k}  \right\} .
    \end{align*}
    For a sufficiently smooth function $f : \R^D \to \R$, we define its $k$-th order derivative tensor $\nabla^k f$ componentwise as
    \begin{align*}
        (\nabla^k f(y))_{i_1, \dots, i_k} = \partial^{(i_1, \dots, i_k)}f(y),
        \quad
        (i_1, \dots, i_k) \in \{1, \dots, D\}^k .
    \end{align*}
\end{Lem}

\begin{Lem}
\label{lem:max_psi_norm}
    Let $\xi_1, \dots, \xi_n$ be identically distributed random variables with $\|\xi_1\|_{\psi_1} < \infty$ and $n \in \N$.
    Then it holds that
    \begin{align*}
        \left\| \max_{1 \leq i \leq n} |\xi_i| \right\|_{\psi_1} \leq \|\xi_1\|_{\psi_1} \log_2(2n) .
    \end{align*}
    
\end{Lem}

\begin{proof}
The H\"older inequality implies that
\begin{align*}
    \E \exp\left\{\frac{\max_{1 \leq i \leq n}|\xi_i|}{\|\xi_1\|_{\psi_1} \log_2(2n)}\right\}
    &
    \leq \left(\E \exp\left\{ \max\limits_{1 \leq i \leq n}|\xi_i| \big\slash \|\xi_1\|_{\psi_1} \right\} \right)^{1 / \log_2(2n)}
    \\&
    \leq \left( \sum\limits_{1 \leq i \leq n} \E \exp\left\{ |\xi_i| \big\slash \|\xi_1\|_{\psi_1} \right\} \right)^{1 / \log_2(2n)}
    \\&
    \leq \left( n \; \E \exp\left\{ |\xi_1| \big\slash \|\xi_1\|_{\psi_1} \right\} \right)^{1 / \log_2(2n)} .
\end{align*}
By definition of the $\psi_1$-norm, we have that
\begin{align*}
    \E \exp\left\{ \max\limits_{1 \leq i \leq n} |\xi_i| \big\slash \big( \|\xi_1\|_{\psi_1} \log_2(2n) \big) \right\}
    \leq (2n)^{1 / \log_2(2n)} = 2.
\end{align*}
Therefore, the claim follows.

\end{proof}

\newpage
\tableofcontents

\end{document}